\documentclass[a4paper,11pt,oneside]{amsart}

\usepackage[a4paper,top=3cm,bottom=3cm,left=2.5cm,right=2.5cm]{geometry}

\usepackage[utf8]{inputenc}
\usepackage[english]{babel}
\usepackage{todonotes}
\usepackage{csquotes}

\usepackage[hyperfootnotes=false]{hyperref} 
\usepackage{color}
\hypersetup{
	colorlinks = true,
	linkcolor = {blue},
	urlcolor = {red},
	citecolor = {blue}
}

\usepackage{amsmath}
\usepackage{amsthm}
\usepackage{amssymb}
\usepackage{mathrsfs} 
\usepackage{eucal}
\usepackage{enumerate}
\usepackage{mathtools}
\usepackage{comment}

\newcounter{results}[section]

\theoremstyle{plain}
\newtheorem{theorem}[results]{Theorem}
\newtheorem{metatheorem}[results]{Meta-Theorem}
\newtheorem{lemma}[results]{Lemma}
\newtheorem{proposition}[results]{Proposition}
\newtheorem{corollary}[results]{Corollary}
\newtheorem{assumption}[results]{Assumption}

\theoremstyle{remark}
\newtheorem{remark}[results]{Remark}
\newtheorem{example}[results]{Example}

\theoremstyle{definition}
\newtheorem{definition}[results]{Definition}

\numberwithin{equation}{section}

\usepackage{scalerel,stackengine}
\stackMath
\newcommand\reallywidehat[1]{%
\savestack{\tmpbox}{\stretchto{%
  \scaleto{%
    \scalerel*[\widthof{\ensuremath{#1}}]{\kern-.6pt\bigwedge\kern-.6pt}%
    {\rule[-\textheight/2]{1ex}{\textheight}}
  }{\textheight}%
}{0.5ex}}%
\stackon[1pt]{#1}{\tmpbox}%
}

\newcommand\addtxtr[1]{\textcolor{black}{#1}}

\newcommand{\R}{\ensuremath{\mathbb R}} 
\newcommand{\C}{\ensuremath{\mathbb C}} 
\newcommand{\T}{\ensuremath{\mathbb T}} 
\newcommand{\Z}{\ensuremath{\mathbb Z}} 
\newcommand{\N}{\ensuremath{\mathbb N}} 
\newcommand{\de}{\ensuremath{\,\mathrm d}}
\newcommand{\eps}{\ensuremath{\varepsilon}} 
\newcommand{\scalprod}[2]{\ensuremath{\langle #1, #2\rangle}} 
\newcommand{\prob}{\ensuremath{\mathcal{P}}} 
\newcommand{\diff}{\ensuremath{\mathrm D}}
\newcommand{\ii}{\ensuremath{\mathrm i}} 
\newcommand{\rmC}{\mathrm C} 
\DeclarePairedDelimiter{\norm}{\lVert}{\rVert} 
\renewcommand{\div}{\ensuremath{\mathrm{div}}}
\newcommand{\lf}{\ensuremath{\delta_\mu}} 

\newcommand{\bbI}{\ensuremath{\mathbb I}} %
\newcommand{\cL}{\ensuremath{\mathcal L}} %
\newcommand{\cA}{\ensuremath{\mathcal A}} %
\newcommand{\cF}{\ensuremath{\mathcal F}} %
\newcommand{\cG}{\ensuremath{\mathcal G}} %
\newcommand{\cB}{\ensuremath{\mathcal B}} %
\makeatletter
\@namedef{subjclassname@2020}{\textup{2020} Mathematics Subject Classification}
\makeatother

\title[Fourier Galerkin approximation of MFC problems]{Fourier Galerkin approximation of mean field control problems}
\author{François Delarue}
\address{F. Delarue: Université C\^ote d’Azur, CNRS, Laboratoire J. A. Dieudonné , 06108 Nice, France}
\email{francois.delarue@univ-cotedazur.fr}
\author{Mattia Martini}
\address{M. Martini: Université C\^ote d’Azur, CNRS, Laboratoire J. A. Dieudonné , 06108 Nice, France}
\email{mattia.martini@univ-cotedazur.fr}

\thanks{\newline The authors acknowledge the financial support of the European Research Council (ERC) under the European Union’s Horizon 2020 research and innovation programme (AdG ELISA project, Grant agreement No. 101054746).}

\begin{document}
\begin{abstract}
The purpose of this work is to provide a finite dimensional approximation of the solution to 
a mean field optimal control problem set on the $d$-dimensional torus. 
The approximation is obtained by means of a Fourier-Galerkin 
method, the main principle of which is to
convolve probability measures on the torus by the Dirichlet kernel or, equivalently,  
to truncate the Fourier expansion of probability measures on the torus. 
However, this 
operation has the main feature not to leave the space of probability measures invariant, 
which drawback is know as \textit{Gibbs}' phenomenon. 
In spite of this, 
we manage to prove that, for initial conditions in the `interior' of the space
of probability measures
and for sufficiently large levels of truncation, 
the Fourier-Galerkin method induces 
a new finite dimensional control problem
whose trajectories take values in 
the space of probability measures with 
 a finite number of Fourier 
coefficients. Our main result asserts that, whenever the 
cost functionals are smooth and convex, the distance 
between the optimal trajectories of the 
original and approximating control problems 
decreases at a polynomial rate as the index of truncation in 
the Fourier-Galerkin method tends to $\infty$. A similar result 
holds for the distance between the corresponding value functions. 
From a practical point of view, our approach provides an efficient 
strategy to 
approximate 
 mean field control
 optimizers by finite dimensional 
 parameters and opens new perspectives 
 for the numerical analysis of mean field control problems. It may 
 be also applied to discretize more general mean field game systems. 
\end{abstract}

\maketitle

\section{Introduction}

\subsection{Presentation of the Contribution}

\subsubsection*{Background.}
Since its inception twenty years ago by Lasry and Lions \cite{lasrylions1,lasrylions2,lasrylions3,lions} and 
Caines, Huang and Malhamé \cite{HuangCainesMalhame1,Huang2006,Huang2007} in connection with the 
companion theory 
of mean field games
(see also the more recent 
monographs or reviews \cite{cardaliaguetporretta-cetraro,carmonadelarue1,carmonadelarue2}), mean field control theory has become a popular approach 
to the analysis of large cooperative games with weak interaction. 
In its standard form, a (stochastic) mean field control problem 
is a control problem set over controlled trajectories 
with values in the space of probability measures ${\mathcal P}({\mathbb R}^d)$ over ${\mathbb R}^d$, for some $d \geq 1$. 
A typical instance of such trajectories are 
solutions to second-order Fokker-Planck equations with a controlled 
velocity field. From a modelling point of view,  
those trajectories describe the statistical marginal laws of \textit{continua} of
(dynamical) agents 
obeying a central planner and 
arise as the mean field limits of empirical measures 
computed over large (but finite) clouds
of agents under the supervision \addtxtr{of a common} central planner. 
\textit{Per se}, mean field control problems are thus infinite dimensional 
control problems on 
${\mathcal P}({\mathbb R}^d)$. 
As such, they are  
related 
with the 
calculus of variations on ${\mathcal P}({\mathbb R}^d)$.

Quite naturally,
the
peculiar geometry of ${\mathcal P}({\mathbb R}^d)$ 
plays a key role in the analysis of mean field control problems. 
Intuitively, one must indeed understand what an elementary step over 
${\mathcal P}({\mathbb R}^d)$ is. 
Briefly, there are two main available approaches in the literature:
the first one is to regard ${\mathcal P}({\mathbb R}^d)$ 
as a convex subset of the (linear) space of signed 
measures and thus to consider, from a functional analysis viewpoint, probability measures
as Schwarz distributions; another way is to write probability measures
as statistical distributions of random variables defined 
on an exogenous probability space. 
Interestingly, this leads to different notions of `straight lines', depending on whether 
segments are obtained by interpolating two probability measures or 
two random variables. This dichotomy 
is present in all the literature on mean field control problems, 
in which one usually distinguishes between the Partial Differential Equation (PDE)
viewpoint 
(see for instance \cite{ambrosiogiglisavare,BrianiCardaliaguet,gan-swi,lasrylions2} for a tiny 
examples)
and the probabilistic one
(see \cite{lacker,Lacker2020SuperpositionAM}). While the 
PDE approach directly addresses the 
macroscopic behaviour of the \textit{continuum} of agents (through, say, the aforementioned 
Fokker-Planck equation 
set in a distributional sense), the probabilistic 
one relies on a (stochastic) differential equation 
(typically of McKean-Vlasov type) accounting for the evolution of one typical player in the \textit{continuum}. 

\subsubsection*{Objectives.}
Whatever the approach that is adopted, the main questions under study in the literature are the following ones: 
$(i)$
existence and (possibly) uniqueness of an optimal trajectory; $(ii)$ 
formulation of a convenient form of dynamical programming principle; $(iii)$ 
analysis of a corresponding Pontryagin principle (which is very much connected 
to the theory of mean field games); $(iv)$ analysis of the related Hamilton-Jacobi equation
(which is set on ${\mathcal P}({\mathbb R}^d))$; $(v)$ connection between the mean field model and the original 
controlled problem set over a finite particle system.  
In addition to all these points, another wide problem is to provide numerical approximations 
(or even statistical approximations when dealing with observations) to the optimal trajectories. 
In the current contribution, the question we address is 
mostly of a theoretical essence but is expected to have some numerical applications: 
\textit{how to provide an efficient parametric approximation of mean field control problems?}
It is worth emphasizing that the word `parametric' here refers to the notion of parameter in statistics, which 
is usually understood as being finite dimensional. 

Obviously, the
very first parametric approximation one may think of is the original 
particle system itself.  Indeed, the state variable 
in the original control problem set over a finite population of 
size $N$
 is the vector comprising the current states of all the individuals in the population, which is a vector of dimension $d \times N$.
 However, the resulting accuracy is typically caped by $N^{-1}$, which is the bound obtained in 
 \cite{cardaliaguetdelaruelasrylions} for the distance between the 
 value function of the mean field control problem 
 and the value function of the original one when the former is smooth (see 
 \cite{cardaliaguet2023algebraic,daudindelaruejackson} for refinements when the value function is not smooth, in which case the rate is worse). 
 In this approach, the approximation is constructed on the random variables themselves, 
 consistently with the second of the two main approaches 
 we reported above for mean field control. 
 In the present work, our philosophy is different: we directly want to expand the solution 
 to the Fokker-Planck equation (driving the controlled dynamics of the population) 
 along a `basis' of a convenient functional (or Schwarz distributional) space
 and then to retain the first $N$ coefficients of it as  a parametric approximation 
 of the controlled flow of probability measures. 
 Very briefly, we call this a Galerkin approximation. 
 The subsequent question concerns the choice of the basis itself. 
 Similarly with many of the theoretical works that have been released in the field,  
 we here restrict ourselves to the torus ${\mathbb T}^d:=({\mathbb R}/{\mathbb Z})^d$. Also, since the Fokker-Planck equations we are dealing with are driven by a 
 standard Laplacian, it seems very appropriate to work with the Fourier basis, as a 
 consequence of which our approach should be seen as a Fourier-Galerkin 
method.    
 Very nicely, the truncation operation, consisting in retaining \addtxtr{only the first $N$}  Fourier modes in the expansion of a probability measure, can be interpreted as a convolution operation by means of the so-called Dirichlet kernel $D^N$ of degree $N$. 
 The first step in our program is thus to introduce an approximating control problem in which the state variable (namely, 
 the statistical state of the population) is merely convoluted by $D^N$. 
 Our analysis then focuses on the well-posedness of the approximating control problem and on 
 the distance between the resulting optimal trajectories (of the  
approximating and mean field problems) and on the difference between the resulting value functions.  
Among others, this provides a finite-dimensional approximation of the Hamilton-Jacobi equation associated with 
the mean field control problem, which is a parabolic partial differential equation originally set on ${\mathcal P}({\mathbb T}^d)$
(see for instance \cite{cardaliaguetdelaruelasrylions,chasscrisdel,gan-swi}). 

Although this approach looks quite natural, it raises in fact some difficulties. Some of them have already been reported in the earlier article \cite{cecchindelarue} in which the Fourier expansion is used to make sense of the 
PDE
satisfied, on the space 
${\mathcal P}({\mathbb R}^d)$, 
 by the optimal feedback function
 to the mean field control problem. 
 The main obstacle is that the Dirichlet kernel $D^N$ is not a density, which implies in particular that
 the convolution  $\mu * D^N$ between an arbitrary probability 
$\mu \in {\mathcal P}({\mathbb T}^d)$ and $D^N$ may not be a probability measure. Even worse, 
$\mu * D^N$ may take negative values, which
is known as \textit{Gibbs}' phenomenon and which 
 has dramatic consequences in the control problem 
if one substitutes naively $\mu * D^N$ for $\mu$ when 
$\mu$ is understood as the generic controlled state of the population. 
Indeed, it is a known fact in mean field control theory that one cannot easily extend 
cost functionals (as they are usually defined) to a wider space of signed 
measures, because this basically amounts to reversing all the signs 
in the cost functional, which would have dramatic consequences (very crudely, think of replacing $\inf(f)$ by 
$\inf(-f)$ in a minimization problem). 
In  \cite{cecchindelarue}, the strategy 
to circumvent this difficulty 
is to use $S^N$ instead of $D^N$, 
where $S^N$ is the Féjer kernel of index $N$. The very good point is that 
$S^N$ is a density, as a consequence of which $\mu * S^N$ is always a probability measure. 
For sure, this makes much easier any attempt to substitute $\mu * S^N$
for $\mu$ in the control problem. However, the resulting rate of approximation is rather disappointing 
(see for instance the monograph \cite{Zhizhiashvili} for a complete review on this): intuitively, the 
error between $f* S^N$ and $f$, for an arbitrary function $f$, cannot be improved by choosing 
$f$ more and more regular, which may be easily noticed by 
choosing $d=1$ and $f(x) = \cos( 2 \pi x)$, in which case $f*S_N(x) -f(x)=1/N \cos(2 \pi x)$.

\subsubsection*{Results.}
The fact that the Féjer kernel features a rather disappointing rate of convergence explains why we come back to the Dirichlet kernel and use 
$\mu * D_N$ as an approximation of $\mu$. 
We show that, 
in order to force $\mu * D_N$ to be a probability measure, 
one may choose $N$ large enough
and then 
work with initial conditions (to the mean field control problem) having a strictly positive density, the collection of such 
initial conditions being referred below to as the `interior' of ${\mathcal P}({\mathbb R}^d)$. Our first result in this direction is to show that the 
resulting optimal control problem (which is in the end the approximating problem we are dealing with in the paper) 
has a unique optimal trajectory. Furthermore, the latter can be characterized by means of a forward-backward system 
very \addtxtr{close to a} mean field game system. Our second and main result provides an upper bound for the rate 
of convergence of the value functions and of the optimal trajectories when the population starts from 
a strictly positive smooth density. In short, the error between the value functions is (at most) of order 
$1/N^{q-1}$ when the coefficients and the initial condition have $q$ (say bounded) derivatives
and the error (in $L^2({\mathbb T}^d)$ norm) between the optimal trajectories 
is (at most) of order $1/N^{q-1-d/2}$. Throughout, the coefficients are also assumed to be convex in the measure argument and in the control parameter, see Subsection \ref{subse:assumptions} for a complete description of the required conditions. 
Of course, the thrust of our result is to prove that, whenever the data are sufficiently smooth with respect to the dimension $d$ (and convexity is in force), 
the mean field control problem can be nicely approximated by means of a finite family of parameters 
\addtxtr{whose cardinal is} independent of $d$. 
This result should be seen as a proof of concept exemplifying the possibility to overcome the curse of dimensionality in 
mean field control problems with smooth data. 

Interestingly, it is clear that 
all the tools that are used here could be readapted in the framework of mean field games with smooth and monotone coefficients. 
To wit, our result applies directly to the approximation of potential mean field games. 
For (almost) free, 
the same result can be obtained for more general running and terminal costs coefficients. 
That said, we feel that the interpretation of the 
Fourier-Galerkin 
approximation is more striking
when dealing with control (instead of games) as 
it can be seen as a new control problem but in finite dimension, hence our choice to focus on this problem. 
We elaborate on the extension to mean field games in Subsection \ref{subsubse:MFG}.

\subsubsection*{Prospects.}
Our paper opens several directions of research. A very first technical question concerns the possible extension of our Galerkin approach to 
\addtxtr{non-periodic domains. 
For instance, mean field control problems can be set on the entire Euclidean space 
${\mathbb R}^d$, see \cite[Chapter 6]{carmonadelarue1}. As discussed in 
 \cite{cecchindelarue},  
a tentative strategy for approximating the elements 
of ${\mathcal P}({\mathbb R}^d)$ is then to 
replace the Fourier basis by the Hermite one. 
Mean field games, to which our analysis applies (up to some extent),  
have also been studied on subdomains of ${\mathbb R}^d$, see
for instance \cite{cirant:neumann,dipersio:garbelli:ricciardi,porretta:arma}. 
Although  boundary conditions would  directly impact the
structure of the  
Fokker-Planck equations modelling the evolution of the underlying populations
(among others, their solutions  would no longer be probability measures in presence 
of  Dirichlet conditions), we can consider using a Galerkin method similar to that introduced here, for example
by replacing the Fourier functions with 
 Chebyshev polynomials 
 (see the monograph 
\cite{boyd}).} A more delicate objective is to relax the assumptions and in particular to remove the convexity conditions that we here require on the cost coefficients (in the measure argument). 
In this regard, the work   \cite{cardaliaguet-souganidis:2} offers a promising road: around a unique optimal path, 
the value function of the mean field control problem features extra regularity properties even though the coefficients are not convex (but are also smooth). 
The very good point is that this property has a form of genericity because uniqueness happens for `many' initial conditions (we refrain from formalizing 
the notion of `many' here and just refer to \cite{cardaliaguet-souganidis:2} for more comments on this point). Clearly, 
our hope is to benefit from these extra smoothness properties to approximate the value function locally (in comparison, 
our study here is global).  
Another prospect is to address numerical and statistical applications. 
We refer for instance to \cite{achdoulauriere,lauriere} for an overview of these topics. 
This is our choice not to address this question in the current paper, but this is part of our project to do so in the future. In short, our results suggest that, under the standing assumption, one should obtain a good numerical or statistical approximation of the optimal trajectory (for a given initial condition) by solving numerically or by learning a finite dimensional control problem parametrized by a finite number of Fourier coefficients. Consistently, this should offer a way to approximate numerically (or statistically) the value function, originally defined on a space of infinite dimension, by a function depending on a finite number of parameters only.  
Of course, in order to get a fully implementable scheme, one should also discretize the time parameter; we refer to 
the review paper
\cite{zuazua}
for an overview of the underlying stakes in the framework
of control theory of partial differential equations.

\subsection{Notations and functional setting}
Throughout, 
we work on the $d$-dimensional torus $\T^d:=\R^d\slash \Z^d$, $d\in\N\setminus\{0\}$, whose elements are denoted by $x = (x_1,\dots,x_d)$. We write $x\cdot y$ for the usual scalar product among two $d$-dimensional vectors, and $\lvert \cdot\rvert$ for the Euclidean norm. For any $k\in\Z^d$, the notation $\lvert k\rvert$ stands for $\max_{i=1,\dots,d}\lvert k_i\rvert$. For $z\in\C$, we denote by $\bar{z}$ its conjugate, by $\Re(z)$ its real part and 
by $\Im(z)$ its imaginary part.

\subsubsection*{Space of measures, spaces of functions.}
We refer to the set of Borel probability measures over $\T^d$ by writing $\prob(\T^d)$, and we equip it with the weak topology. Notice that, since $\T^d$ is compact,  $\prob(\T^d)$ is  also compact. Moreover, in this case,  weak convergence is equivalent, for instance, to convergence with respect to the $1$-Wasserstein distance
\begin{equation*}
	W_1(\mu,\nu):=\sup_{\phi}\left\lvert \int_{\T^d}\phi(x)(\mu - \nu)(\!\de x)\right\rvert,\quad\mu,\nu\in\prob(\T^d),
\end{equation*}
\addtxtr{where the supremum is taken over the 1-Lipschitz functions from ${\mathbb T}^d$ to ${\mathbb R}$}. When a probability measure $\mu\in\prob(\T^d)$ has a density (with respect to the Lebesgue measure $\de x$), we identify $\mu$ and its density $\de\mu\slash\de x$, which we write $x\mapsto\mu(x)$. Given a positive measure $\mu$ and a measurable map $\varphi$, the integral of $\varphi$ with respect to $\mu$ is denoted by $\scalprod{\mu}{\varphi} = \int_{\T^d}\varphi(x)\mu(\!\de x)$. We use the same notation for the scalar product in $L^2(\T^d)$. \addtxtr{More generally, we will also use the spaces $L^q(\T^d)$}, $q\in\N\setminus\{0\}$, \addtxtr{which} are meant \addtxtr{as the standard $L^q$ spaces} with respect to the Lebesgue measure. The $L^q(\T^d)$ norm will be denoted by $\norm{\cdot}_q$. 
\vskip 6pt

\noindent 
We  also introduce various spaces of real-valued smooth functions over $\T^d$. For $q\in\N$, we denote by $\rmC^q(\T^d)$ the space of $q$-times continuously differentiable real-valued functions over $\T^d$, with the convention $\rmC(\T^d):=\rmC^0(\T^d)$. For $\varphi\in\rmC^q(\T^d)$, we denote its derivative of order $q$-th by $\nabla^q\varphi\colon\T^d\to\R^{d\times q}$, whilst for a multi-index $a\in\N^d$, we write $\partial^a\varphi:=\partial^{a_1}_{x_1}\dots\partial^{a_d}_{x_d}\varphi$. If $\varphi\in\rmC(\T^d)$, we set $\norm{\varphi}_\infty:=\sup_{x\in\T^d}\lvert \varphi(x)\rvert$, and if $\varphi\in\rmC^q(\T^d)$, we introduce $\norm{\varphi}_{\infty,q} := \norm{\varphi}_\infty + \sum_{i=1}^q \norm{\nabla^i\varphi}_\infty$.
\vskip 6pt

\noindent 
More generally, we use the notation $\rmC^k(E_1;E_2)$ (resp. $\rmC_b^q(E_1;E_2)$) for mappings $\varphi\colon E_1\to E_2$ that are $q$ times continuously differentiable (resp. $q$ times continuously differentiable with bounded derivatives up to the order  $q$), \addtxtr{with $E_1$ and $E_2$ being appropriate finite dimensional spaces (we will omit the notation $E_2$ whenever it is equal to $\R$)}. In particular, we can extend in a natural way the definition of the norms $\norm{\cdot}_\infty$ and $\norm{\cdot}_{\infty,q}$ to \addtxtr{functions defined on $E_1=\T^d \times \R^d$ and valued in 
$E_2=\R$.} For $\alpha\in(0,1)$ and $q\in\N$, we denote by $\rmC^{q + \alpha}(E_1;E_2)$ the space of functions in $\rmC^q(E_1;E_2)$ with \addtxtr{$\alpha$-H\"older continuous derivatives up to the order $q$.} \addtxtr{In order to deal with} functions \addtxtr{defined} over a product space $E_1\times E_2$ and $q\in\N$, \addtxtr{we denote by}  $\rmC^{1,q}(E_1\times E_2)$ \addtxtr{the space} of functions \addtxtr{that}   are \addtxtr{once} differentiable \addtxtr{with respect to the argument in} $E_1$ and $q$ times 
\addtxtr{differentiable with respect to the argument} in $E_2$, with jointly continuous derivatives. 
The extension of this definition to the other cases discussed above is straightforward.

\subsubsection*{Fourier analysis} \addtxtr{We also introduce frequently used notation from Fourier analysis}. Let $\{e_k\}_{k\in\Z^d}$ be the Fourier basis on $\T^d$ given by the mappings $x\mapsto e_k(x):=e^{\ii2\pi k\cdot x}\in\C$, where $\ii^2 = -1$. Thus, for a Borel measure $\mu$
\addtxtr{on $\T^d$,} we can \addtxtr{denote its} family of Fourier coefficients $\{\widehat{\mu}(k)\}_{k\in\Z^d}\subset\C$ by
\begin{equation*}
	\widehat{\mu}(k):=\int_{\T^d}\addtxtr{e_{-k}}(x)\mu(\!\de x), \quad k \in {\mathbb Z}^d.
\end{equation*}
Notice that $\mu(\T^d) =\widehat{\mu}(0)$. Similarly, for a measurable function $\varphi\colon\T^d\to\R$ we set $\widehat{\varphi}(k):=\int_{\T^d}e_k(x)\varphi(x)\de x$. For any $N\in\N$, we introduce the subspace $\prob_N(\T^d)\subset\prob(\T^d)$ of probability measures over $\T^d$ whose Fourier coefficients of order greater than $N$ are equal to $0$. This set is one-to-one with an open subset of $\R^{\kappa(N)}$, with $\kappa(N)<2(2N-1)^d$ (see \cite[Proposition 3.1]{cecchindelarue}).

Given the family of Fourier coefficients $\{\widehat{\mu}(k)\}_{\in\Z^d}$ associated to $\mu$, we can look at the function obtained by truncating the Fourier series at level $N\in\N\setminus\{0\}$:
\begin{equation}
\label{eq:truncation:Fourier:DN:def}
	\T^d\ni x\mapsto(\mu*D^N)(x) = \sum_{\lvert k\rvert\leq N} \widehat{\mu}(k)e_k(x).
\end{equation}
The decomposition clearly shows that $\reallywidehat{\mu*D^N}(k) = \widehat{\mu}(k)$ for any $\lvert k\rvert\leq N$, and $\reallywidehat{\mu*D^N}(k) = 0$ otherwise. The notation $\mu*D^N$ is reminiscent of the fact that the truncation can be obtained by convolution of $\mu$ with  the so-called $N$-th Dirichlet kernel $D^N\colon\T^d\to\R$ (see, e.g.~\cite[Chapter 3.1.2]{grafakos}). \addtxtr{Under appropriate regularity assumptions,}  the truncation $\mu*D^N$ 
\addtxtr{has} good properties of convergence in suitable norms (see Appendix \ref{sec: prelimiraries}). \addtxtr{However}, one of the major drawbacks is the fact that $D^N$ is not a density, since it may become negative. Thus, even if $\mu\in\prob(\T^d)$, \addtxtr{$\mu*D^N$ does not belong (in general) to $\prob_N(\T^d)$}.  Further properties of the Dirichlet kernels can be found for instance in \cite[Chapter 3.1.3]{grafakos}. 

For $q\in\R$, we can now introduce the Sobolev space $H^q(\T^d)$ as the space of maps $\varphi\colon\T^d\to\R$ such that 
\begin{equation*}
	\norm{\varphi}^2_{2,q}:= \sum_{k\in\Z^d} (1 + \lvert k \rvert ^2)^{q} \lvert \widehat{\varphi}(k) \rvert^2
\end{equation*}
is finite. The functional $\norm{\cdot}_{2,q}$ is a norm on $H^q(\T^d)$ induced by a scalar product \addtxtr{that}  makes the space Hilbertian. Notice that $H^0(\T^d) = L^2(\T^d)$. \addtxtr{We will simply denote
$\sum_{k\in\Z^d}\lvert \widehat{\varphi}(k)\rvert^2$ by
$\norm{\cdot}^2_{2}$ and not by $\norm{\cdot}^2_{2,0}$}.

\subsubsection*{Euclidean derivatives.} For a mapping $\T^d\times\R^d\ni(x,p)\mapsto\varphi(x,p)\in\R$, we will often deal with partial or global derivatives of $\varphi$, which prompts us to briefly describe 
\addtxtr{some corresponding notation}. We denote by $\diff_x\varphi$ and $\diff_p\varphi$ the partial derivatives of $\varphi$ with respect to the first and second variables respectively. Notice that these derivatives are actually gradients since $x\in\T^d$ and $p\in\R^d$. In particular $\diff_x\varphi = (\partial_{x_1}\varphi,\dots,\partial_{x_d}\varphi)^\top$ and $\diff_p\varphi = (\partial_{p_1}\varphi,\dots,\partial_{p_d}\varphi)^\top$, where $\{\partial_{x_i}\varphi\}_{i=1}^d, \{\partial_{p_i}\varphi\}_{i=1}^d$ are $\R$-valued partial derivatives. 
It will also be useful to \addtxtr{have a shorter} notation for the mapping $\T^d\ni x\mapsto \varphi(x, g(x))$, \addtxtr{for a given} $g\colon\T^d\to\R^d$. In this case, we will write $\varphi(\cdot, g)$ and \addtxtr{$\nabla [\varphi(\cdot, g)] = \diff_x\varphi(\cdot,g) +\nabla^\top g\diff_p\varphi (\cdot,g)$, whenever $\varphi$ and $g$ are smooth enough and with $\nabla^\top g$ denoting the transposed Jacobian matrix of $g$.}

\subsubsection*{Derivatives on $\prob(\T^d)$}
In our discussion, it will be necessary to differentiate functions $\phi\colon\prob(\T^d)\to\R$. 
\addtxtr{Among the many notions available in the literature (see e.g.,~ \cite[Chapter 5]{carmonadelarue1} or \cite[Chapter 2.2, Appendix A.1]{cardaliaguetdelaruelasrylions}), the derivative   we adopt here is the so-called flat (or linear) derivative.}
	\begin{definition}\label{def: lf}
		 We say that a continuous function $\phi:\prob(\T^d) \to \R$ is continuously differentiable in the linear sense if there exists a continuous map $\lf\phi:\addtxtr{\prob(\T^d) \times\T^d} \to \R$ such that, for any $\mu,\nu\in\prob(\T^d)$, 
		\begin{equation}
    			\phi(\mu)-\phi(\nu) = \int_0^1\int_{\R^d} \lf\phi(t\mu + (1-t)\nu, x) (\mu - \nu)(\!\de x)\de t.
		\end{equation}
	The space of continuous functions which are linearly continuously differentiable is denoted by $\rmC^1(\prob(\T^d))$.
	\end{definition}
Note that $\lf\phi$ is defined up to an additive constant. \addtxtr{A standard condition 
to guarantee uniqueness is to set:}
\begin{equation*}
	\int_{\T^d}\lf\phi(\mu,x)\mu(\!\de x) = 0,\quad \mu\in\prob(\T^d).
\end{equation*}
Moreover, by iterating Definition \ref{def: lf}, we can introduce the spaces $\rmC^k(\prob(\T^d))$ for $k\in\N\setminus\{0\}$.

\subsubsection*{About constants.} Finally, throughout the paper we use $C$ for positive constants that may change from line to line. Most often, we will highlight the parameters on which $C$ depends.

\subsection{Organization of the paper}
The problem is exposed in detail in Section 
\ref{se:2}. Therein, we also clarify the main assumptions that are used in our analysis and we provide 
a preliminary version of the main results in the form of meta-statements that are refined next, see in particular 
Meta-Theorems \ref{meta-thm:1}, \ref{meta-thm:2}, 
\ref{mthm: conv_value}
and 
\ref{mthm: approf_four}.
Section  
\ref{se:3}
is dedicated to the analysis of the 
approximating control problem obtained by Fourier-Galerkin approximation
of the original one. 
The main result provides a characterization of the optimizers as the solutions of 
a forward-backward system, which reads as a finite-dimensional version of the standard 
mean field game system
in the original 
 mean field control problem. 
After a series of preliminary results on Fourier-Dirichlet truncations of Fokker-Planck 
and Hamilton-Jacobi-Bellman equations on the torus, 
we establish 
the well-posedness of this forward-backward system. 
Finally, Section 
\ref{sec: convergence} contains the main results of the paper about the rate of convergence of the 
Fourier-Galerkin approximation: 
Theorem 
\ref{thm: conv_of_controls}
(approximation of the optimal feedback function), 
Theorem 
\ref{thm: unif_conv_control}
(approximation of the optimal trajectory)
and
Proposition 
\ref{prop: conv_of_the_value}
(approximation of the value function). \addtxtr{Important facts from Fourier analysis are recalled in Appendix.}

\section{Preliminaries and main results}
\label{se:2}
In this section we first give a complete formulation of the problem we want to study. Then, we state the main assumptions we need in our discussion, together with some classical results from Mean Field Control (MFC) theory that we will use repeatedly. 
\subsection{Problem formulation}
We introduce the MFC problem we are interested in, together with its Fourier-Galerkin approximation. 

\subsubsection{The MFC problem}\label{sec: mfc_problem}
We fix a time horizon $T>0$ and a filtered probability space $(\Omega, \cF, \mathbb{F} = \{\cF_t\}_{t\in[0,T]},\mathbb{P})$ satisfying the usual conditions, endowed with a $d$-dimensional Brownian motion. For $\alpha\colon[0,T]\times\T^d\to \R^d$ bounded and Borel measurable, $t\in[0,T]$ and $m\in\prob(\T^d)$, let us consider the following controlled process
\begin{equation}\label{eqn: controlled_proc}
	\begin{cases}
		\de X_s &= \alpha_s(X_s)\de s + \sqrt{2}\de B_s,\quad s\in(t,T],\\
		\cL(X_t) &= m,
	\end{cases}
\end{equation}
whose marginal laws $(\mu^\infty_s : =\cL(X_s))_{t \leq s \leq T} $ satisfy the Fokker-Planck equation
\begin{equation}\label{eqn: fp}
	\begin{cases}
		\partial_s \mu^\infty_s &= \Delta\mu^\infty_s - \div(\alpha_s(\cdot)\mu^\infty_s),\quad s\in(t,T],\\
		\mu^\infty_t &= m.
	\end{cases}
\end{equation}
Since $X$ and $\mu^\infty$ depend on $t,m$ and $\alpha$, we will often write $X^{t,m,\alpha}$ and $\mu^{\infty,t,m,\alpha}$ to highlight these dependencies.\\\\
The data for our optimization problem consist of three functions
\begin{equation*}
	\cG,\cF\colon\prob(\T^d)\to\R,\quad L\colon\T^d\times\R^d\to\R,
\end{equation*}
whose properties will be specified later. We also associate with the Lagrangian $L = L(x,a)$ the Hamiltonian $H\colon\T^d\times\R^d\to\R$ via the usual formula
\begin{equation*}
	H(x,p):=\sup_{a\in\R^d} \{- L (x,a) - a\cdot p\}.
\end{equation*}
The mean field control (MFC) problem consists in minimizing (with respect to all the admissible feedback controls $\alpha$) the following cost:
\begin{equation*}
	J^\infty(\alpha,t,m):= \cG(\mu^{\infty,t,m,\alpha}_T) + \int_t^T \left\{\cF(\mu^{\infty,t,m,\alpha}_s) + \int_{\T^d}L(x,\alpha_s(x)) \mu^{\infty,t,m,\alpha}_s(\!\de x)\right\}\de s.
\end{equation*}
We introduce the value function $V\colon[0,T]\times\prob(\T^d)\to\R$ associated with the MFC problem, that is
\begin{equation}\label{eqn: value}
	V^\infty(t,m) := \inf_{\alpha\in\cA} J^\infty(\alpha, t, m),
\end{equation}
where {$\cA:=\{\alpha\colon[0,T]\times\T^d\to \R^d\}$ denotes the set of bounded and Borel measurable feedback controls}. Note that in this formulation, the MFC problem reads as a deterministic control problem. The main challenge lies in the infinite-dimensional, specifically measure-valued, nature of the state (which satisfies \eqref{eqn: fp}). Notice that, in our setting the formulation of the problem with feedback controls coincides with the formulation with open-loop controls, {see \cite{lacker}}.
\subsubsection{The approximating MFC problem}\label{sec: approx_MFC}
Our main goal is to provide a way to approximate the MFC problem presented above. Let us fix $N\in\N\setminus\{0\}$ and, for $t\in[0,T]$, $m\in\prob(\T^d)$ and a feedback control $\alpha\colon[0,T]\times\T^d\to \R^d$, bounded and Borel measurable, let us consider the following  approximating controlled Fokker-Planck equation:
\begin{equation}
\label{eqn: fp_n_OC}
	\begin{cases}
		\partial_s \mu^N_s &= \Delta\mu^N_s - \div(\alpha_s(\cdot)(\mu^N_s*D^N)),\quad s\in(t,T],\\
		\mu^N_t &= m,
	\end{cases}
\end{equation}
where $D^N$ denotes the $N$-th Dirichlet's kernel, and the convolution $\mu*D^N$ can be read as the truncated Fourier series of $\mu$ (recall \eqref{eq:truncation:Fourier:DN:def}, see also Appendix \ref{sec: prelimiraries} for further details). We introduce the approximating costs $\mu\mapsto\cF^N(\mu):=\cF(\mu*D^N)$ and $\mu\mapsto\cG^N(\mu) := \cG(\mu*D^N)$, whenever $\mu*D^N\in\prob(\T^d)$. \vskip 6pt

\noindent We define the functional to minimize in our new MFC problem as
\begin{equation*}
	J^N(\alpha,t,m):= \cG^N(\mu^{N,t,m,\alpha}_T) + \int_t^T \left\{\cF^N(\mu^{N,t,m,\alpha}_s) + \int_{\T^d}L(x,\alpha_s(x)) (\mu^{N,t,m,\alpha}_s*D^N)(\!\de x)\right\}\de s,
\end{equation*}
and the associated value function $V^N\colon[0,T]\times\prob(\T^d)\to\R$ as
\begin{equation}\label{eqn: value_N}
	V^N(t,m) := \inf_{\alpha} J^N(\alpha, t, m),
\end{equation}
the infimum being taken over the bounded and measurable feedback
functions for which
the 
equation 
\eqref{eqn: fp_n_OC}
(when formulated in the weak sense) 
has 
a unique solution $\mu^N=(\mu^N_t)_{0 \leq t \leq T}$ such that  $\mu^N*D^N=(\mu^N_t*D^N)_{0 \le t \le T}$ 
takes values in ${\mathcal P}({\mathbb T}^d)$. Solvability 
of \eqref{eqn: fp_n_OC} is discussed in more detail in Subsection
\ref{sec: Pontryagin}. 

In Subsection \ref{sec: finite_dum_OC_prob}, we will show that the approximating MFC problem \eqref{eqn: value_N}, which we have stated here as an  infinite-dimensional control problem, can be reformulated as a finite-dimensional problem. \addtxtr{To stay focused on the main goal, we postpone the description of the ingredients of this new formulation to Subsections  \ref{sec: fourier_properties} and  \ref{sec: finite_dum_OC_prob}.}

\begin{remark}\label{rmk: notsoeasy}
	\addtxtr{Notice that to ensure the well-posedness of the approximating MFC problem, we restrict ourselves to controls for which the corresponding trajectory $(\mu^N_t * D^N)_{0 \le t \le T}$ takes values in ${\mathcal P}({\mathbb T}^d)$. This restriction is necessary because  Dirichlet  kernels are not densities,  as a consequence of which  the truncation $\mu^N_t * D^N$  may not be  a probability measure, even if $\mu^N$ is. Additionally, since \eqref{eqn: fp_n_OC} does not  satisfy the maximum principle (due to its non-local structure), it is neither obvious nor generally true that $\mu^N$  itself is positive valued. We will address these questions in Subection \ref{sec: prop_of_mutilde} and more specifically in \S\ref{sec: mutilde_and_truc_are_prob}. In the limiting MFC problem (i.e.,  $N = \infty$), there is no need to impose such a requirement on the  controls,  because the solution of the Fokker-Planck equation \eqref{eqn: fp} remains a probability measure for any bounded and measurable $\alpha$. }	 
\end{remark}
\subsection{Assumptions}
\label{subse:assumptions}
 We now present all the assumptions required in the following sections. Let us recall that the costs are defined as
\begin{equation*}
	\cF\colon\prob(\T^d)\to\R,\quad\cG\colon\prob(\T^d)\to\R.
\end{equation*}
Moreover, rather than imposing conditions on the Lagrangian $L$, we will focus on the Hamiltonian $H$. Subsequently, we will briefly discuss how these properties are reflected on $L$.
\begin{assumption}\label{hp: 1+2} 
	Let $q\in\N$ be such that $q\geq d+3$. We require that: 
	\begin{itemize}
		\item[(FG.1)] $\cF,\cG\in\rmC^1(\prob(\T^d))$, with derivatives denoted by $\lf\cF$ and $\lf\cG$.
		\item[(FG.2)] \addtxtr{Both $\lf\cF$ and $\lf\cG$ are of class $\rmC^{q}(\T^d)$ with respect to the state variable, and the norms $\norm{\lf\cF(\mu,\cdot)}_{\infty, q}$ and 
		 $\norm{\lf\cG(\mu,\cdot)}_{\infty, q}$ are bounded uniformly in $\mu\in\prob(\T^d)$.}
		\item[(FG.3)] \addtxtr{Both the restrictions of $\lf\cF$ and $\lf\cG$ to 
		$(\prob(\T^d)\cap L^2(\T^d)) \times {\mathbb T}^d$ 	
		are Lipschitz  continuous in $\mu$ with respect to
		 the $L^2$ norm, uniformly in the state variable $x$.}
		More precisely, there exists a positive constant $L_{\cF}$ such that, for any $\mu,\mu'\in\prob(\T^d)\cap L^2(\T^d)$ and $x\in\T^d$, it holds
		\begin{equation*}
			\left\lvert\lf\cF(\mu,x) - \lf\cF(\mu',x)\right\rvert\leq L_{\cF}\norm{\mu - \mu'}_2,
		\end{equation*}
		and the same for $\lf\cG$.  
		\item[(FG.4)] $\cF$ and $\cG$ are convex for the linear structure of $\prob(\T^d)$: for any $\mu,\mu'\in\prob(\T^d)$ and $\lambda\in[0,1]$, it holds $\cF(\lambda \mu + (1-\lambda)\mu')\leq\lambda\cF(\mu) + (1-\lambda)\cF(\mu')$, and the same  for $\cG$.
		\item[(H.5)] $H\in\rmC^{q} (\T^d\times\R^d)$, and for any $(x,p)\in\T^d\times\R^d$,
		\begin{equation*}
			\frac{1}{C_H}\bbI_{d\times d}\leq \diff^2_{pp} H(x,p)\leq C_H \bbI_{d\times d},
		\end{equation*}
		for a certain $C_H>1$, 
		where $\bbI_{d\times d}$ denotes the $d$-dimensional identity matrix.
	\end{itemize}
\end{assumption}

\begin{remark}
\label{rem:choice of q}
The reason why the regularity parameter $q$ is required 
to be an integer is mostly for convenience. Indeed, we need to compute $\| \cdot \|_{\infty,q}$ norms in the analysis, which is easier to do when $q$ belongs to ${\mathbb N}$. However, we see no major technical difficulties in adapting the results to 
$q$ being real. Very likely, we would even obtain a threshold better than 
$d+3$ for the parameter $q$ by allowing all the regularity indices to be reals in the analysis, but the gain would be very minor in the end.
\end{remark}

\begin{remark}\label{rem: convex-monoton}
	The convexity assumption Assumption \ref{hp: 1+2} - (FG.4) is equivalent to the 
	 Lasry-Lions monotonicity  property  of $\lf\cF$ and $\lf\cG$. More precisely, for any $\mu,\mu'\in\prob(\T^d)$,  one has
	\begin{equation*}
		\int_{\T^d} \left(\lf\cF(\mu,x) - \lf\cF(\mu',x)\right)(\mu-\mu')(\!\de x)\geq0.
	\end{equation*}
	and the same holds for $\lf\cG$ (see, e.g.,~\cite[Remark 5.75]{carmonadelarue1}).
\end{remark}
\begin{remark}
	Notice that Assumption \ref{hp: 1+2} - (H.5) implies \addtxtr{that $H$ is strongly convex and at most of quadratic growth in $p$}.
	Regarding the Lagrangian $L$, Assumption \ref{hp: 1+2} - (H.5) implies \addtxtr{that $L$ is convex in $a$}. Moreover,  there exists a positive constant $C$ such that $C + C\lvert a\rvert^2\geq L(x,a)\geq -C + \frac{1}{C}\lvert a \rvert^2$. 
	
	\addtxtr{For instance, Assumption \ref{hp: 1+2} - (H.5) is satisfied by $H(x,p):= \frac{1}{2}\lvert p\rvert^2 + \nu(x)\cdot p$, for a  suitably smooth vector field $\nu\colon\T^d\to\R^d$.} \end{remark}
\begin{example}\label{ex: 0}
	Regarding Assumption \ref{hp: 1+2} - (FG.1-2-3), typical examples for $\cF$ and $\cG$ are the so-called cylindrical functions. A function $\Phi\colon\prob(\T^d)\to\R$ is cylindrical if
	\begin{equation*}
		\Phi(\mu):=\phi(\scalprod{\mu}{\psi_1},\dots,\scalprod{\mu}{\psi_k}),
	\end{equation*}
	where $k\in\N\setminus\{0\}$, $\phi\colon\R^k\to\R$ is in $\rmC^2_b(\R^k)$ and $\psi_1,\dots,\psi_k\colon\T^d\to\R$ are
	 in $\rmC^{q}(\T^d)$. \addtxtr{If $\phi$ is also asked to be convex, then it satisfies} Assumption \ref{hp: 1+2} - (FG.4). 
\end{example}
\begin{example}\label{ex: 1}
	\addtxtr{In Example \ref{ex: 0}, we can choose the functions  $\{\psi_j\}_{j=1}^k$ 
	 as elements of the Fourier basis $\{x\mapsto\sin(2\pi j\cdot x), \ x\mapsto\cos(2\pi j\cdot x) ; \ j \in \Z^d\}$}. Then, $\cF$ is convex if it reads as a convex function of the $k$ Fourier coefficients of $\mu$.
\end{example}

\addtxtr{
\begin{example}
While Example 
\eqref{ex: 0}
is somewhat finite-dimensional, we 
can also construct truly infinite-dimensional examples in a systematic manner. Take for instance 
a function $\Phi$ that is convex in the sense of (FG.4) and that has a jointly continuous derivative 
$\delta_\mu \Phi : {\mathcal P}({\mathbb T}^d) \times \T^d\to\R$ which is Lipschitz continuous 
in the argument $\mu$ with respect to $W_1$. These are standard assumptions in MFC theory
(see for instance \cite{cardaliaguetdelaruelasrylions,daudindelaruejackson}). 
Consider also a smooth compactly supported even function $\varphi : {\mathbb R}^d \rightarrow {\mathbb R}$. We claim that 
the function $\mu \mapsto \Phi(\mu*\varphi)$ satisfies (FG.2), which follows from the 
formula 
\begin{equation*} 
\delta_\mu \Bigl[ \Phi \bigl( \mu * \varphi \bigr) \Bigr](\mu,x)  
=  \Bigl[ \delta_\mu \Phi(\mu*\varphi)(\mu,\cdot) * \varphi\Bigr](x), \quad \mu \in {\mathcal P}({\mathbb T}^d), 
\ x \in {\mathbb T}^d. 
\end{equation*} 
As for (FG.3), the fact that $\delta_\mu \Phi$ is Lipschitz continuous in the measure argument 
with respect to $W_1$ automatically implies that it is Lipschitz continuous
with respect to $\| \cdot \|_2$, when the measure argument is restricted to 
${\mathcal P}({\mathbb T}^d) \cap L^2({\mathbb T}^d)$, 
see Remark 
\ref{rem: tv_vs_sobolev}. 
\end{example} }

\begin{example}
	We recall that  $\{{e}_k\}_{k\in\Z^d}$ denotes the complex Fourier basis. For fixed $\mu_0\in\prob(\T^d)$ and $r\in\N$ such that $r \geq d+2$, we then consider 
	$$\Phi(\mu) := \norm{\mu - \mu_0}^2_{H^{-r}} = \sum_{ k \in\Z^d}\frac{\lvert \hat{\mu}(k) - \hat{\mu}_0(k)\rvert^2}{(1 + \lvert k\rvert^2)^r}.$$ Then, $\Phi$ is clearly convex. Moreover, $\Phi\in\rmC^1(\prob(\T^d))$ and it holds
$$\lf\Phi(\mu,x) = 2\sum_{k\in\Z^d}
\Re \biggl[ 
\frac{ \hat{\mu}(k) - \hat{\mu}_0(k)}{(1 + \lvert k\rvert^2)^r} {e}_k(x) \biggr].$$ For any $\mu,\mu'\in\prob(\T^d)\cap L^2(\T^d)$ and $x\in\T^d$,  Cauchy-Schwarz inequality  yields
	\begin{equation*}
		\left\lvert\lf\Phi(\mu,x) - \lf\Phi(\mu',x)\right\rvert = 2\sum_{k\in\Z^d}
		\Re \biggl[ 
		\frac{ \hat{\mu}(k) - \hat{\mu}'(k)}{(1 + \lvert k\rvert^2)^r} {e}_k(x) \biggr] \leq C(d,r) \norm{\mu - \mu'}_2,
	\end{equation*} with $C(d,r):= 2\big(\sum_{k\in\Z^d}(1 + \lvert k\rvert^2)^{-2r}\big)^{\frac{1}{2}}$, and so Assumption \ref{hp: 1+2} - (FG.3) holds. Regarding (FG.2), every time we differentiate $\lf\Phi$ with respect to $x$ we obtain a multiplicative term of order $\lvert k\rvert $ in the series. Thus, to keep the series converging, we can differentiate up to the $q$-th order, with $q:=2r-d-1$. Notice that, since $r \geq d+2$, it holds $q \geq d+3$. With a minor modification one can also consider the case $r\in\R$.
\end{example}

\noindent To conclude this section, we state the regularity assumptions we need for the initial condition of the Fokker-Planck equation in the MFC problem:
\begin{assumption}[Assumptions for the initial measure]\label{hp: 3}
	The initial condition $m\in\prob(\T^d)$ is such that:
	\begin{itemize}
		\item[(IC.1)] $m\in H^{q-1}(\T^d)$,  for a certain $q\geq d+3$, i.e. $m$ has a density with respect to the Lebesgue measure which is of class $H^{q-1}(\T^d)$; 
		\item	[(IC.2)]$\inf_{x\in\T^d} m(x)\geq\gamma$ for a certain \addtxtr{$\gamma \in (0,1)$}. We will often use the notation $m\geq\gamma$.
	\end{itemize}
\end{assumption}

\subsubsection{Classical results for the MFC problem}\label{sec: classicresult}
Here we present some results about the MFC control problem introduced in Section \ref{sec: mfc_problem},  which  can be found  in \cite[Chapter 6]{carmonadelarue1} or \cite[Chapter 3.7]{cardaliaguetdelaruelasrylions} (see also the references therein).
\vskip 6pt

\noindent Let us consider the value function $V^\infty\colon[0,T]\times\prob(\T^d)\to\R$ associated with the MFC problem defined in \eqref{eqn: value}. In our framework, the Dynamic Programming Principle (DPP) holds and we can introduce the following Hamilton-Jacobi-Bellman (HJB) equation:
\begin{equation}\label{eqn: HJB}
\begin{cases}
	\partial_t V^\infty(t,m)
	+ \int_{\T^d} \Delta [\lf V^\infty](t,m,x)m(\!\de x)\\
	\quad\quad- \int_{\T^d}H(x,\nabla[\lf V^\infty](t,m,x)) m(\!\de x)
	+ \cF(m) = 0 \quad &\text{ in }[0,T]\times\prob(\T^d),\\
	\quad\\
	 V^\infty(T,m)  = \cG(m) \qquad\qquad\qquad\qquad &\text{ in }\prob(\T^d).
\end{cases}
\end{equation}
If the data are smooth enough, it has been proved that the value function $V^\infty$ is the unique classical solution to equation above (see Theorem 3.7.1 in \cite{cardaliaguetdelaruelasrylions}). \vskip 6pt

\noindent Under our regularity assumptions, we can rely on the Pontryagin principle (see, e.g.~ \cite[Chapter 6.2.4]{carmonadelarue1}) to study the MFC problem. In particular, let us consider the forward-backward system
\begin{equation}\tag{$FB_\infty$}\label{eqn: fwdbkw_infty_intro}
	\begin{cases}
		\partial_t \mu^\infty_t &= \Delta\mu^\infty_t + \div(\diff_p H(\cdot,\nabla u^\infty_t)\mu^\infty_t),\quad  t\in[0,T],\\
		\mu^\infty_0 &= m,\\
		\partial_ t u_t^\infty  &= - \Delta u_t^\infty  + H(\cdot,\nabla u^\infty_t) - \lf\cF(\mu^\infty_t,\cdot),\quad t\in[0,T],\\
		u^\infty_T &= \lf\cG(\mu^\infty_T, \cdot).
	\end{cases}
\end{equation}
Thanks to Assumption \ref{hp: 1+2}  and by a classical result from mean field game theory, there exists a unique  {classical solution $(\mu^\infty, u^\infty)$ to \eqref{eqn: fwdbkw_infty_intro} (see \cite[Proposition 3.1.1]{cardaliaguetdelaruelasrylions}) ($\mu^\infty$ being in particular a flow of smooth densities).} 
Thus,  $\alpha^{*,\infty}(\cdot) = -\diff_p H(\cdot, \nabla u^\infty(\cdot))$ is the optimal feedback for the MFC problem \eqref{eqn: value}, with $\mu^\infty$ being the associated optimal trajectory. Moreover, under higher regularity conditions (and hence higher regularity properties of $V^\infty$, 
see \cite[ Proposition 3.7.2]{cardaliaguetdelaruelasrylions}), it holds that $\alpha_t^{*,\infty}(\cdot) = -\diff_p H(\cdot, \nabla[\lf V^\infty](t,\mu^\infty_t,\cdot) )$. 

Notice that Assumption \ref{hp: 1+2} is slightly different from the hypothesis of \cite[Proposition 3.1.1]{cardaliaguetdelaruelasrylions}. However, the same fixed-point argument as the one used in the original proof can be repeated, provided that the following adjustments are made: exchanging the 1-Wasserstein distance with the $L^2$ distance, considering  the fixed-point as an element of a space of time-dependent  probability densities  in $L^2(\T^d)$, and choosing an initial condition with a density $m\in L^2(\T^d)$.
\\

The main goal of our work, see Section \ref{sec: main_results}, is first to provide a suitable approximation for the system \eqref{eqn: fwdbkw_infty_intro} by  means of the auxiliary problem \eqref{eqn: fp_n_OC}--\eqref{eqn: value_N}. The next step is to derive 
an approximation of the value function $V^\infty$, and so an approximation of the solution to the infinite dimensional HJB equation \eqref{eqn: HJB}
(with the latter being 
understood in a classical sense when the coefficients are smooth enough
and in a viscosity sense otherwise). A key feature of our approach will be the fact that all the terms involved in the approximation procedure can be computed by solving a finite dimensional problem, reducing in a relevant way the complexity of the problem.\\

To conclude this subsection, we collect some regularity results for $\mu^\infty$ and $u^\infty$. We first focus on the regularity of $u^\infty$, and we present a result that can be obtained by repeating the proof of Proposition \ref{prop: bound_aux_eq_tilde} below.
\begin{proposition}\label{prop: estimates_uinf}
	Under Assumption \ref{hp: 1+2}, for any $t\in[0,T]$, $u^\infty_t$  
	is in the space $ \rmC^{1,{q}}([0,T]\times\T^d)$. Moreover, there exists a positive constant $C =  C(d,T,\cF,\cG,H,q)$, independent of $\mu^\infty$, such that
	\begin{equation*}
		\sup_{t\in[0,T]}\norm{u^\infty_t}_{\infty,q} 
		\leq C(d,T,\cF,\cG,H,q).
	\end{equation*}
\end{proposition}
\noindent Similarly, we state a result concerning the regularity of $\mu^\infty$, which proof is analogous to the one of Proposition \ref{prop: smooth_muN}. More precisely, we provide estimates for the Sobolev norms of $\mu^\infty$. Although stronger results about the $\rmC^q$ regularity of $\mu^\infty$ are available in the literature (see, e.g.~ \cite[Proposition 3.1.1]{cardaliaguetdelaruelasrylions}), it turns out that Sobolev estimates better fit our purposes.
\begin{proposition}\label{prop: smooth_mu_inf}
	Let Assumption \ref{hp: 1+2} hold and $m$ be in $H^{q-1}(\T^d)$. Then, $\mu^\infty$ belongs to {$\rmC^{1/2, 0}([0,T]\times\T^d$) and $\mu^\infty_t$ is a density in $H^{q-1}(\T^d)$ for any $t\in[0,T]$}. Moreover, there exists a positive constant $C = C(d,T,\cF,\cG,H,q)$ such that 
	\begin{equation*}
		\sup_{t\in[0,T]}\norm{\mu^\infty_t}_{2,q-1}\leq C(d,T,\cF,\cG,H,q)\norm{m}_{2,q-1}.
	\end{equation*}
\end{proposition}
\noindent Since $q-1>\frac{d}{2}$, Sobolev embedding and Proposition \ref{prop: smooth_mu_inf} imply 
\begin{equation}\label{eqn: muinft_bound}
	\sup_{t\in[0,T]}\norm{\mu^\infty_t}_\infty \leq C(d,T,\cF,\cG,H,q)\norm{m}_{2,q-1}.
\end{equation}
Finally, we  give a condition to guarantee that $\mu^\infty_t*D^N$ remains a probability measure for any $t\in[0,T]$. This result is key in our analysis.
\begin{lemma}\label{lemma: lower_bound_mu_infty}
	Let Assumptions \ref{hp: 1+2} and \ref{hp: 3} be in force. Then, there exists a constant \addtxtr{$\bar\gamma \in (0,1)$}, which depends only on $\gamma$, such that $\inf_{t\in[0,T]}\mu^\infty_t\geq\bar\gamma>0$. Moreover, for $N$ large enough, it also holds that $\inf_{t\in[0,T]}(\mu^\infty_t*D^N)\geq\bar\gamma>0$.
\end{lemma}
\noindent As for the previous proposition, the proof is analogous to other proofs given in the following pages, and we omit it here for the sake of brevity. In particular, to obtain the first lower bound in Lemma \ref{lemma: lower_bound_mu_infty} we can argue as in the proof of Lemma \ref{lemma: lower_bound_muN}, \addtxtr{while for the second,} we can proceed as in Remark \ref{rmk: mutilde_n_prob}.

\subsection{Main results}\label{sec: main_results}
Our main results concern the convergence of the optimal control, the optimal trajectory  and the value function of the approximating MFC problem \eqref{eqn: value_N} to the solution of the MFC problem \eqref{eqn: value}. To state them, we need to introduce the analogue of the forward-backward system \eqref{eqn: fwdbkw_infty_intro} for our approximating problem:
 \begin{equation}\tag{$FB_N$}\label{eqn: FWBKW_intro}
	\begin{cases}
		\partial_t \mu^N_t &= \Delta\mu^N_t + \div(\diff_p H(\cdot,\nabla u^N_t)(\mu^N_t*D^N)),\quad t\in[0,T],\\
		\smallskip
		\mu^N_0 &= m,\\
		\smallskip
		\partial_ t  u^N_t  &= - \Delta  u^N_t  +  H( \cdot,\nabla  u^N_t)*D^N  - \lf\cF(\mu^N_t*D^N,\cdot)*D^N,\quad t\in[0,T],\\
		\smallskip
		  u^N_T &= \lf\cG(\mu^N_T*D^N,\cdot)*D^N.
	\end{cases}
\end{equation}
Section \ref{sec: approx_fwdbkw_problem} is entirely devoted to the well-posedness of the system \eqref{eqn: FWBKW_intro} and to the study of the properties of its solutions. The main result we obtain can be summarized in the following:
\begin{metatheorem}
\label{meta-thm:1}
Let Assumptions \ref{hp: 1+2} and \ref{hp: 3} be in force. Then, for $N$ large enough, the system \eqref{eqn: FWBKW_intro} is well-posed. {More precisely, it has a unique solution $(\mu^N,u^N)$ with the following two features: $u^N$ is a classical solution and belongs to $\rmC^{1,q}([0,T]\times\T^d)$ whilst $\mu^N$ is in $\rmC^{1/2,0}([0,T]\times \T^d)$ and admits a density in $H^{q-1}(\T^d)$.}
Moreover:
\begin{enumerate}[i.]
	\item $\norm{ u^N_t}_{\infty,\lfloor q-1-d/2 \rfloor}$ and $\norm{\mu^N_t}_{2,\lfloor q-2-d/2 \rfloor}$ are bounded uniformly in $t\in[0,T]$ \addtxtr{and $N$ (large)}.
	\smallskip
	\item $(\mu^N_t)_{0\leq t\leq T}$ is a time-dependent probability measure as well as $(\mu^N_t*D^N)_{0\leq t\leq T}$;
	\smallskip
	\item $\alpha^{*,N}_t(\cdot) = -\diff_p H(\cdot,\nabla  u^N_t(\cdot))$ is the optimal feedback for the approximating MFC problem.
\end{enumerate}
\end{metatheorem}
Once the existence, uniqueness and smoothness properties of the couple $(\mu^N, u^N)$ have been clarified, we can focus on the convergence. This is done in Section \ref{sec: convergence}, with a particular attention to rates. In all the following statements, $N$ is supposed to be large enough. The first result we state in this direction concerns the convergence of the optimal control, and summarizes the results of Theorem \ref{thm: conv_of_controls}, Corollary \ref{cor: conv_of_controls} and Theorem \ref{thm: unif_conv_control}.
\begin{metatheorem}
\label{meta-thm:2}
	Let Assumptions \ref{hp: 1+2} and \ref{hp: 3} be in force. Then there exists a constant $C = C(d,T,\cF,\cG,H,q)>0$ such that
	\begin{equation*}
		\left (\int_0^T\norm{\diff_p H(\cdot,\nabla  u^N_s) - \diff_p H(\cdot,\nabla u^\infty_s)}_2^2 \de s\right)^{\frac{1}{2}} \leq \frac{C}{N^{q-1}}\norm{m}^2_{2,q-1}.
	\end{equation*}
	If we also assume $\lvert \nabla \lf\cG(\mu,x) - \nabla\lf\cG(\nu,x)\rvert\leq L_{\cG}\norm{\mu - \nu}_2$, for a suitable $L_{\cG}>0$ independent of $x\in\T^d$, then there exists a constant $C = C(d,T,\cF,\cG,H,q)>0$, such that
	\begin{equation*}
		\sup_{t\in[0,T]}\norm{\diff_p H(\cdot,\nabla  u^N_t) - \diff_p H(\cdot,\nabla u^\infty_t)}_\infty \leq \frac{C}{N^{q-1-\frac{d}{2}}}\norm{m}^2_{2,q-1},
	\end{equation*}
	where $C$ also depends implicitly on $L_\cG$.
\end{metatheorem}
Once the convergence of the control has been determined, it can be used as a tool to deduce the convergence (and the associated rate) of the associated optimal trajectory and of the value function. We collect in the meta-theorem below the results of Theorem \ref{thm: unif_conv_control}, Remark \ref{rmk: conv_of_approx} and Proposition \ref{prop: conv_of_the_value}.
\begin{metatheorem}\label{mthm: conv_value}
	Let Assumptions \ref{hp: 1+2} and \ref{hp: 3} be in force. Then there exists a constant $C = C(d,T,\cF,\cG,H,q)>0$ such that
	\begin{equation*}
		\sup_{t\in[0,T]} \norm{\mu^\infty_t - \mu^N_t}_2\leq \frac{C}{N^{q-1}}\norm{m}^2_{2,q-1}, \qquad \sup_{t\in[0,T]} \norm{\mu^\infty_t - \mu^N_t}_\infty\leq \frac{C}{N^{q-1-\frac{d}{2}}}\norm{m}^2_{2,q-1},
	\end{equation*}
	and the same results hold with $\mu^N*D^N$ instead of $\mu^N$.
	Moreover, if we set for $R>0$,
	\begin{equation*}
		\cB^{q,\gamma}_R:=\{\mu\in\prob(\T^d)\text{ s.t. Assumption \ref{hp: 3} holds and } \norm{m}^2_{2,q-1}\leq R\},
	\end{equation*}
	then there exists a positive constant $C = C(d,T,\cF,\cG,H,\gamma, q, R)$ such that
	\begin{equation*}
		\sup_{m\in \cB^{q,\gamma}_R}\sup_{t\in[0,T)} \lvert V(t,m) -   V^N(t,m)\rvert\leq \frac{C}{N^{q-1}}.
	\end{equation*}
\end{metatheorem}
 To achieve the $L^\infty$ estimate for $\mu^\infty - \mu^N$ in the meta-theorem above, it is actually necessary to also require that $\lvert \nabla \lf\cG(\mu,x) - \nabla\lf\cG(\nu,x)\rvert\leq L_{\cG}\norm{\mu - \nu}_2$, for a suitable $L_{\cG}>0$ independent of $x\in\T^d$. We have omitted this additional assumption in the statement of Meta-Theorem \ref{mthm: conv_value} to provide a clearer and more concise exposition at this preliminary stage (see Theorem \ref{thm: unif_conv_control} for the complete statement).\\

To conclude, we describe a key feature of our approximation, which is pointed out in Subsection \ref{sec: fourier_properties}. Since the final goal is to provide an explicit way to compute the approximation, we focus on the Fourier coefficients of $\mu^N$ and $ u^N$. More precisely, the following statement provides a way to compute $\{\widehat{\mu}^N(k)\}_{k\in\Z^d}$ and $\{\widehat{ u}^N(k)\}_{k\in\Z^d}$. Notice that from this, we can easily obtain the Fourier coefficients of $\nabla u^N$ by the relation $\reallywidehat{\nabla u^N}(k) = \ii2\pi k \widehat{ u}^N(k)$, for any $k\in\Z^d$.
\begin{metatheorem}\label{mthm: approf_four}
	Let Assumptions \ref{hp: 1+2} and \ref{hp: 3} be in force. Then, the Fourier coefficients  $\{(\widehat{\mu}^N(k), \widehat{ u}^N(k)\}_{\lvert k\rvert \leq N}$ can be computed by solving the following nonlinear system of ODEs:
	\begin{equation*}
		\begin{cases}
			\frac{\de}{\de t} \widehat{\mu}^N_t(k) &= -\lvert 2\pi k \rvert^2 \widehat{\mu}^N_t(k) + \ii 2\pi k\cdot \sum_{\lvert k - l\rvert\leq N} \reallywidehat{\diff_p H(\cdot,\nabla u^N_t)}(l)\widehat{\mu}^N_t(k-l),\quad \lvert k \rvert \leq N,\\
			\smallskip
			 \widehat{\mu}^N_0(k) &= \widehat m (k),\\
			 \smallskip
			 \frac{\de}{\de t} \widehat{ u}^N_t(k) &= \lvert2\pi k\rvert^2  \widehat{ u}^N_t(k) - \reallywidehat{ \lf\cF(\mu^N_t*D^N,\cdot)}(k) + \reallywidehat{H(\cdot,\nabla  u^N_t)}(k),\quad\lvert k\rvert\leq N,\\
			 \smallskip
			 
	\widehat{  u}^N_T (k)& = \reallywidehat{ \lf\cG(\mu^N_T*D^N,\cdot)}(k).
		\end{cases}
	\end{equation*}
	Moreover, the coefficients $\{ \widehat{ u}^N(k)\}_{\lvert k\rvert > N}$ are identically null, \addtxtr{while} the coefficients $\{\widehat{\mu}^N(k)\}_{\lvert k\rvert > N}$ \addtxtr{solve} the linear system 
	\begin{equation*}
		\begin{cases}
			\frac{\de}{\de t} \widehat{\mu}^N_t(k) &= -\lvert2\pi k \rvert^2 \widehat{\mu}^N_t(k) +  g_k( u^N,\{\widehat{\mu}^N_t(j)\}_{\lvert j \rvert\leq N}),\quad \lvert k\rvert>N,\\
			\quad\\
			\widehat{\mu}^N_0(k) &= \widehat m (k),
		\end{cases}
	\end{equation*}
	where $ g_k( u^N, \{\widehat{\mu}^N_t(l)\}_{\lvert l \rvert\leq N}) :=  \ii 2\pi k\cdot \sum_{\lvert k - l\rvert\leq N}  \reallywidehat{\diff_p H(\cdot,\nabla u^N_t)}(l)\widehat{\mu}^N_t(k-l)$, $\lvert k \rvert >N$ and $\nabla u^N$ is expressed in terms of $\{ \widehat{ u}^N(k)\}_{\lvert k\rvert \leq N}$.
\end{metatheorem}
\begin{remark}
	The final part of Meta-Theorem \ref{mthm: conv_value} furnishes a systematic approach for approximating the value function $V^\infty$ and hence the solution to the Hamilton-Jacobi-Bellman (HJB) equation \eqref{eqn: HJB}. Notably, we observe that the 
	parameter $q$ driving the 	
	smoothness of the data (see Assumptions \ref{hp: 1+2} and \ref{hp: 3}) directly influences the rate of convergence of the approximation. 
\end{remark}

{\color{black}
\begin{remark}
The fact that the convergence rates are valid only for sufficiently large values of $N$ raises  
obvious practical and numerical questions. In fact, if the quantities $\mu^N$, $u^N$
and $V^N$ appearing in the three Meta-Theorems \ref{meta-thm:1}, \ref{meta-thm:2} and 
\ref{mthm: conv_value} 
  were well-defined for all $N$ (in particular for $N$ small) and if the different norms (the same as those used to measure convergence) of these quantities were finite, it would be easy to modify the constants $C$ (in the three statements) to obtain the same bounds but for all $N$.
  
  However, a difficulty repeatedly encountered in the proof comes precisely from the fact that, for $N$ small, the above quantities may, \textit{a priori}, not be defined, in the sense that the equations \eqref{eqn: FWBKW_intro} on which they depend may not be well-posed: (i) the backward equation (for $u^N$) is non-local (under the effect of convolution by $D^N$), as a result of which it does not verify a maximum principle; as it is quadratic (under the effect of $H$), it may only be solvable in small time for $N$ small; (ii) the forward equation (for $\mu^N$) is also non-local, and, for this reason, its solution may `leave' the space of probability measures.
  
The practical remedies that we suggest to deal with these two problems are inspired by the theoretical analysis carried out in the following sections. In short, the principle is to modify the structure of the two equations of the system 
\eqref{eqn: FWBKW_intro} so that the new system is well solvable for any $N$ (including $N$ small) and 
its solution coincides with that of the original version of 
\eqref{eqn: FWBKW_intro} for $N$ large. 
As regards the backward equation (i.e., the equation for $u^N$), one possibility is to modify the 
Hamiltonian for large values of momentum so that it becomes Lipschitz; equivalently, 
one can restrict oneself to controls that are bounded by some constant $\kappa$. 
This approach is used in Subsection 
\ref{sec: nonlocalHJ} to treat the backward equation in 
\eqref{eqn: FWBKW_intro}, but from a theoretical point of view. 
As far as the forward equation is concerned, one possibility is to force, for example by a reflection argument, 
$\mu^N$ to remain a probability measure. Although the notion of reflection may not be clear 
when dealing with general elements of ${\mathcal P}({\mathbb T}^d)$ (which is 
an infinite dimensional convex set), it becomes clearer once elements have been discretized 
with their discretizations living in a finite-dimensional space. 
We do not discuss this idea in its full scope, but we give an 
overview of it in Subsection 
\ref{sec: Pontryagin}. 

In this way, we can reasonably expect to be able to construct an algorithm that returns an output for any $N$, and not just for $N$ large. Of course, the key is then to decide whether this output should be considered as a solution 
(or almost) of the original version of \eqref{eqn: FWBKW_intro} or simply 
of its modified version.  
In other words, when can we be sure of no longer seeing the impact of our two remedies, truncation of the Hamiltonian on the one hand and  reflection of ${\mu}^N$
on the other? In absence of an explicit 
 (or at least numerically computable)
value
for the threshold $\widetilde N$ beyond
which  Meta-Theorems
 \ref{meta-thm:1}, \ref{meta-thm:2}, 
\ref{mthm: conv_value}
 and 
\ref{mthm: approf_four} 
apply,  the answer can only be empirical. Precisely, 
Meta-Theorems
 \ref{meta-thm:1}, \ref{meta-thm:2}, 
\ref{mthm: conv_value}
 and 
\ref{mthm: approf_four} 
say that
for 
large values of $\kappa$ and 
$N$, the numerical 
solutions should not be impacted by the truncation and reflection procedures: 
in words, $ \nabla u^N$ (or at least the `numerical' gradient) should be strictly bounded by $\kappa$
and  $\mu^N$ (or at least the `numerical' trajectory) should stay inside the `interior' of the 
space of probability measures. So, numerically, we should be satisfied with the choices of 
$\kappa$ and $N$ if we observe the latter two phenomena. 
\end{remark}}

\section{The approximated forward-backward system}\label{sec: approx_fwdbkw_problem}
\label{se:3}
Now that we have formulated the approximating problem, we can move on to its analysis.
 In particular, the main objective is to characterize and determine the optimal controls and the associated optimal trajectories, and to detail their regularity. Nevertheless, \addtxtr{we must first clarify}  the meaning of the problem \eqref{eqn: value_N} itself. As already pointed out in Remark \ref{rmk: notsoeasy}, the difficulty is to prove that $\mu^N$ remains a probability measure, which is all the more demanding at this stage as we have no \textit{a priori} information on the 
 properties of the optimal control (of the approximating problem). 
Proving that our approximation preserves the space ${\mathcal P}({\mathbb T}^d)$ is one of the main challenges we have to face, and most of the effort in this section is devoted to it.\\
 
Let us assume for a while that the solution $\mu^N$ of the approximated Fokker-Planck equation \eqref{eqn: fp_n_OC}, at least when it is driven by a control $\alpha$ that is candidate for being optimal, 
is a flow of probability measures and similarly for the truncation $\mu^N*D^N$. We then conjecture (see Section \ref{sec: Pontryagin} for more about this) that the adjoint equation in the Pontryagin method is the following non-local Hamilton-Jacobi equation: on $[0,T]\times\T^d$,
\begin{equation}\tag{$HJ_N$}\label{eqn: nonlocal_HJ_tilde}
	\begin{cases}
		\partial_ t  u^N_t(x) \!\!\!&=- \Delta  u^N_t(x)  + H( \cdot,\nabla  u^N_t)*D^N (x) - \lf\cF(\mu^N_t*D^N,\cdot)*D^N(x),\\
		\quad\\
		 u^N_T(x)\!\!\!&= \lf\cG(\mu^N_T*D^N,\cdot)*D^N(x),
	\end{cases}
\end{equation}
where $H(\cdot, \nabla u^N_t)*D^N (x)$ denotes the (Dirichlet) truncation of the mapping $x\mapsto H(x, \nabla u^N_t(x))$ and  $\lf\cF(\mu^N_t*D^N,\cdot)*D^N(x)$ denotes the (Dirichlet) truncation of $x\mapsto \lf\cF(\mu^N_t*D^N,x)$ (and the same for $\lf\cG$). Accordingly, our guess is that $\alpha^*(x) = - \diff_p H(x,\nabla u^N_t(x))$ is the optimal strategy. In turn, the optimal trajectory should be:
\begin{equation}\tag{$FP_N$}\label{eqn: opt_fp_prox}
	\begin{cases}
		\partial_t \mu^N_t &= \Delta\mu^N_t + \div(\diff_p H(\cdot,\nabla u^N_t)(\mu^N_t*D^N)),\quad t\in(0,T],\\
		\quad\\
		\mu^N_0 &= m.
	\end{cases}
\end{equation} 
Thus, we have to solve two interconnected problems: we need to establish the well-posedness of the system \eqref{eqn: nonlocal_HJ_tilde} - \eqref{eqn: opt_fp_prox}, and we have to prove that $\mu^N$ and $\mu^N*D^N$ are flows of probability measures (at least under suitable conditions). We describe here briefly our approach.

First, in Subsection \ref{sec: nonlocalHJ}, we fix a flow of probability measures $\nu\colon[0,T]\to\prob_N(\T^d)$ and we study a version of \eqref{eqn: nonlocal_HJ_tilde} where $\nu$ is substituted for $\mu^N*D^N$. We prove the existence of a solution $ u^{N,\nu}$ to the aforementioned equation, and we obtain estimates on $ u^{N,\nu}$ that are independent of $\nu$.
In Subsection \ref{sec: prop_of_mutilde}, we then consider the approximating Fokker-Planck equation \eqref{eqn: opt_fp_prox} with $\nabla  u^{N,\nu}$ in place of $\nabla u^N$. Within this context, we show that the solution $\mu^{N,\nu}$ and its truncation are probability-measure valued. Next, in Subsection \ref{sec: fixed_point}, we combine these results with a fixed-point argument to establish well-posedness of the system \eqref{eqn: nonlocal_HJ_tilde} - \eqref{eqn: opt_fp_prox}.
Lastly, in Subsection \ref{sec: Pontryagin}, we apply a Pontryagin-type argument (the sufficient condition of it)  to conclude that the solution 
 \eqref{eqn: nonlocal_HJ_tilde} - \eqref{eqn: opt_fp_prox} is indeed the (unique) optimal control of the approximating MFC problem.
 
\subsection{The non-local HJ equation for a fixed flow in $\prob_N(\T^d)$}\label{sec: nonlocalHJ} 
Let us fix $N\in\N\setminus\{0\}$ and a flow of probability measures $\nu\colon[0,T]\to\prob_N(\T^d)$ with null Fourier coefficients of order higher than $N$, i.e., $\widehat{\nu}_t(k)=0$ for $\vert k \vert > N$. Our aim is to study the following non-local Hamilton-Jacobi equation: on $[0,T]\times\T^d$,
\begin{equation}\tag{$HJ_{N,\nu}$}\label{eqn: nonlocal_HJ_tilde_nu}
	\begin{cases}
		\partial_ t  u^{N,\nu}_t(x) &=- \Delta  u^{N,\nu}_t(x)  + H( \cdot,\nabla  u^{N,\nu}_t)*D^N (x) - \lf\cF(\nu_t,\cdot)*D^N(x),\\
		\quad\\
		 u^{N,\nu}_T(x)&= \lf\cG(\nu_T,\cdot)*D^N(x).
	\end{cases}
\end{equation}
A point to clarify is the existence and uniqueness of a solution to \eqref{eqn: nonlocal_HJ_tilde_nu}. For this purpose, we focus on studying the Fourier coefficients of $u^{N,\nu}$. A first consideration is the fact that any solution to the backward equation \eqref{eqn: nonlocal_HJ_tilde_nu} has a finite number of non-zero Fourier coefficients. Indeed, due to the presence of the Dirichlet kernels, for $k\in\Z^d$ such that $\lvert k\rvert>N$ it holds
\begin{equation*}
	\begin{cases}
		\medskip
		\frac{\de}{\de t} \widehat{u}^{N,\nu}_t(k)  &= \lvert2\pi k\rvert^2 \widehat{ u}^{N,\nu}_t(k),\quad t\in[0,T],\\
		\widehat{ u}^{N,\nu}_T(k) &= 0,
	\end{cases}
\end{equation*}
and so $\widehat{ u}^{N,\nu}(k)$ is identically equal to $0$ for $\lvert k \rvert>N$. By recalling that $ \reallywidehat{\nabla u^{N,\nu}}(k) = \ii 2\pi k \widehat{ u}^{N,\nu}(k)$, we can deduce that also $\nabla u^{N,\nu}$ has null Fourier coefficients of order $\lvert k\rvert >N$. On the other hand, the Fourier coefficients of $ u^{N,\nu}$ (and, consequently, of $\nabla  u^{N,\nu}$) of order $\lvert k\rvert\leq N$ are determined by the following system of nonlinear ODEs
\begin{equation}\label{eqn: fourier_utilden_nu}
\begin{cases}
	\frac{\de}{\de t} \widehat{ u}^{N,\nu}_t(k) &= \lvert2\pi k\rvert^2  \widehat{ u}^{N,\nu}_t(k) - \reallywidehat{ \lf\cF(\nu_t,\cdot)}(k) + \reallywidehat{H(\cdot,\nabla  u^{N,\nu}_t)}(k),\quad\lvert k\rvert\leq N,\\
	\quad\\
	\widehat{  u}^{N,\nu}_T (k)& = \reallywidehat{ \lf\cG(\nu_T,\cdot)}(k).
\end{cases}
\end{equation}
Thus, if there exists a unique solution to the system \eqref{eqn: fourier_utilden_nu} given by the sequence $\{\widehat{u}^{N,\nu}(k)\}_{\lvert k\rvert\leq N}$, then we will be able to uniquely determine $u^{N,\nu}$ through its Fourier series $u^{N,\nu} = \sum_{\lvert k\rvert\leq N}\widehat{u}^{N,\nu}(k)e_k$.
\begin{proposition}\label{prop: ex_uniq_u_nu}
	Let Assumption \ref{hp: 1+2} hold . Then, there exists a unique \addtxtr{maximal} solution $\{\widehat{u}^{N,\nu}(k)\}_{\lvert k\rvert\leq N}$ to the system \eqref{eqn: fourier_utilden_nu}. In particular, for any fixed $N\in\N\setminus\{0\}$ and $\nu\colon[0,T]\to\prob_N(\T^d)$, there exists a unique \addtxtr{maximal} classical solution to \eqref{eqn: nonlocal_HJ_tilde_nu}.
\end{proposition} 

\addtxtr{The notion of maximal solution is that of the (local) Cauchy-Lipschitz theorem.} 

\begin{proof}
	Let us recall that $\nabla u^{N,\nu} = \ii 2\pi\sum_{\lvert j\rvert\leq N} j \hat{u}^{N,\nu}(j)e_j$. So, we can write \eqref{eqn: fourier_utilden_nu} as:
	\begin{equation}\label{eqn: sys_ode_un_nu}
		\frac{\de}{\de t} \widehat{ u}^{N,\nu}_t(k) = \lvert2\pi k\rvert^2  \widehat{ u}^{N,\nu}_t(k) - \reallywidehat{ \lf\cF(\nu_t,\cdot)}(k) + f_k\Bigl(\left\{ \widehat{ u}^{N,\nu}_t(j)\right\}_{\lvert j\rvert\leq N}\Bigr),\quad\lvert k\rvert\leq N,
	\end{equation}
	where $f_k\colon\C^{N^d}\to\C$ is defined for any sequence $\mathbf{z}=\{z_j\}_{\lvert j\rvert\leq N}$, indexed by $j\in\Z^d$ and such that $\bar{z}_j = z_{-j}$, through the formula
	\begin{equation*}
		f_k(\mathbf{z}) = \reallywidehat{H\bigg(\cdot, \ii 2\pi\sum_{\lvert j\rvert\leq N} j z_je_j\bigg)}(k) = \int_{\T^d}H \bigg(x,\ii 2\pi\sum_{\lvert j\rvert\leq N} j z_je_j(x)\bigg)\addtxtr{e_{-k}}(x)\de x.
	\end{equation*}
	\addtxtr{Using the local Lipschitz  property of $H$ in the variable $p$, we can easily prove that 
	$f_k$ is locally Lipschitz in the variable ${\mathbf{z}}$. As a result, local existence and uniqueness for  \eqref{eqn: sys_ode_un_nu} follow from the usual (local) Cauchy-Lipschitz theory for ODEs.}
\end{proof}

In the next subsections, we
\addtxtr{prove that, for $N$ sufficiently large, 
$u^{N,\nu}$ is in fact a global solution over the entire interval $[0,T]$ (in which case, it is also globally unique) and satisfies suitable regularity properties.}
In particular, part of our effort
consists in obtaining  bounds that are independent of $\nu$ and $N$ \addtxtr{(at least for $N$ large)}.

\subsubsection{The auxiliary HJ equations}\label{sec: aux_problem}
The main difficulty we have to face in the study of \eqref{eqn: nonlocal_HJ_tilde_nu} (and in general of \eqref{eqn: nonlocal_HJ_tilde}) comes from the non-linear term $H(\cdot,\nabla u^{N,\nu})*D^N$. As the latter is also non-local, we cannot apply the usual approach for Hamilton-Jacobi (HJ) equations.

We proceed as follows.
 First, we introduce an auxiliary HJ equation in which the Hamiltonian $H$ is replaced by a 
new one, denoted by $\widetilde{H}$, whose properties (clarified later in this section) facilitate the analysis of the corresponding HJ equation. We refer to the solution of this auxiliary (non-local) HJ equation as $ \tilde{u}^{N,\nu}$. Subsequently, we establish certain (uniform in $\nu$) estimates for both $ \tilde{u}^{N,\nu}$ and $\nabla \tilde{u}^{N,\nu}$. To achieve this, it becomes necessary to introduce a secondary auxiliary HJ equation, which is
in fact the (usual) HJ equation arising in the mean field game system associated with $(\delta_\mu {\mathcal F}, 
\delta_\mu {\mathcal G})$ and 
 which is thus well-understood. Thereafter, we return from the secondary to the primary auxiliary HJ equations and infer 
the desired regularity properties on  $\tilde{u}^{N,\nu}$. Finally, we show that we can tune $\widetilde{H}$ appropriately in order to transfer all the results to the initial (non-local) HJ equation \eqref{eqn: nonlocal_HJ_tilde_nu} and  its solution $u^{N,\nu}$.\\

\noindent We first explain the constraints we put on the new Hamiltonian $\widetilde{H}$. 
For an arbitrary $M>0$ (we choose the value later, see Remark \ref{rmk: H_is_useless}), we are considering 
a continuously differentiable function $\widetilde{H}\colon \T^d\times\addtxtr{\R^d}\to\R$ such that:
\begin{enumerate}[i.]
	\item There exists a positive constant $C = C(M)$, independent of $x$, such that for any $p \in\R^d$, $\big\lvert\widetilde{H}(x,p)\big\rvert\leq C(M)(\lvert p\rvert+1)$, 
	$\big\lvert \diff_x \widetilde{H}(x,p)\big\rvert\leq C(M)$ and $\big\lvert \diff_p  \widetilde{H}(x,p) \big\rvert\leq C(M)$;
	\smallskip
	\item For any $x\in\T^d$ and $p\in\R^d$ such that $\lvert \diff_p H(x,p)\rvert\leq M$, it holds $\widetilde{H}(x,p) = H(x,p)$.
\end{enumerate}
\addtxtr{Such an $\widetilde H$ can be easily constructed by multiplication of $H$ by a cut-off function. 
Importantly, we do not need $\widetilde H$ to be convex. Indeed, 
in the computations where the specific structure of $\widetilde H$ plays a role, 
we do not use any argument based on a stochastic control interpretation of the approximating HJ equation. 
Instead, we always invoke standard estimates from the literature on non-degenerate semi-linear PDEs
(in which Hamiltonians can be quite general). These estimates allow us to prove that, for $N$ large enough, the solution
to the approximating HJ equation driven by $\widetilde H$ actually lives in the domain where 
$\widetilde H$ and $H$ coincide. This makes it possible to use convexity again for the rest of the analysis.} 

\vskip 6pt

\noindent We can now state our primary auxiliary HJ equation: on $[0,T]\times\T^d$, 
\begin{equation}\tag{$\widetilde{HJ}_{N,\nu}$}\label{eqn: dual_fpn_H}
	\begin{cases}
		\partial_ t \tilde{u}^{N,\nu}_t(x) &= - \Delta \tilde{u}^{N,\nu}_t(x)  + \widetilde{H}(\cdot, \nabla \tilde{u}^{N,\nu}_t)*D^N (x) - \lf\cF(\nu_t,\cdot)*D^N(x),\\
		\quad\\
		\tilde{u}^{N,\nu}_T(x) &= \lf\cG(\nu_T,\cdot)*D^N(x).
	\end{cases}
\end{equation}
For the moment, we focus on \eqref{eqn: dual_fpn_H} (and forget temporarily \eqref{eqn: nonlocal_HJ_tilde_nu}). 
Only at the end of our analysis, we prove that, under certain conditions, the two equations actually coincide.
\vskip 6pt

\noindent In order to deal with \eqref{eqn: dual_fpn_H}, let us introduce a secondary auxiliary HJ equation:
\begin{equation}\tag{$\overline{HJ}_{\nu}$}\label{eqn: dual_fpn_H_auxiliary}
	\begin{cases}
		\partial_ t v^{\nu}_t(x)  &= - \Delta v^{\nu}_t(x)  + H(x, \nabla v^{\nu}_t(x)) - \lf\cF(\nu_t, x)  \quad \text{ in } [0,T]\times\T^d,\\
		\quad\\
		v^{\nu}_T(x) &= \lf\cG(\nu_T, x)\quad \text{ in } \T^d.
	\end{cases}
\end{equation}
Notice that the Hamiltonian  in \eqref{eqn: dual_fpn_H_auxiliary}  is not the new $\widetilde{H}$, but the original  $H$.
In particular, 
\eqref{eqn: dual_fpn_H_auxiliary} matches the usual HJ equation arising in mean field game/control theory
(see the backward equation in 
\eqref{eqn: fwdbkw_infty_intro}). 
In fact, we will show in Section \ref{sec: est_aux_syst} 
how to pass from the equation driven by $\widetilde H$ to the equation driven by ${H}$ and conversely, by
choosing $M$ appropriately. 

\subsubsection{Estimates on $v^{\nu}$}\label{sec: est_aux_syst}
We study the smoothness of $v^{\nu}$ by using standard techniques for parabolic equations. We first notice 
from the classical theory of Hamilton-Jacobi equations (see, e.g.,~\cite[Theorem V.6.1]{ladyzanskajasolonnikovuraltseva})
that there exists a unique classical solution $v^{\nu}\in\rmC^{1,2}([0,T]\times\T^d)$ to \eqref{eqn: dual_fpn_H_auxiliary}. In the next proposition, we recall some (standard) regularity estimates on $v^{\nu}$. Even though these estimates are already known, we give a sketch of the proof for completeness.
\begin{proposition}\label{prop: bound_aux_eq_tilde}
	Under Assumption \ref{hp: 1+2}, 
	 $v^{\nu}\in  \rmC^{1,{q}}([0,T]\times\T^d)$. Moreover, there exists a positive constant $C = C(d,T,\cF,\cG,H,q)$, independent of $M$ and $\nu$, such that
	\begin{equation*}
		\sup_{t\in[0,T]}\norm{v^{\nu}_t}_{\infty,q} 
		\leq C(d,T,\cF,\cG,H,q).
	\end{equation*}
\end{proposition}
\begin{proof}
	By a classical result on nonlinear parabolic equations (see e.g.~ \cite[Theorem V.3.1]{ladyzanskajasolonnikovuraltseva}), there exists a constant $C = C(d,T,\cF,\cG,H)$ such that $\sup_{t\in[0,T]}\norm{\nabla v^{N,\nu}_t}_\infty\leq C(d,T,\cF,\cG,H)$. In particular,  the bound is independent of $\nu$. This, together with the continuity of $H$, entails the (uniform in time) boundedness of the mapping $x\mapsto H(x,\nabla v^{\nu}(x))$. Denoting by $p = p_t(x)$ the usual heat kernel on the $d$-dimensional torus, we have from 
	Duhamel's formula:
	\begin{align*}
		\lvert v^{\nu}_t(x)\rvert &\leq C(d,T, \cF, \cG,H) +\left\lvert \int_t^T \int_{\T^d} p_{s-t}(x-y) H\left(y,\nabla v^{\nu}_s(y)\right)\de s \de y\right\rvert\\
		&\leq C(d,T,\cF,\cG,H)+ \int_t^T\norm{H\left(\cdot, \nabla v^{\nu}_s\right)}_\infty\de s
		\leq C(d,T,\cF,\cG,H).
	\end{align*}
	Here we exploited the uniform in $\mu$ bounds on $\lf\cF$ and $\lf\cG$ to control $\lf\cF(\nu_t,\cdot)$ and $\lf\cG(\nu_T,\cdot)$ in \eqref{eqn: dual_fpn_H_auxiliary}.
	Finally, regularity up to order $q$ (with uniform in time bounds) 
	can be obtained 
	by using Lemma A.3 in \cite{daudindelaruejackson}, thanks to 
	the smoothness of $\lf\cF$, $\lf\cG$, and $H$.
\end{proof}
\begin{remark}\label{rmk: H_is_useless_for_uinfty}
	Let us set $\overline{M}:=\max\{\sup_{t \in [0,T]} \norm{\diff_pH(\cdot,\nabla u^\infty_t)}_\infty,\addtxtr{\sup_{\nu}}\sup_{t \in [0,T]}\norm{\diff_pH(\cdot,\nabla v^\nu_t)}_\infty\}$. Notice 
	\addtxtr{from Propositions  \ref{prop: estimates_uinf}
	and
	\ref{prop: bound_aux_eq_tilde} together with the continuity of $\diff_pH$} that \addtxtr{$\overline{M}$ is} bounded by a constant $C(d,T,\cF,\cG,H,q)$.
	Consequently, \addtxtr{if $M$ in the definition of $\widetilde H$ is greater than $\overline{M}$, then it is guaranteed that} for any $x\in\T^d$ and $t\in[0,T]$,  $\widetilde{H}(x,\nabla u^\infty_t(x)) = H(x,\nabla u^\infty_t(x))$ and $\widetilde{H}(x,\nabla v^{\nu}_t(x)) = H(x,\nabla v^{\nu}_t(x))$. This ensures that both $v^{\nu}$ and $u^\infty$ remain solutions of \eqref{eqn: dual_fpn_H_auxiliary} and \eqref{eqn: fwdbkw_infty_intro}, respectively, with $\widetilde{H}$ replacing $H$. Furthermore, these solutions obviously retain the regularity properties outlined in Propositions \ref{prop: estimates_uinf} and \ref{prop: bound_aux_eq_tilde}.
\end{remark}
\begin{remark}
	In Propositions \ref{prop: estimates_uinf} and \ref{prop: bound_aux_eq_tilde}, we ask Assumpion \ref{hp: 1+2} to be in force and thus $q$ to be greater than or equal to $d+3$. This is to make the exposition simpler. In fact, the results hold for any choice of $q>2$, no matter the value of the dimension $d$.
\end{remark}

\subsubsection{Estimates on $\tilde{u}^{N,\nu}$}\label{sec: est_utildeN}
We now address
the primary auxiliary equation  \eqref{eqn: dual_fpn_H}. Due to the non-local term $\widetilde{H}(\cdot,\nabla\tilde{u}^{N,\nu})*D^N$, we cannot treat it by using Duhamel's formula as  simply as we did in the proof of 
Proposition 
\ref{prop: bound_aux_eq_tilde}. Instead, 
our strategy is to combine the regularity properties established in 
Proposition 
\ref{prop: bound_aux_eq_tilde}
with the convergence properties of $D^N$ in order to establish similar properties on $\tilde{u}^{N,\nu}$. A key role will be played by the Lipschitz property of $\widetilde{H}$ in the $p$ variable.
\vskip 6pt

\noindent 
\addtxtr{We observe that existence and uniqueness of a \addtxtr{global} classical solution $\tilde{u}^{N,\nu}$
(to  \eqref{eqn: dual_fpn_H}) can be obtained as in Proposition \ref{prop: ex_uniq_u_nu}, 
with  the difference that Cauchy-Lipschitz theorem now applies in its global version because 
$\widetilde H$ is globally Lipschitz.} Subsequently, we can concentrate ourselves on the regularity of  $\tilde{u}^{N,\nu}$. We first establish estimates in $L^2$ norm. They will be crucial to prove the main estimates of this subsection, which are stated in $\rmC^{\lfloor q -1 - d/2\rfloor}$ norm. 

We start with the following statement, \addtxtr{which is true whatever the value of $M$  in the definition of $\widetilde{H}$ (in particular there is no need to choose $M$  according to Remark \ref{rmk: H_is_useless_for_uinfty}).}
\begin{proposition}\label{prop: l2estimates_uN}
	Under Assumption \ref{hp: 1+2}, there exists a positive constant $C = C(d,T,\cF,\cG,H,M)$ such that
	\begin{equation*}
		\sup_{t\in[0,T]}(\norm{\tilde{u}^{N,\nu}_t}_2 + \norm{\nabla\tilde{u}^{N,\nu}_t}_2) \leq C(d,T,\cF,\cG,H,M).
	\end{equation*}
\end{proposition}
\begin{proof}	
	We integrate \eqref{eqn: dual_fpn_H} with respect to $\tilde{u}^{N,\nu}$. By combining integration by parts and Lemma \ref{lemma: trivial_est_trunc}, we have
	\begin{align*}
		\frac{1}{2}\norm{\tilde{u}^{N,\nu}_t}_2^2 &+ \int_t^T\norm{\nabla\tilde{u}^{N,\nu}_s}^2_2\de s \\
		&\leq \frac{1}{2}\norm{\lf\cG(\nu_T,\cdot)}_2^2 
		+ \int_t^T \lvert\scalprod{\lf\cF(\nu_s,\cdot)*D^N}{\tilde{u}^{N,\nu}_s}\rvert\de s\\&\quad +\int_t^T \lvert\scalprod{\widetilde{H}(\cdot,\nabla\tilde{u}^{N,\nu}_s)*D^N}{\tilde{u}^{N,\nu}_s}\rvert\de s\\
		&\leq C(d, T) (\norm{\lf\cG}^2_\infty + \norm{\lf\cF}^2_\infty) + \int_t^T \norm{\tilde{u}^{N,\nu}_s}_2^2\de s
		  \\
		 &\quad + \frac{C(M)^2}{2}\int_t^T\norm{\tilde{u}^{N,\nu}_s}_2^2\de s+\frac{1}{2C(M)^2}\int_t^T\norm{\widetilde{H}(\cdot,\nabla\tilde{u}^{N,\nu}_s)*D^N}_2^2\de s\\
		&\leq C(d, T,\cF,\cG,H,M) + \int_t^T \norm{\tilde{u}^{N,\nu}_s}_2^2\de s
		 + \frac{1}{2}\int_t^T\norm{\nabla\tilde{u}^{N,\nu}_s}_2^2\de s + \frac{C(M)^2}{2}\int_t^T\norm{\tilde{u}^{N,\nu}_s}_2^2\de s,
	\end{align*}
	where we used $\lvert\widetilde{H}(x,p)\rvert\leq C(M)(\lvert p \rvert+1)$ 
	in the penultimate line. Gronwall's inequality leads to $\norm{\tilde{u}^{N,\nu}_t}_2\leq C(d,T,\cF,\cG,H,M)$. 
	
	 To estimate  $\nabla\tilde{u}^{N,\nu}$, we can take the $x$-derivative of  \eqref{eqn: dual_fpn_H}, and then integrate the differentiated equation against $\nabla\tilde{u}^{N,\nu}$. We obtain
	\begin{align*}
		\frac{1}{2}\norm{\nabla\tilde{u}^{N,\nu}_t}_2^2 &+ \int_t^T\norm{\nabla^2\tilde{u}^{N,\nu}_s}^2_2\de s 
		\leq \frac{1}{2}\norm{\nabla [\lf\cG](\nu_T,\cdot)}_2^2 \\
		&\quad+ \int_t^T \lvert\scalprod{\nabla [\lf\cF](\nu_s,\cdot)*D^N}{\nabla\tilde{u}^{N,\nu}_s}\rvert\de s +\int_t^T \lvert\scalprod{\nabla[\widetilde{H}(\cdot,\nabla\tilde{u}^{N,\nu}_s)*D^N]}{\nabla\tilde{u}^{N,\nu}_s}\rvert\de s\\
		&\leq C(d, T)(\norm{\nabla [\lf\cG]}_\infty^2 + \norm{\lf\cF}_\infty^2)
		\\&\quad + \int_t^T \norm{\nabla^2\tilde{u}^{N,\nu}_s }^2_2\de s
		+C(M)\int_t^T \norm{\nabla^2\tilde{u}^{N,\nu}_s}_2\norm{\nabla\tilde{u}^{N,\nu}_s}_2\de s,
	\end{align*}
	and the claim is obtained as for $\tilde{u}^{N,\nu}$. 
\end{proof}
From now on, let us consider $M\geq\overline{M}$, where $\overline{M}$ is chosen as in Remark \ref{rmk: H_is_useless_for_uinfty}. Thus we can \addtxtr{replace $\widetilde H$ by $H$ in 
 \eqref{eqn: dual_fpn_H_auxiliary}}. We will refine again the requirements on $M$ later on in this section.
\begin{lemma}\label{lemma: conv_l2_gradu_utilde}
	Under Assumption \ref{hp: 1+2}, there exists a positive constant $C = C(d, T, \cF, \cG, H, M)$ such that 
	\begin{equation*}
		\sup_{t\in[0,T]}\norm{\nabla v^{\nu}_t - \nabla\tilde{u}^{N,\nu}_t}_2\leq \frac{C(d,T,\cF,\cG,H,M)}{N^{q-1}}.
	\end{equation*}
\end{lemma}
\begin{proof}
	From \eqref{eqn: dual_fpn_H} and \eqref{eqn: dual_fpn_H_auxiliary}, we have
	\begin{align*}
		\frac{1}{2}\norm{\nabla v^{\nu}_t - \nabla \tilde{u}^{N,\nu}_t}^2_2& + \int_t^T \norm{\nabla^2v^{\nu}_s - \nabla^2\tilde{u}^{N,\nu}_s}^2_2\de s\\
		&\leq \frac{1}{2}\norm{\nabla[ \lf\cG](\nu_T,\cdot) - \nabla[ \lf\cG](\nu_T,\cdot)*D^N}^2_2\\ 
		&\quad+\int_t^T \lvert \scalprod  {\lf\cF(\nu_s,\cdot) - \lf\cF(\nu_s,\cdot)*D^N}{\nabla^2 v^{\nu}_s - \nabla^2 \tilde{u}^{N,\nu}_s}\rvert\de s\\
		&\quad+\int_t^T \lvert \scalprod{\widetilde{H}(\cdot,\nabla v^{\nu}_s) - \widetilde{H}(\cdot,\nabla\tilde{u}^{N,\nu}_s)*D^N}{\nabla^2 v^{\nu}_s - \nabla^2 \tilde{u}^{N,\nu}_s}\rvert\de s.
	\end{align*}
	Lemma \ref{lemma: estl2truncation} allows us to control the $L^2$ norms of $(\nabla \lf\cG -\nabla \lf\cG*D^N)$ and $(\nabla \lf\cF -\nabla \lf\cF*D^N)$. Indeed, for every $\mu\in\prob(\T^d)$ we have that $\lf\cF(\mu,\cdot),\lf\cG(\mu,\cdot)\in H^{{q}}(\T^d)$, with Sobolev norms that are independent of $\mu$. By combining this with Cauchy-Schwarz inequality, it follows
	\begin{equation}\label{eqn: ineq_est_diff_grad_l2_1}
	\begin{aligned}
		\frac{1}{2}\norm{\nabla v^{\nu}_t& - \nabla \tilde{u}^{N,\nu}_t}^2_2 + \int_t^T \norm{\nabla^2v^{\nu}_s - \nabla^2\tilde{u}^{N,\nu}_s}^2_2\de s\\
		&\leq C(d,\cF,\cG)\left(\frac{1}{N^{{2q-2}}} +\frac{1}{N^{{2q}}}\right) + \frac{1}{2} \int_t^T \norm{\nabla^2v^{\nu}_s - \nabla^2\tilde{u}^{N,\nu}_s}^2_2\de s\\
		&\quad+\int_t^T \norm{\widetilde{H}(\cdot,\nabla v^{\nu}_s) - \widetilde{H}(\cdot,\nabla v^{\nu}_s)*D^N}_2\norm{\nabla^2 v^{\nu}_s - \nabla^2 \tilde{u}^{N,\nu}_s}_2\de s\\
		&\quad+ \int_t^T \norm{\widetilde{H}(\cdot,\nabla v^{\nu}_s)*D^N - \widetilde{H}(\cdot,\nabla \tilde{u}^{N,\nu}_s)*D^N}_2\norm{\nabla^2 v^{\nu}_s - \nabla^2 \tilde{u}^{N,\nu}_s}_2\de s. 
	\end{aligned}
	\end{equation}
	Now, regarding the term $\norm{\widetilde{H}(\cdot,\nabla v^{\nu}_s)*D^N - \widetilde{H}(\cdot,\nabla \tilde{u}^{N,\nu}_s)*D^N}_2$, we can first remove $D^N$ (convolution by $D^N$ is a contraction in $L^2$, see Lemma \ref{lemma: trivial_est_trunc}) and then use the Lipschitz property of $\widetilde{H}$ to obtain
	\begin{equation*}
		\norm{\widetilde{H}(\cdot,\nabla v^{\nu}_s)*D^N - \widetilde{H}(\cdot,\nabla \tilde{u}^{N,\nu}_s)*D^N}_2\leq C(M)\norm{\nabla v^{\nu}_s - \nabla \tilde{u}^{N,\nu}_s}_2.
	\end{equation*}
	To control $\norm{\widetilde{H}(\cdot,\nabla v^{\nu}_s) - \widetilde{H}(\cdot,\nabla v^{\nu}_s)*D^N}_2$, let us first show that the map
	$x\mapsto\widetilde{H}(x,\nabla v^{\nu}(x))$ is of class $\rmC^{{q-1}}(\T^d)$, with a norm bounded by a positive constant $C = C(d,T,\cF,\cG,H)$. Indeed, since $M\geq\overline{M}$ by assumption (see Remark \ref{rmk: H_is_useless_for_uinfty} for the definition of $\overline{M}$), we have $\widetilde{H}(x,\nabla v^{\nu}(x)) = H(x,\nabla v^{\nu}(x))$. Thus, the claim follows by recalling that $H\in\rmC^{{q}}(\T^d\times\R^d)$ together with Proposition \ref{prop: bound_aux_eq_tilde}. Finally, by Lemma \ref{lemma: estl2truncation} it follows that 
	\begin{equation*}
		\norm{\widetilde{H}(\cdot,\nabla v^{\nu}_s) - \widetilde{H}(\cdot,\nabla v^{\nu}_s)*D^N}_2\leq\frac{C(d,T,\cF,\cG,H)}{N^{q-1}}.
	\end{equation*}
	By plugging these estimates in \eqref{eqn: ineq_est_diff_grad_l2_1} we obtain
	\begin{align*}
		\frac{1}{2}\norm{\nabla v^{\nu}_t - \nabla \tilde{u}^{N,\nu}_t}^2_2& + \int_t^T \norm{\nabla^2v^{\nu}_s - \nabla^2\tilde{u}^{N,\nu}_s}^2_2\de s\\
		 &\leq C(d,\cF,\cG)\frac{1}{N^{{2q-2}}} + \frac{1}{2} \int_t^T \norm{\nabla^2v^{\nu}_s - \nabla^2\tilde{u}^{N,\nu}_s}^2_2\de s\\
		 &\quad+\int_t^T \norm{\widetilde{H}(\cdot,\nabla v^{\nu}_s) - \widetilde{H}(\cdot,\nabla v^{\nu}_s)*D^N}_2\norm{\nabla^2 v^{\nu}_s - \nabla^2 \tilde{u}^{N,\nu}_s}_2\de s\\
		&\quad+ C(M)\int_t^T \norm{\nabla v^{\nu}_s - \nabla \tilde{u}^{N,\nu}_s}_2\norm{\nabla^2 v^{\nu}_s - \nabla^2 \tilde{u}^{N,\nu}_s}_2\de s\\
		&\leq C(d,T,\cF,\cG,H,M)\left(\frac{1}{N^{{2q-2}}} + \int_t^T \norm{\nabla v^{\nu}_s - \nabla \tilde{u}^{N,\nu}_s}^2_2\de s\right) \\
		&\quad+ \int_t^T \norm{\nabla^2v^{\nu}_s - \nabla^2\tilde{u}^{N,\nu}_s}^2_2\de s,
	\end{align*}
	and so the result follows by Gronwall's inequality. 
\end{proof}
\noindent We now have all the ingredients to bound $\tilde{u}^{N,\nu}$ and $\nabla\tilde{u}^{N,\nu}$ in $L^\infty$ norm:
\begin{proposition}\label{prop: infinitybounds_tilde}
	Let Assumption \ref{hp: 1+2} be in force. For any $N\in\N\setminus\{0\}$ it holds:
	\begin{equation*}
		 \sup_{t\in[0,T]} ( \norm{\tilde{u}^{N,\nu}_t}_\infty + \norm{\nabla\tilde{u}^{N,\nu}_t}_\infty) \leq C(d,T,\cF,\cG,H,M,\addtxtr{q}),
	\end{equation*}
	for a positive constant $C = C(d,T,\cF,\cG,H,M,q)$.
	\addtxtr{And, there exist a positive constant 
	$\widetilde{M} = \widetilde{M}(d,T,\cF,\cG,H,q)$ 
	and an integer $\widetilde{N} = \widetilde{N}(d,T,\cF,\cG,H,q)$ such that, for 
	$M = \widetilde{M}$ and $N \geq \widetilde{N}$, the constant $C$ is independent of 
	$M$.} 
	\addtxtr{Moreover, again} for $M=\widetilde{M}$ and   $N\geq\widetilde N$, 
	{ $\tilde{u}^{N,\nu}\in\rmC^{1, q}([0,T]\times\T^d)$ and }
	\begin{equation*}
		\sup_{t\in[0,T]}\norm{\tilde{u}^{N,\nu}_t}_{\infty,\lfloor q-1-d/2 \rfloor} 
		\leq C(d,T,\cF,\cG,H,q),
	\end{equation*}
	for a positive constant $C = C(d,T,\cF,\cG,H,q)$.
\end{proposition}
	\begin{remark}\label{rmk: q_not_2q}
		We emphasize that, differently from $u^\infty$ and $v^\nu$ in
		Propositions \ref{prop: estimates_uinf} and \ref{prop: bound_aux_eq_tilde}, 
		$\tilde{u}^{N,\nu}$ cannot be estimated in $C^q$ independently of $N$. Bounds are uniform in $N$ up to the regularity order $\lfloor q-1-d/2\rfloor$. This is due to the presence of the non-local term in \eqref{eqn: dual_fpn_H}.
	\end{remark}
\begin{proof}
		\textit{First Step.} 
	We first establish the two bounds on $\tilde{u}^{N,\nu}$ and its gradient.
	 With $p = p_t(x)$ denoting the usual heat kernel on $\T^d$, 
	we deduce from Duhamel's representation formula that, for every $(t,x)\in [0,T)\times\T^d$,
	\begin{equation}\label{eqn: duhamel}
	\begin{aligned}
		\tilde{u}^{N,\nu}_t(x) &= \int_{\T^d}p_{T-t}(x-y)\left [\lf\cG(\nu_T,\cdot)*D^N\right](y)\de y \\
		&\quad - \int_t^T\int_{\T^d}p_{s-t}(x-y) \left[\widetilde{H}(\cdot,\nabla \tilde{u}^{N,\nu}_s)*D^N\right](y)\de y \de s\\
		&\quad + \int_t^T\int_{\T^d}p_{s-t}(x-y)\left[ \lf\cF(\nu_s,\cdot)*D^N\right](y)\de y \de s.
	\end{aligned}
	\end{equation}
	Then we have
	\begin{align}
		\bigg\lvert \int_t^T\int_{\T^d}p_{s-t}&(x-y) \left[\widetilde{H}(\cdot,\nabla \tilde{u}^{N,\nu}_s)*D^N\right](y)\de y \de s\bigg\rvert\label{eqn: inftybound_0th}\\
		&\leq \int_t^T\int_{\T^d} \lvert p_{s-t}(x-y)\rvert \left\lvert \left(\widetilde{H}(\cdot,\nabla\tilde{u}^{N,\nu}_s) - \widetilde{H}(\cdot,\nabla v^{\nu}_s)\right)*D^N(y)\right\rvert\de y \de s \label{eqn: inftybound_1st}\\
		& \quad+ \int_t^T\int_{\T^d} \lvert p_{s-t}(x-y)\rvert \left\lvert \widetilde{H}(\cdot, \nabla v^{\nu}_s)*D^N(y)\right\rvert\de y \de s.\label{eqn: inftybound_2nd}
	\end{align} 
	Regarding \eqref{eqn: inftybound_1st},  Cauchy-Schwarz inequality and Lemma \ref{lemma: trivial_est_trunc} yield, for any $\eps >0$, 
	\begin{align*}
		\eqref{eqn: inftybound_1st}&\leq \int_{t + \eps}^T \norm{p_{s-t}}_2 \norm{\widetilde{H}(\cdot,\nabla \tilde{u}^{N,\nu}_s) - \widetilde{H}(\cdot,\nabla v^{\nu}_s)}_2\de s\\
		&\quad  + \int_t^{t+\eps}\norm{(\widetilde{H}(\cdot,\nabla \tilde{u}^{N,\nu}_s) - \widetilde{H}(\cdot,\nabla v^{\nu}_s))*D^N}_\infty \de s\\
		&\leq \int_{t + \eps}^T \norm{p_{s-t}}_2 \norm{\widetilde{H}(\cdot,\nabla \tilde{u}^{N,\nu}_s) - \widetilde{H}(\cdot,\nabla v^{\nu}_s)}_2\de s\\
		&\quad + C(d)N^{\frac{d}{2}}\int_t^{t+\eps}\norm{\widetilde{H}(\cdot,\nabla \tilde{u}^{N,\nu}_s) - \widetilde{H}(\cdot,\nabla v^{\nu}_s)}_2 \de s.
	\end{align*}
	Then, we can combine Lemma \ref{lemma: conv_l2_gradu_utilde}, the Lipschitz property of $\widetilde{H}$ and the estimates for the heat kernel in Lemma \ref{lemma: est_heat_kernel} to get
	\begin{align*}
		\eqref{eqn: inftybound_1st} &\leq C(d,T,\cF,\cG,H,M)\left( \frac{1}{N^{{q-1}}}\int_{t+\eps}^T\frac{1}{(s-t)^{\frac{d}{4}}}\de s +\frac{\eps N^{\frac{d}{2}}}{N^{{q-1}}}\right)\\
		&\leq C(d,T,\cF,\cG,H,M)   \frac{\eps^{-\frac{d}{4} +1} + \eps N^{\frac{d}{2}}}{N^{{q-1}}}.
	\end{align*}
	 If we choose $\eps = N^{-2}$, we have
	\begin{equation*}\label{eqn: est_on_1st}
		\eqref{eqn: inftybound_1st}\leq \frac{C(d,T,\cF,\cG,M,H)}{N^{{q+1-\frac{d}{2}}}}.
	\end{equation*}
	We notice that the exponent in the denominator is positive because 
	$q\geq d+3$.

	For \eqref{eqn: inftybound_2nd}, we recall that the map $x\mapsto\widetilde{H}(x,\nabla v^{\nu}(x))$ is of class $\rmC^{{q-1}}(\T^d)$, with a norm uniformly bounded by a positive constant $C = C(d,T,\cF,\cG,H,q)$. This is a consequence of Proposition \ref{prop: bound_aux_eq_tilde} and the fact that $\widetilde{H}(x,\nabla v^{\nu}(x)) = H(x,\nabla v^{\nu}(x))$ for $M\geq\overline{M}$. Then, from Remark \ref{rmk: infinity_est_trunc},
	\begin{equation*}
		\eqref{eqn: inftybound_2nd}\leq C(d,T,q)\int_t^T\left(\norm{\nabla v^{\nu}_s}_\infty + \frac{ {\norm{\nabla v^{\nu}_s}_{\infty,q-1}}}{N^{q-1-\frac{d}{2}}}\right)\de s\leq C(d,T,\cF,\cG,H,q),
	\end{equation*}
	where, to get the last constant, we used once again the fact that  
	$q>\frac{d}{2}+1$. 

	Similarly, we have 
	\begin{equation*}
		\left\lvert \int_{\T^d}p_{T-t}(x-y)\left [\lf\cG(\nu_T,\cdot)*D^N\right](y)\de y \right\rvert
		\leq C(d,T,q)\left(\norm{\lf\cG}_\infty + \frac{{\norm{\lf\cG}_{\infty,q} }}{N^{q-\frac{d}{2}}}\right),
	\end{equation*}
	and
	\begin{equation*}
		\left\lvert \int_t^T\int_{\T^d}p_{s-t}(x-y)\left[ \lf\cF(\nu_s,\cdot)*D^N\right](y)\de y \de s\right\rvert
		\leq C(d,T,q)\left(\norm{\lf\cF}_\infty + \frac{{\norm{\lf\cF}_{\infty,q} }}{N^{q-\frac{d}{2}}}\right).
	\end{equation*}
	Finally, by combining all these estimates with the Duhamel representation formula \eqref{eqn: duhamel}, we obtain 
	\begin{equation}\label{eqn: est_u_dep_M}
		\sup_{t\in[0,T]}\norm{\tilde{u}^{N,\nu}_t}_\infty\leq\left(\frac{C_1(d,T,\cF,\cG,H,M)}{N^{\addtxtr{q  - \frac{d}{2}}}} + C_2(d,T,\cF,\cG,H,q) \right) \leq C(d,T,\cF,\cG,H,M,q). 
	\end{equation}
	Above, in \eqref{eqn: est_u_dep_M}, we introduced two different constants $\{C_i\}_{i=1,2}$ in order to better keep track of the influence of the various parameters.
	
\medskip

	\noindent For the gradient, we can proceed similarly, by differentiating the equation solved by $\tilde{u}^{N,\nu}$. Duhamel's formula, combined with an integration by parts, gives for any $(t,x)\in[0,T)\times\T^d$:
	\begin{equation}\label{eqn: duhamel_grad}
	\begin{aligned}
		\nabla \tilde{u}^{N,\nu}_t(x) &= \int_{\T^d}p_{T-t}(x-y)\left [\nabla[\lf\cG](\nu_T,\cdot)*D^N\right](y)\de y \\
		&\quad- \int_t^T\int_{\T^d}\nabla p_{s-t}(x-y) \left[\widetilde{H}(\cdot,\nabla \tilde{u}^{N,\nu}_s)*D^N\right](y)\de y \de s\\
		&\quad + \int_t^T\int_{\T^d}\nabla p_{s-t}(x-y)\left[ \lf\cF(\nu_s,\cdot)*D^N\right](y)\de y \de s.
	\end{aligned}
	\end{equation}
	We proceed as for $\tilde{u}^{N,\nu}$. First,
	following the derivation of 
	\eqref{eqn: inftybound_1st}, 
	we get for any $\eps>0$: 
	\begin{equation}
	\label{eq:prop:3:7:nablap:duhamel}
	\begin{split}
		\int_t^T\int_{\T^d}\lvert \nabla &p_{s-t}(x-y)\rvert\lvert (\widetilde{H}(\cdot,\nabla\tilde{u}^{N,\nu}_s) - \widetilde{H}(\cdot,\nabla v^{\nu}_s))*D^N (y) \rvert\de y \de s\\
		&\quad+\int_t^T\int_{\T^d}\lvert \nabla p_{s-t}(x-y)\rvert\lvert \widetilde{H}(\cdot,\nabla v^{\nu}_s)*D^N(y)\rvert\de y \de s\\
		&\leq \int_{t+\eps}^T\norm{\nabla p_{s-t}}_2\norm{\widetilde{H}(\cdot,\nabla\tilde{u}^{N,\nu}_s) - \widetilde{H}(\cdot,\nabla v^{\nu}_s)}_2\de s\\
		&\quad+C(d)\int_{t}^{t+\eps}\frac{N^{\frac{d}{2}}}{\sqrt{s-t}}\norm{\widetilde{H}(\cdot,\nabla\tilde{u}^{N,\nu}_s) - \widetilde{H}(\cdot,\nabla v^{\nu}_s)}_2\de s\\
		&\quad+C(d,T,q)\int_t^T\frac{1}{\sqrt{s-t}}\left(\norm{\nabla v^{\nu}_s}_\infty + \frac{
		 {\norm{\nabla v^{\nu}_s}_{\infty,q-1}}}{N^{q-1-\frac{d}{2}}}\right)\de s\\
		&\leq \frac{C(d,T,\cF,\cG,H,M)}{N^{q-1}}\left(\eps^{-\frac{d}{4} + \frac{1}{2}} + \eps^{\frac{1}{2}}N^{\frac{d}{2}}\right)
		+ C(d,T,\cF,\cG,H,q),
	\end{split}
	\end{equation}
where we passed the  term $1/N^{q-1-d/2}$ appearing on the penultimate line into the last constant on the last line.
By taking $\eps = N^{-2}$ we obtain 
	\begin{align*}
		\left\lvert \int_t^T\int_{\T^d}\nabla p_{s-t}(x-y)\left[\widetilde{H}(\cdot,\nabla \tilde{u}^{N,\nu}_s)*D^N\right](y)\de y \de s\right\rvert
		\leq \frac{C(d,T,\cF,\cG,M,H)}{N^{q-\frac{d}{2}}}+ C(d,T,\cF,\cG,H,q).
	\end{align*}
	Moreover, it holds
	\begin{equation*}
		\left\lvert \int_{\T^d}p_{T-t}(x-y)\left[ \nabla [\lf\cG](\nu_T,\cdot)*D^N\right](y)\de y \right\rvert
		\leq C(d,T,q)\left(\norm{\nabla \lf\cG}_\infty + \frac{{\norm{\nabla \lf\cG}_{\infty,q-1} }}{N^{q-1 - \frac{d}{2}}}\right),
	\end{equation*}
	and
	\begin{equation*}
		\left\lvert \int_t^T\int_{\T^d}\nabla p_{s-t}(x-y) \left[\lf\cF(\nu_s,\cdot)*D^N\right](y)\de y \de s\right\rvert
		\leq C(d,T,q)\left(\norm{\lf\cF}_\infty + \frac{{\norm{ \lf\cF}_{\infty,q} }}{N^{q-\frac{d}{2}}}\right).
	\end{equation*}
	\noindent Thus, we deduce that 
	\begin{equation}\label{eqn: est_grad_dep_M}
		\sup_{t\in[0,T]}\norm{\nabla\tilde{u}^{N,\nu}_t}_\infty\leq\frac{C_1(d,T,\cF,\cG,H,M,q)}{N^{\addtxtr{q-1   -\frac{d}{2}}}} + C_2(d,T,\cF,\cG,H,q)  \leq C(d,T,\cF,\cG,H,M,q). 
	\end{equation}
	\medskip
	
\noindent	\textit{Second Step.} We now address the second part of the statement. 
\addtxtr{To do so, we notice that it is possible to choose $M$ and $N\geq\widetilde{N}$, for a certain $\widetilde{N}$ depending only on $d,T,\cF,\cG,H,q$ (see also Remark \ref{rmk: H_is_useless} below), such that the constants $C$ in \eqref{eqn: est_u_dep_M} and \eqref{eqn: est_grad_dep_M} become independent of $M$. }
To make it clear, let us consider $\check{M}:=\sup_{\lvert p\rvert\leq C_2+1} \norm{\diff_p H(\cdot,p)}_\infty$, 
where $C_2 := C_2(d,T,\cF,\cG, H,q)$ comes from the estimates above, and set $\widetilde{M} := \max\{\check{M}, \overline{M}\}$, where $\overline{M} := \overline{M}(d,T,\cF,\cG, H,q)$ is defined in Remark \ref{rmk: H_is_useless_for_uinfty}. Then, $\widetilde{M}$ depends only on the data $d,T,\cF,\cG, H, q$. In the definition of $\widetilde{H}$, we can choose $M=\widetilde{M}$.  Then, we can take $\widetilde{N} :=\widetilde{N}(\widetilde{M})$, depending only on $d,T,\cF,\cG, H, q$,  such that $\sup_{t\in[0,T]}\norm{\tilde{u}^{N,\nu}_t}_\infty\leq 
C_2+1 =: C(d,T,\cF,\cG, H,q)$. The same holds also for  $\nabla\tilde{u}^{N,\nu}$ (up to changing the value of $C$).

If we choose
as above	
$M=\widetilde{M}$ and $N\geq\widetilde{N}$, 
not only we have 
the bound $\sup_{t\in[0,T]}\norm{\nabla\tilde{u}^{N,\nu}_t}_\infty\leq C(d,T,\cF,\cG,H,q)$
but we also have the identity 
 $\widetilde{H}(x, \nabla \tilde{u}^{N,\nu}_t(x)) = H(x, \nabla \tilde{u}^{N,\nu}_t(x))$
 for $t \in [0,T]$ and $x \in {\mathbb T}^d$
  (see also Remark \ref{rmk: H_is_useless} below), since $\lvert \diff_p H(x,\nabla \tilde{u}^{N,\nu}_t(x))\rvert\leq\widetilde{M}$ for any $x\in\T^d$ and $t\in[0,T]$. Thus, 
   in \eqref{eqn: dual_fpn_H},  we can identify $\widetilde H$ with $H$ and hence assume it to be of class $\rmC^{q}(\T^d\times\R^d)$.
  \vskip 6pt

  To estimate $\tilde{u}^{N,\nu}_t$ in $\rmC^{{\lfloor q - 1 - d/2 \rfloor}}(\T^d)$ independently of $N$, we 
  proceed by induction, proving iteratively bounds in 
  $\rmC^l(\T^d)$ for $l=1,\cdots,\lfloor q-1-d/2\rfloor$. The case $l=1$ follows from the first step. 
  Next, we assume that there exist $l \in \{1,\cdots,\lfloor q-1 - d/2 \rfloor - 1\}$ and a natural number $\widetilde{N} = \widetilde{N}(d,T,\cF,\cG,H,\addtxtr{q},l)$ such that, for every $N\geq\widetilde{N}$, 	it holds that $\tilde{u}^{N,\nu}\in\rmC^{1, l}([0,T]\times\T^d)$ and 
	\begin{equation}
	\label{eq:new:label:revision:induction:l}
		\sup_{t\in[0,T]}\norm{\tilde{u}^{N,\nu}_t}_{\infty,l}   \leq C(d,T,\cF,\cG,H,\addtxtr{q},l),
	\end{equation}
	for a positive constant $C(d,T,\cF,\cG,H,\addtxtr{q},l)$.

As in the first step, our strategy relies on  Duhamel's representation formula. Back to \eqref{eqn: duhamel_grad}, 
		the strategy is to take $l$ derivatives in the right-hand side. 
		We notice that all the resulting derivatives of order $l$ indeed exist thanks to the presence of the convolution by $D^N$, as a result of which we deduce that 
		$\nabla^{l+1} \tilde u^{N,\nu}_t$ also exists and 
		\begin{equation}
		\label{eq:induction:l:l+1:3.7} 
		\sup_{t  \in [0,T]} \norm{\tilde{u}^{N,\nu}_t}_{\infty,l+1} < \infty.
		\end{equation} 
However,  the bound clearly depends on $N$ at this stage of the proof
(and this is the reason why one can prove that 
 $\tilde{u}^{N,\nu}_t$ belongs in fact to  $\rmC^{q}(\T^d)$, as claimed in the statement, but with a norm depending on $N$).
 
		The objective is thus to get a bound
		for \eqref{eq:induction:l:l+1:3.7}
		 that is independent of $N$, for $N$ large enough\footnote{In fact, once the bound has been proved to be independent of 
		$N$, for 
		$N$ large enough, one should be able to obtain a bound that is independent of $N$, for all $N$, just by taking the maximum with the bounds that hold true for smaller values of $N$. 
		We refrain from doing so because it would be useless for us.}.
		To do this, 
		the key point is to provide a bound for 
	\begin{equation}\label{eqn: the_last_bad_guy}
	 \int_t^T\int_{\T^d}\nabla p_{s-t}(x-y) 
	 \nabla^l 
	 \left[\widetilde{H}(\cdot,\nabla \tilde{u}^{N,\nu}_s)*D^N\right](y)\de y \de s.
	\end{equation}
	Proceeding as in the first step, we have
	\begin{align}
	\lvert &\eqref{eqn: the_last_bad_guy}\rvert\nonumber\\
	 &\leq \int_t^T\int_{\T^d}\vert \nabla p_{s-t}(x-y) \vert  \left\vert  \nabla^l  \left[\widetilde{H}(\cdot,\nabla \tilde{u}^{N,\nu}_s)*D^N\right](y) - \nabla^l \left[  \widetilde{H}(\cdot,\nabla v^{\nu}_s) *D^N\right](y)  \right\rvert \de y \de s \label{eqn: aaa}
	 \\
	 &\quad + 	 \int_t^T\int_{\T^d}
	 \vert \nabla p_{s-t}(x-y) \vert 
	  \left\vert 
	 \nabla^l 
	 \left[  \widetilde{H}(\cdot,\nabla v^{\nu}_s)	 
	 *D^N\right](y)  \right\vert \de y \de s. \label{eqn: bbb}
		\end{align}
	\addtxtr{We recall that the map $x\mapsto\widetilde{H}(x,\nabla v^{\nu}(x))$ is of class $\rmC^{q-1}(\T^d)$, with a norm uniformly bounded by a positive constant $C = C(d,T,\cF,\cG,H,q)$. Thus, the map $x\mapsto\nabla^l\left[\widetilde{H}(x,\nabla v^{\nu}(x))\right]$ is of class $\rmC^{q-1-l}(\T^d)$, with $q-1-l \geq q -1 - (\lfloor q-1 - d/2 \rfloor - 1) > d/2$. Then, by invoking the second statement of Lemma \ref{lemma: trivial_est_trunc}, we obtain
	\begin{align*}
	 	\eqref{eqn: bbb} &\leq \int_t^T
	   \frac{C(d,q)}{\sqrt{s-t}}\left\lVert\nabla^l\left[\widetilde{H}(\cdot,\nabla v^{\nu}_s)\right]*D^N\right\rVert_{\infty}\de s\\
	  &  \leq \int_t^T \frac{C(d,q)}{\sqrt{s-t}}\left\lVert\nabla^l\left[\widetilde{H}(\cdot,\nabla v^{\nu}_s)\right]\right\rVert_{2,q-1-l}\de s
	   \leq C(d,T,\cF,\cG,H,q),
	 \end{align*}
	 where, to get the last constant, we used the inequality 
	 \begin{equation*}
	 	\left\lVert\nabla^l\left[\widetilde{H}(\cdot,\nabla v^{\nu}_s)\right]\right\rVert_{2,q-1-l}\leq \sup_{s\in[0,T]}\norm{\widetilde{H}(\cdot,\nabla v^{\nu}_s)}_{\infty,q-1}\leq C(d,T,\cF,\cG,H,q).
	\end{equation*}}
	 
	 For \eqref{eqn: aaa}, we follow again the argument 
	  used to treat
	 \eqref{eqn: inftybound_1st}. Inserting directly the resulting bound into 
	 \eqref{eqn: the_last_bad_guy}, we get
	 \begin{equation}\label{eq:DDD}
	 \begin{aligned}
	 	\lvert \eqref{eqn: the_last_bad_guy}\rvert&\leq \int_{t+N^{-2}}^T\norm{\nabla p_{s-t}}_2\norm{
\widetilde{H}(\cdot,\nabla\tilde{u}^{N,\nu}_s) - \widetilde{H}(\cdot,\nabla v^{\nu}_s)}_{2,l} \de s
\\
		&\quad+C(d)\int_{t}^{t+N^{-2}}\frac{N^{\frac{d}{2}}}{\sqrt{s-t}}\norm{\widetilde{H}(\cdot,\nabla\tilde{u}^{N,\nu}_s) - \widetilde{H}(\cdot,\nabla v^{\nu}_s)}_{2,l} \de s
		 +C(d,T,\cF,\cG,H,q).	 \end{aligned}
	 \end{equation}

	\medskip
	
	 \noindent\textit{Third Step.} The difficulty now is to get an estimate for $\norm{\widetilde{H}(\cdot,\nabla\tilde{u}^{N,\nu}_s) - \widetilde{H}(\cdot,\nabla v^{\nu}_s)}_{2,l}$.
	 By using the bounds for the derivatives of $\widetilde{u}^{N,\nu}$  up to the order $l$ and the derivatives of $v^{\nu}$  up to the  order $l+1$, we obtain (the proof is given below in the fourth step)
	 \begin{equation} 
	 \label{eq:EEE}
	 \norm{\widetilde{H}(\cdot,\nabla\tilde{u}^{N,\nu}_s) - \widetilde{H}(\cdot,\nabla v^{\nu}_s)}_{2,l} \leq C(l)
	 \norm{ \tilde{u}^{N,\nu}_s - v^{\nu}_s}_{2,l+1},
	 \end{equation} 
	 for a constant $C(l)$ depending on $l$ through the available 
	 bounds for $\sup_{t \in [0,T]} \| \widetilde{u}_t^{N,\nu} \|_{\infty,l}$ and 
	 $\sup_{t \in [0,T]} \| v^\nu_t \|_{\infty,l+1}$, with the former being given by the induction 
	 hypothesis  \eqref{eq:new:label:revision:induction:l}.
	 
	We now turn to the estimates of the right-hand side in \eqref{eq:EEE}. 
	 Following the proof of Lemma 
	 \ref{lemma: conv_l2_gradu_utilde},
	 we take $l+1$ derivatives in the equations 
	 \eqref{eqn: dual_fpn_H} and \eqref{eqn: dual_fpn_H_auxiliary}
	 and then make the difference between the two of them. Using the regularity of $\lf\cF$, $\lf\cG$, and $\widetilde{H}(\cdot,\nabla v^\nu)$, we have
	 \begin{equation} 
	 \label{eq:tilde uN:2:l+1}
	 \begin{split} 
	 &\| v_t^\nu - \tilde u_t^{N,\nu} \|^2_{2,l+1} 
	 + \int_t^T \| v_s^\nu - \tilde u_s^{N,\nu} \|^2_{2,l+2} \de s
	 \\
	 &\leq C(d,T,\cF,\cG,H,\addtxtr{q}, l)\left( \frac{1}{N^{2 (q - 1 - l )}} + \int_t^T \| \widetilde H(\cdot, \nabla v_s^\nu) * D^N - 
	 \widetilde H(\cdot, \nabla \tilde u^{N,\nu}) * D^N \|^2_{2,l} \de s\right),
	 \end{split}
	 \end{equation} 
	 where the exponent $q-1-l$ comes from the fact that 	 the $l+1$ derivatives of 
	 $\lf\cF$ and $\lf\cG$ in of the two equations  
	  \eqref{eqn: dual_fpn_H} and \eqref{eqn: dual_fpn_H_auxiliary}
	 belong to $C^{q-1-l}$ and the derivatives of $\widetilde{H}(\cdot,\nabla v^\nu)$
	 up to the order $l$ are also in $C^{q-1-l}$. Within this class, the rate of convergence
	 of the Dirichlet kernel is given by Lemma 
	 \ref{lemma: estl2truncation} and gives, at rank $N$, an error of order 
	 $N^{-(q-1-l)}$ (and hence a squared error of order 
	 $N^{-2(q-1-l)}$).
	 
	  Here, by Lemma \ref{lemma: trivial_est_trunc} and \eqref{eq:EEE},
	  \begin{equation*} 
	   \| \widetilde H(\cdot, \nabla v_s^\nu) * D^N - 
	 \widetilde H(\cdot, \nabla \tilde u^{N,\nu}) * D^N \|^2_{2,l} 
	 \leq 
	  \| \widetilde H(\cdot, \nabla v_s^\nu)  - 
	 \widetilde H(\cdot, \nabla \tilde u^{N,\nu}_s)  \|^2_{2,l} 
	 \leq C(l)
	  \| \tilde  v_s^\nu - \tilde u^{N,\nu}_s  \|^2_{2,l+1}.
	  \end{equation*} 
	  By inserting the above into 
	  	 \eqref{eq:tilde uN:2:l+1} and then applying 
	  Gronwall's lemma, we get
	  \begin{equation*}
	  	\| v_t^\nu - \tilde u_t^{N,\nu} \|_{2,l+1}\leq \frac{C(d,T,\cF,\cG,H,\addtxtr{q},l)}{N^{q-1-l}},
	  \end{equation*}
	  which makes it possible to come back to \eqref{eq:DDD}. By 
	  using 
	  the fact that $q-l-d/2 \geq q- (q-2-d/2)-d/2=2$,
	  we get
	  \begin{align*}
	  	\lvert\eqref{eqn: the_last_bad_guy}\rvert \leq C(d,T,\cF,\cG,H,\addtxtr{q},l)\left ( 1 + \frac{1}{N^{ q-l-\frac{d}{2}}}\right)\leq C(d,T,\cF,\cG,H,\addtxtr{q},l).
	  \end{align*}
	  In \eqref{eqn: duhamel_grad}, the remaining two terms of Duhamel's representation of $\nabla^{l+1} \tilde{u}^{N,\nu}$ can be handled as follows
	  \begin{equation*}
		\left\lvert \int_{\T^d}p_{T-t}(x-y)\left[ \nabla^{l+1} [\lf\cG](\nu_T,\cdot)*D^N\right](y)\de y \right\rvert
		\leq C(d,T,q)\left(\norm{ \lf\cG}_{\infty,l+1} + \frac{ {\norm{ \lf\cG}_{\infty,q}}}{N^{q-1 - l -\frac{d}{2}}}\right),
	\end{equation*}
	and
	\begin{equation*}
		\left\lvert \int_t^T\int_{\T^d}\nabla p_{s-t}(x-y) \left[\nabla^l[\lf\cF](\nu_s,\cdot)*D^N\right](y)\de y \de s\right\rvert
		\leq C(d,T,q)\left(\norm{\lf\cF}_{\infty,l} + \frac{ {\norm{\lf\cF}_{\infty,q}}}{N^{q-l -\frac{d}{2}}}\right).
	\end{equation*}
	Thus, combining the last three estimates
	and using the bound $q-1-l-d/2 \geq 1$, we obtain 
	\begin{equation*}
		\sup_{t\in[0,T]}\norm{\tilde{u}^{N,\nu}_t}_{\infty,l+1}   \leq C(d,T,\cF,\cG,H,  \addtxtr{q} ,l+1),
	\end{equation*}
	which the completes the induction. 
	\medskip
	
	  \noindent\textit{Fourth Step.} It remains to prove \eqref{eq:EEE}. It suffices to study, for $r\in\{1,\dots,l\}$, 
	  \begin{equation*}
	  \bigl\| \nabla^r \bigl( \widetilde{H}(\cdot,\nabla\tilde{u}^{N,\nu}_s) - \widetilde{H}(\cdot,\nabla v^{\nu}_s) \bigr) \bigr\|_2.
	  \end{equation*} 
	  We focus on the worst case, that is $r=l$. As done e.g.~in the proof of \cite[Lemma 5.4]{daudindelaruejackson},
	  we notice that the difference above can be written as sum of terms of the form 
\begin{equation*}
\begin{split} 
&\partial_{x_i}^{ l-k} \partial_{p_{i_1}\dots p_{i_b}} H(\cdot, \nabla\tilde{u}^{N,\nu}_s ) \partial_{x_i}^{j_1} \partial_{x_{i_1}} \tilde{u}^{N,\nu}_s \dots \partial_{x_i}^{j_b} \partial_{x_{i_b}} \tilde{u}^{N,\nu}_s
\\
&\quad - \partial_{x_i}^{ l-k} \partial_{p_{i_1}\dots p_{i_b}} H(\cdot, \nabla v^{\nu}_s ) \partial_{x_i}^{j_1} \partial_{x_{i_1}} v^{\nu}_s \dots \partial_{x_i}^{j_b} \partial_{x_{i_b}} v^{\nu}_s,
\end{split}
\end{equation*}
with $k,b\in\{0,\dots,l\}$, $j_1+\dots +j_b=k$, 
$j_1,\dots,j_b \geq 1$, 
$i,i_1,\dots,i_b\in\{1,\dots,d\}$. When $b\geq 2$ or $k \leq l-1$ all the derivatives 
 of $v_s^\nu$ appearing right above are of order less than $l-1+1=l$ and are bounded by the induction hypothesis. 
Therefore, we can 
\addtxtr{easily reformulate the above difference 
of products as an increment of the form 
$$\Psi\Bigl(  \nabla\tilde{u}^{N,\nu}_s, \partial_{x_i}^{j_1} ,\partial_{x_{i_1}} \tilde{u}^{N,\nu}_s, \dots ,\partial_{x_i}^{j_b} \partial_{x_{i_b}} \tilde{u}^{N,\nu}_s
\Bigr) 
- 
\Psi\Bigl(  \nabla {u}^{N,\nu}_s, \partial_{x_i}^{j_1} ,\partial_{x_{i_1}} {u}^{N,\nu}_s, \dots ,\partial_{x_i}^{j_b} \partial_{x_{i_b}} {u}^{N,\nu}_s
\Bigr),$$ for a Lipschitz continuous function 
$\Psi$. It is then easy to bound the above difference by 
the right-hand side in 
\eqref{eq:EEE}.}

When $b=1$, we have 
\begin{equation*}
\partial_{x_i}^{ l-k} \partial_{p_{i_1}}H(\cdot, \nabla\tilde{u}^{N,\nu}_s ) \partial_{x_i}^{k} \partial_{x_{i_1}} \tilde{u}^{N,\nu}_s
- \partial_{x_i}^{ l-k} \partial_{p_{i_1}} H(\cdot, \nabla v^{\nu}_s ) \partial_{x_i}^{k} \partial_{x_{i_1}} v^{\nu}_s,
\end{equation*}
and the conclusion is the same if $k \leq l-1$. In fact, the main difficulty arises when $b=1$ and $k=l$. 
To make the exposition clear, we feel better to switch back from partial derivatives to gradients and thus rewrite the above difference (in this case) in the form
\begin{equation*}
\begin{split} 
&\diff_p H(\cdot, \nabla \tilde{u}^{N,\nu}_s )  \nabla^{l+1} \tilde{u}^{N,\nu}_s - \diff_p H(\cdot, \nabla v^{\nu}_s )  \nabla^{l+1} v^{\nu}_s\\
&= \diff_p H(\cdot, \nabla \tilde{u}^{N,\nu}_s )  \Bigl( \nabla^{l+1} \tilde{u}^{N,\nu}_s -    \nabla^{l+1} v^{\nu}_s \Bigr) 
+\Bigl( \diff_p H(\cdot, \nabla \tilde{u}^{N,\nu}_s )  - \diff_p H(\cdot, \nabla v^{\nu}_s ) \Bigr)\nabla^{l+1} v^{\nu}_s,
\end{split}
\end{equation*}
which is enough to conclude. Indeed, to deal with the first term on the right-hand side we notice that $ \diff_p H(\cdot, \nabla \tilde{u}^{N,\nu}_s )$ is bounded thanks to the first step and the regularity of $H$, while for the second one we combine the Lipschitz regularity of $p\mapsto \diff_pH(\cdot,p)$  and the boundedness of $\nabla^{l+1} v^{\nu}$ (see Proposition \ref{prop: bound_aux_eq_tilde})
\end{proof}
\begin{remark}
	By Proposition \ref{prop: infinitybounds_tilde}, $\nabla\tilde{u}^{N,\nu}$ is bounded (in $L^\infty$) by a constant that does not depend on $M$, at least when $N$ is chosen larger than $\widetilde{N}$. This implies that the $L^2$ bounds obtained in Proposition \ref{prop: l2estimates_uN} are also independent of $M$ (for the same range of values 
	for $N$). 
\end{remark}
\begin{remark}\label{rmk: H_is_useless}
{Consistently with the argument used in the proof of 
Proposition \ref{prop: infinitybounds_tilde}, we stress that, under the conditions $M = \widetilde{M}$ 
and $N \geq \widetilde{N}$, the auxiliary Hamiltonian  $\widetilde{H}$ satisfies
  $\widetilde{H}(x,\nabla\tilde{u}^{N,\nu}_t(x))=H(x, \nabla\tilde{u}^{N,\nu}_t(x))$ for any $t\in[0,T]$ and $x\in\T^d$.  In particular, under the same conditions, the two equations \eqref{eqn: nonlocal_HJ_tilde_nu} and \eqref{eqn: dual_fpn_H} coincide. More generally, it holds that $\widetilde{H}(x,p) = H(x,p)$ whenever evaluated at $p=\nabla u^\infty (x)$, $\nabla v^{N,\nu}(x)$, and $\nabla\tilde{u}^{N,\nu}(x)$.}
\end{remark}

\subsubsection{Regularity of $u^{N,\nu}$} 
\label{subsubse:regularity:uNnu}
We now return to the solution $u^{N,\nu}$ of the original equation \eqref{eqn: nonlocal_HJ_tilde_nu}.
As previously mentioned in Remark \ref{rmk: H_is_useless}, the crucial point is that, for $M = \widetilde M$ and 
$N \geq \widetilde{N}$, 
$\widetilde{H}(x,\nabla\tilde{u}^{N,\nu}_t(x))$
coincides with $H(x, \nabla\tilde{u}^{N,\nu}_t(x))$. In this context, $\tilde{u}^{N,\nu}$ 
solves  \eqref{eqn: nonlocal_HJ_tilde_nu} \addtxtr{on the entire $[0,T] \times {\mathbb T}^d$, and by uniqueness 
of 
the maximal solution to the latter equation (see Proposition \ref{prop: ex_uniq_u_nu}), 
$u^{N,\nu}$
is also defined globally and coincides with 
$\tilde u^{N,\nu}$.} Consequently, the results obtained in Proposition \ref{prop: infinitybounds_tilde} automatically transfer to $u^{N,\nu}$ whenever $N \geq\widetilde{N}$. For convenience, let us restate these results in the context we are interested in. 
\begin{proposition}\label{prop: infinitybounds}
	Let Assumption \ref{hp: 1+2} be in force and let us consider $N\geq\widetilde{N}$. 
	 Then, $u^{N,\nu}$ 
	 \addtxtr{is a global solution of  \eqref{eqn: nonlocal_HJ_tilde_nu} and
	 belongs to $\rmC^{1,{q}}([0,T] \times \T^d)$. Moreover,} 
	 there exists a positive constant $C = C(d,T,\cF,\cG,H,q)$ such that  
	\begin{equation*}
		\sup_{t\in[0,T]}\norm{u^{N,\nu}_t}_{\infty,{\lfloor q-1-d/2 \rfloor}} 
		\leq C(d,T,\cF,\cG,H,q).
	\end{equation*}
\end{proposition}
\noindent
A key feature of the bounds we have obtained is the fact that they are independent of the fixed flow $\nu\colon[0,T]\to\prob_N(\T^d)$, as well as of the index $N$ (throughout, it is chosen greater than $\widetilde{N}$). This will play a key role in our fixed-point argument.
\begin{remark}\label{rmk: smoothness_H_of_grad}
	We have proved that for $N\geq\widetilde{N}$, 
	$\sup_{t\in[0,T]}(\norm{u^\infty_t}_{\infty,q} + \norm{v^{\nu}_t}_{\infty,q}+\norm{u^{N,\nu}_t}_{\infty,{\lfloor q-1-d/2 \rfloor}})\leq C(d,T,\cF,\cG,H,q)$. 
	Moreover, we know from  Assumption \ref{hp: 1+2} - (H.5) that $H$ belongs to $\rmC^{q}(\T^d\times\R^d)$. Therefore,  the mappings $x\mapsto H(x,\nabla u_t^\infty(x))$,  $x\mapsto H(x,\nabla v_t^\nu(x))$, and their derivatives up to order ${q-1}$, are bounded by a (uniform in time) constant which depends on $d,T,\cF,\cG$, $H,q$. 
	The same holds for the mapping $x\mapsto H(x,\nabla u^{N,\nu}_t(x))$ and its derivatives up to order $\lfloor q-2-d/2 \rfloor$. 
We recall that the estimates on $u^{N,\nu}_t$ are not in the same class as those of $u^\infty_t$ and $v^{\nu}_t$, see Remark \ref{rmk: q_not_2q}.
\end{remark}

In the whole rest of the paper, \textbf{we will always consider $N\geq\widetilde{N}$}, with $\widetilde{N}$ as in Remark \ref{rmk: H_is_useless}. However, we will emphasize or recall this point when necessary.

\subsection{The approximating FP equation for a fixed flow in $\prob_N(\T^d)$}\label{sec: prop_of_mutilde}
Let us fix $N\in\N$, $N\geq\widetilde{N}$, and consider again a flow $\nu\colon[0,T]\to\prob_N(\T^d)$. From Subsection \ref{sec: nonlocalHJ}, there exists a unique classical solution $u^{N,\nu}$ to \eqref{eqn: nonlocal_HJ_tilde_nu}, which is bounded uniformly in time, and whose derivatives up to order $\lfloor q-1-d/2 \rfloor$ are also bounded. 
 We can then focus on the following version of the approximating Fokker-Planck equation \eqref{eqn: opt_fp_prox}:
\begin{equation}\tag{$FP_{N,\nu}$}
\label{eqn: opt_fp_prox_nu}
	\begin{cases}
		\partial_t \mu^{N,\nu}_t &= \Delta \mu^{N,\nu}_t + \div\Big(\diff_p H(\cdot,\nabla u^{N,\nu}_t)( \mu^{N,\nu}_t*D^N)\Big),\quad t\in(0,T],\\
		 \mu^{N,\nu}_0 &= m.
	\end{cases}
\end{equation} 
Our main goal is to show that $\mu^{N,\nu}$ is actually a probability measure, as well as $\mu^N*D^N$, at least when $N$
is large enough. To this aim, we will take advantage of the results on $u^{N,\nu}$ obtained in Subsection \ref{sec: nonlocalHJ}, together with some further Sobolev estimates  on $\mu^{N,\nu}$ itself.\vskip 6pt

\noindent
As initial step, we need to discuss the well-posedness of \eqref{eqn: opt_fp_prox_nu}. We proceed as we did for \eqref{eqn: nonlocal_HJ_tilde_nu}, by studying the Fourier coefficients of $\mu^{N,\nu}$. For  $k\in\Z^d$ fixed, it holds 
\begin{align*}
	\frac{\de}{\de t} \widehat{\mu}^{N,\nu}_t(k) & = -\lvert2\pi k \rvert^2 \widehat{\mu}^{N,\nu}_t(k) + \ii 2\pi k\cdot \reallywidehat{\diff_p H(\cdot,\nabla u^{N,\nu} _t) \left({\mu}^{N,\nu}_t*D^N\right)}(k)
	\\
	& = -\lvert 2\pi k \rvert^2 \widehat{\mu}^{N,\nu}_t(k) + \ii 2\pi k\cdot \sum_{l\in\Z^d} \Big\{\reallywidehat{\diff_p H(\cdot,\nabla u^{N,\nu} _t) }\Big\}(l)\Big\{\reallywidehat{{\mu}^{N,\nu}_t*D^N}\Big\}(k-l)
	\\
	& = -\lvert 2\pi k \rvert^2 \widehat{\mu}^{N,\nu}_t(k) + \ii 2\pi k\cdot \sum_{\lvert k - l\rvert\leq N} \Big\{\reallywidehat{\diff_p H(\cdot,\nabla u^{N,\nu} _t)}\Big\}(l)\widehat{\mu}^{N,\nu}_t(k-l),
\end{align*}
where $\reallywidehat{\diff_p H(\cdot,\nabla u^{N,\nu} _t) ({\mu}^N_t*D^N)}(k)$ denotes the $k$-th Fourier coefficient of the mapping $x\mapsto\diff_p H(x,\nabla u^{N,\nu} _t(x)) \left({\mu}^N_t*D^N\right)(x)$.
In particular, we obtain the following system of ODEs (in closed form) for the Fourier coefficients of order $\lvert k \rvert \leq N$:
\begin{equation}\label{eqn: closed_sys_ODE}
	\begin{aligned}
		\frac{\de}{\de t} \widehat{\mu}^{N,\nu}_t(k) &= -\lvert 2\pi k \rvert^2 \widehat{\mu}^{N,\nu}_t(k) + \ii 2\pi k\cdot 
		\hspace{-4pt} \sum_{\lvert k - l\rvert\leq N} \Big\{\reallywidehat{\diff_p H(\cdot,\nabla u^{N,\nu} _t)}\Big\}(l)\widehat{\mu}^{N,\nu}_t(k-l),\quad \lvert k \rvert \leq N,\\
		 \widehat{\mu}^{N,\nu}_0(k) &= \widehat m (k).
	\end{aligned}
\end{equation}
The remaining coefficients $\{\widehat{\mu}^{N,\nu}_t(k)\}_{\lvert k \rvert> N}$ can be obtained by solving a family of ODEs with a forcing term that depends upon the coefficients up to order $N$ (i.e., $\{\widehat{\mu}^{N,\nu}_t(j)\}_{\lvert j \rvert\leq N}$):
\begin{equation}\label{eqn: remainders_ODE}
	\begin{aligned}
		\frac{\de}{\de t} \widehat{\mu}^{N,\nu}_t(k) &= -\lvert2\pi k \rvert^2 \widehat{\mu}^{N,\nu}_t(k) +  g_k\left({u}^{N,\nu},
		\{\widehat{\mu}^{N,\nu}_t(j)\}_{\lvert j \rvert\leq N}\right),\quad \lvert k\rvert>N,\\
		 g_k\left(u^{N,\nu}, \{\widehat{\mu}^{N,\nu}_t(j)\}_{\lvert j \rvert\leq N}\right) &= \ii 2\pi k\cdot \sum_{\lvert k - l\rvert\leq N} \Big\{\reallywidehat{\diff_p H(\cdot,\nabla u^{N,\nu} _t)}\Big\}(l)\widehat{\mu}^{N,\nu}_t(k-l),
		 \\
		\widehat{\mu}^{N,\nu}_0(k) &= \widehat m (k).
	\end{aligned}
\end{equation}
The sequence of Fourier coefficients $\{\widehat{\mu}^{N,\nu}(k)\}_{ k\in\Z^d}$ is uniquely determined:
	 {\begin{proposition}\label{prop: well_pos_mu}
	Let Assumption \ref{hp: 1+2} hold  and  $m\in L^2(\T^d)$. Then the sequence of coefficients $\{\widehat{\mu}^{N,\nu}(k)\}_{ k\in\Z^d}$ is uniquely determined, and for every $k\in \Z^d$ and $t\in[0,T]$, $\widehat{\mu}^{N,\nu}(k)$ is bounded by a constant which depends only on $d,T,\cF,\cG,H,m,\addtxtr{q}$ and $N$. 
	In particular, for any $N\geq\widetilde{N}$ and $\nu\colon[0,T]\to\prob_N(\T^d)$, there exists a unique solution to 
	\eqref{eqn: opt_fp_prox_nu}  in  $\rmC([0,T];L^2({\mathbb T}^d))$; when 
	$m \in H^1({\mathbb T}^d)$, the solution belongs to 
	$\rmC^{1/2}([0,T];L^2({\mathbb T}^d))$.
\end{proposition}}
\begin{proof}
Thanks to Proposition \ref{prop: infinitybounds} and the smoothness of $H$, the system \eqref{eqn: closed_sys_ODE} 
(for 
Fourier
modes
$\lvert k \rvert \leq N$)
has Lipschitz coefficients and admits a unique solution. Once found  the family $\{\widehat{\mu}^{N,\nu}(j)\}_{\lvert k \rvert \leq N}$, the remaining coefficients can be easily computed by solving the linear ODEs \eqref{eqn: remainders_ODE}.  

\addtxtr{To study the continuity in time of $\mu^{N,\nu}$, we decompose it in the form
\begin{equation*}
\mu^{N,\nu}_t = \mu^{N,\nu,(1)}_t + \mu^{N,\nu,(2)}_t, 
\quad t \in [0,T], 
\end{equation*} 
where $\mu^{N,\nu,(1)}$ solves
the heat equation with $0$ as initial condition and $\div\big(\diff_p H(\cdot,\nabla u^{N,\nu}_t)( \mu^{N,\nu}_t*D^N)\big)$ as a source term, i.e.,
\begin{equation}
\label{eq:muNnu1}
	\begin{cases}
		\partial_t \mu^{N,\nu,(1)}_t &= \Delta \mu^{N,\nu,(1)}_t + \div\Big(\diff_p H(\cdot,\nabla u^{N,\nu}_t)( \mu^{N,\nu}_t*D^N)\Big),\quad t\in(0,T],\\
		 \mu^{N,\nu,(1)}_0 &= 0,
	\end{cases}
\end{equation}
and $\mu^{N,\nu,(2)}$ solves the standard heat equation with $m$ as initial condition and $0$ as  source term.}

Regarding $\mu^{N,\nu,(1)}$, we claim that the term \addtxtr{inside the divergence in \eqref{eq:muNnu1}} belongs to $H^1({\mathbb T}^d)$, uniformly in time. \addtxtr{This follows from the smoothness of $H$ and the regularity of 
$u^{N,\nu}$ established in 
Proposition \ref{prop: infinitybounds} (with $q \geq d+3$ here), together with the smoothness of the convoluted function $\mu^{N,\nu}*D^N$. In turn, this} implies that $\mu^{N,\nu,(1)}$ takes values in $H^1({\mathbb T}^d)$, uniformly in time. 
\addtxtr{Integrating 
\eqref{eq:muNnu1}
with respect to 
$\mu^{N,\nu,(1)}_t - \mu^{N,\nu,(1)}_s$
for a fixed time $s$, we deduce that 
 $\mu^{N,\nu,(1)}$} belongs to $\rmC^{1/2}([0,T];L^2({\mathbb T}^d))$. 
Regarding $\mu^{N,\nu,(2)}$, the continuity in time is a standard consequence of the representation of the solution as a convolution with the heat kernel. Moreover, when $m \in H^1({\mathbb T}^d)$, 
$\mu^{N,\nu,(2)}$ takes values in $H^1({\mathbb T}^d)$, uniformly in time, and thus belongs to 
 $\rmC^{1/2}([0,T];L^2({\mathbb T}^d))$.
\end{proof}

\subsubsection{An auxiliary FP equation}
In this subsection we introduce the standard (hence simpler) version of \eqref{eqn: opt_fp_prox_nu}, which does not involve all the difficulties related to the convolution by the Dirichlet kernel $D^N$. It reads
\begin{equation}\tag{$\widetilde{FP}_{N,\nu}$}\label{eqn: FP_tilde_nu}
	\begin{cases}
		\partial_t \tilde{\mu}^{N,\nu}_t &= \Delta\tilde{\mu}^{N,\nu}_t + \div\big(\diff_p H(\cdot,\nabla u^{N,\nu}_t)\tilde{\mu}^{N,\nu}_t\big),\quad t\in(0,T],\\
		\tilde{\mu}^{N,\nu}_0 &= m\in H^{q-1}(\T^d).
	\end{cases}
\end{equation}
We notice that \eqref{eqn: FP_tilde_nu} is the classical FP equation associated with the drift $-\diff_p H(\cdot,\nabla u^N_t(\cdot))$ and with diffusion of intensity $\sqrt{2}$. Thus, following for instance \cite[Chapter 6]{bogachevkrylovrocknershaposhnikov}, we have that \eqref{eqn: FP_tilde_nu} has a unique solution in $\rmC^{1/2}([0,T];\prob(\T^d))$. Moreover, $\tilde{\mu}^{N,\nu}_s$ has a density in $L^2(\T^d)$ for every $t\in[0,T]$. In the next proposition, we provide some further Sobolev estimates of $\tilde{\mu}^{N,\nu}$.  
\begin{proposition}\label{prop: smooth_muN}
	Let Assumption \ref{hp: 1+2} hold and $m\in H^{{q-1}}(\T^d)$. Then, {$\tilde{\mu}^{N,\nu}$ is in the space $\rmC^{1/2, 0}([0,T]\times\T^d)$} and for any $t\in[0,T]$, $\tilde{\mu}^{N,\nu}_t$ belongs to $H^{{q-1}}(\T^d)$. Moreover, there exists a positive constant $C = C(d,T,\cF,\cG,H,q)$ such that 
	\begin{equation*}
		\sup_{t\in[0,T]}\left(\norm{\tilde{\mu}^{N,\nu}_t}_{2,\lfloor q-2-d/2 \rfloor} + \norm{\tilde{\mu}^{N,\nu}_t}_{\infty}\right) \leq C(d,T,\cF,\cG,H,q)\norm{m}_{2,\lfloor q-2-d/2 \rfloor}.
	\end{equation*}
\end{proposition}
\begin{proof}	
	By Proposition \ref{prop: infinitybounds}, $u^{N,\nu}_t\in\rmC^ q(\T^d)$ and $\sup_{t\in[0,T]}\norm{u^{N,\nu}_t}_{\infty,\lfloor q-1-d/2 \rfloor}\leq C(d,T,\cF,\cG,H,q)$. From \eqref{eqn: FP_tilde_nu} we have
	\begin{align*}
		\norm{\tilde{\mu}^{N,\nu}_t}_2^2 + \int_0^t\norm{\nabla \tilde{\mu}^{N,\nu}_s}_2^2\de s
		&\leq \norm{m}^2_2 + \int_0^t \lvert\scalprod{\diff_p H(\cdot,\nabla u^{N,\nu}_s) \tilde{\mu}^{N,\nu}_s}{\nabla\tilde{\mu}^{N,\nu}_s}\rvert\de s\\
		&\leq \norm{m}^2_2 + \int_0^t \norm{\diff_p H(\cdot,\nabla u^{N,\nu}_s)}_\infty\norm{\tilde{\mu}^{N,\nu}_s}_2\norm{\nabla\tilde{\mu}^{N,\nu}_s}_2\de s\\
		&\leq C(d,T,\cF,\cG,H,q)\left(\norm{m}_2^2 + \int_0^t \norm{\tilde{\mu}^{N,\nu}_s}^2_2 \de s\right)+ \frac{1}{2}\int_0^t\norm{\nabla \tilde{\mu}^{N,\nu}_s}_2^2\de s.
	\end{align*}
	Thus, by Gronwall's inequality, $\sup_{t\in[0,T]}\norm{\tilde{\mu}^{N,\nu}_t}_2\leq C(d,T,\cF,\cG,H,q)\norm{m}_2$.\\
	
	\addtxtr{\noindent Let us proceed similarly for the gradient, and then sketch how to iterate this procedure to higher-order derivatives.	By taking  for any $i\in\{ 1,\dots,d\}$ the $i$-th partial derivative in space of \eqref{eqn: FP_tilde_nu}, we obtain	
	\begin{equation*}
	\begin{cases}
		\partial_t (\partial_{x_i}\tilde{\mu}^{N,\nu}_t) &= \Delta(\partial_{x_i}\tilde{\mu}^{N,\nu}_t) + \div\bigl(\partial_{x_i} \bigl[ \diff_pH(\cdot,\nabla u^{N,\nu}_t) \bigr] \tilde{\mu}^{N,\nu}_t\bigr)+\div(\diff_p H (\cdot,\nabla u^{N,\nu}_t)\partial_{x_i} \tilde{\mu}^{N,\nu}_t),\\
	 \partial_{x_i}\tilde{\mu}^{N,\nu}_0 &= \partial_{x_i} m,
	\end{cases}
	\end{equation*}
	where $\partial_{x_i} [ \diff_pH(\cdot,\nabla u^{N,\nu}_t) ] =( \partial_{x_i}\diff_{p} H)(\cdot,\nabla u^{N,\nu}_t) + \diff^2_{pp}H(\cdot,\nabla u^{N,\nu}_t)(\partial_{x_i}\nabla u^{N,\nu}_t)$, and  $\partial_{x_i}$ acts componentwise when applied to a vector-valued map.
	Thus, by integrating against $\partial_{x_i}\tilde{\mu}^{N,\nu}$, we have
	\begin{align*}
		\norm{\partial_{x_i}\tilde{\mu}^{N,\nu}_t}_2^2 + \int_0^t\norm{\nabla \partial_{x_i}\tilde{\mu}^{N,\nu}_s}_2^2\de s 
		&\leq \norm{\partial_{x_i} m}_2^2 
		+ \int_0^t\norm{\diff^2_{px} H(\cdot,\nabla u^{N,\nu}_t)}_\infty\norm{\tilde{\mu}^{N,\nu}_s}_2\norm{\nabla\partial_{x_i}\tilde{\mu}^{N,\nu}_s}_2\de s \\
		&\quad+ \int_0^t\norm{\nabla^2 u^{N,\nu}_s}_\infty\norm{\diff^2_{pp}H(\cdot,\nabla u^{N,\nu}_t)}_\infty\norm{\tilde{\mu}^{N,\nu}_s}_2\norm{\nabla\partial_{x_i}\tilde{\mu}^{N,\nu}_s}_2\de s \\
		&\quad+ \int_0^t\norm{\diff_pH(\cdot,\nabla u^{N,\nu}_s)}_\infty\norm{\partial_{x_i} \tilde{\mu}^{N,\nu}_s}_2\norm{\nabla\partial_{x_i}\tilde{\mu}^{N,\nu}_s}_2\de s.
	\end{align*}
	By Proposition \ref{prop: infinitybounds}, $u^{N,\nu}_t\in\rmC^q(\T^d)$ with $q\geq2$ and $\sup_{s\in[0,T]} \norm{\nabla^2 u^{N,\nu}_s}_\infty\leq C(d,T,\cF,\cG,H,q)$
	(because $q-1-d/2 \geq 2$). Moreover, $\sup_{s\in[0,T]}(\norm{\diff^2_{px} H(\cdot,\nabla u^{N,\nu}_t)}_\infty + \norm{\diff^2_{pp}H(\cdot,\nabla u^{N,\nu}_t)}_\infty)\leq C(d,T,\cF,\cG,H,q)$  because  $H\in\rmC^q(\T^d\times\R^d)$ and $\nabla u^{N,\nu}$ is uniformly bounded. Hence, we can conclude again by Young's and then Gronwall's  inequalities.
	}
	
	 So far, we have obtained a simple bound \addtxtr{for $\tilde{\mu}^{N,\nu}$ in $H^1(\T^d)$ by using  the fact that $u^{N,\nu}_t\in\rmC^2(\T^d)$ and $H\in\rmC^2(\T^d\times\R^d)$. By the same technique, we can 
	bound
	$\tilde{\mu}^{N,\nu}_t$ in $H^{\lfloor q-2-d/2 \rfloor}(\T^d)$
	by using the fact that $u^{N,\nu}$ and $H$ are of class $\rmC^{\lfloor q-1-d/2 \rfloor}$. 
	 The principle is as follows:}  let us consider $2\leq l \leq \lfloor q-2-d/2 \rfloor$ and 
	  assume by induction that $\sup_{t\in[0,T]}\norm{\tilde{\mu}^{N,\nu}_t}_{2,l-1}\leq C(d,T,\cF,\cG,H,q)\norm{m}_{2,l-1}$. \addtxtr{
	For any $(i_1,\cdots,i_l)\in \{1,\dots,d\}^l$, $\partial^l_{x_{i_1}\dots x_{i_l}}\tilde{\mu}^{N,\nu}$ solves
	\begin{align*}
		\partial_t
		(\partial^l_{x_{i_1}\dots x_{i_l}} \tilde{\mu}^{N,\nu}_t)& = \Delta( 		\partial^l_{x_{i_1}\dots x_{i_l}}  \tilde{\mu}^{N,\nu}_t)
		+ \div\Big\{\diff_p H(\cdot, \nabla u^{N,\nu}_t)
		\partial^l_{x_{i_1}\dots x_{i_l}} 
		\tilde{\mu}^{N,\nu}_t\Big\}
		\\
		&\quad + \sum_{A \subset \{1,\dots,l\} : A \not = \emptyset}
		\div\Big\{
		\partial^{\vert A\vert}_{{\boldsymbol x}_{i_A}}
		[\diff_pH(\cdot, \nabla u^{N,\nu}_t)]
				\partial^{\vert A_c \vert}_{{\boldsymbol x}_{i_{A_c}}}
		\tilde{\mu}^{N,\nu}_t\Big\}, 
	\end{align*}
	where the sum right above runs over non-empty subsets $A$ of 
$\{1,\dots,l\}$. In the divergence, 
 $\vert A\vert$ denotes
the cardinal of $A$, ${\boldsymbol x}_{i_A}$  the tuple $(x_{i_j})_{j \in A}$	
and $A_c$  the complementary of $A$ (with similar notations for $\vert A_c\vert$
and 	${\boldsymbol x}_{i_{A_c}}$). 
Integrating the equation against
$\partial^l_{x_{i_1}\dots x_{i_l}}\tilde{\mu}^{N,\nu}$ 
and then 
denoting by $\nabla^l \tilde{\mu}^{N,\nu}$ the tensor of all 
the $l$-derivatives of $\tilde{\mu}^{N,\nu}$, 
	we obtain (for a constant $C(l)$ depending on $l$)}
	\begin{align*}
		\norm{\nabla^l\tilde{\mu}^{N,\nu}_t}^2_2 &+\int_0^t \norm{\nabla^{l+1}\tilde{\mu}^{N,\nu}_s}^2_2\de s\\
		&\leq \norm{\nabla^l m}^2_2 
		+ \frac{1}{2}\int_0^t \norm{\nabla^{l+1}\tilde{\mu}^{N,\nu}_s}^2_2\de s
		+ 
\addtxtr{C(l)}		\sup_{t\in[0,T]}\norm{\diff_pH(\cdot,\nabla u^{N,\nu}_t)}^2_\infty\int_0^t \norm{\nabla^l\tilde{\mu}^{N,\nu}_s}^2_2 \de s\\
		&\quad + \addtxtr{C(l)}\sup_{t\in[0,T]}\norm{\diff_pH(\cdot,\nabla u^{N,\nu}_t)}^2_{\infty,l}\int_0^t\norm{\tilde{\mu}^{N,\nu}_s}^2_{2,l-1}\de s + \frac{1}{2}\int_0^t \norm{\nabla^{l+1}\tilde{\mu}^{N,\nu}_s}^2_2\de s.
	\end{align*}
	By recalling that $u^{N,\nu}_t\in\rmC^{{\lfloor q-1-d/2 \rfloor}}(\T^d)$, $H\in\rmC^q(\T^d\times\R^d)$ and $l+1\leq{\lfloor q-1-d/2 \rfloor}$,  we have 
	\begin{equation*}
		\sup_{t\in[0,T]}\norm{\diff_pH(\cdot,\nabla u^{N,\nu}_t)}^2_{\infty,l}\leq C(\norm{H}_{\infty,l+1}, \sup_{t\in[0,T]}\norm{u^{N,\nu}}_{\infty,l+1}).
	\end{equation*}By Gronwall's inequality and the induction hypothesis, we obtain, for every $l \leq \lfloor q-2-d/2 \rfloor$,
	\begin{equation*}
		\sup_{t\in[0,T]}\norm{\tilde{\mu}^{N,\nu}_t}_{2,l}\leq C(d,T,\cF,\cG,H,q)\norm{m}_{2,l}.
	\end{equation*}
	Using the fact that ${u}^{N,\nu}$ takes values in $\rmC^q({\mathbb T}^d)$, but with bounds depending on $N$, we deduce in a similar manner that 
	the left-hand side is finite when $l=q-1$, but it may not be bounded uniformly in $N$.
	
	It remains to check that $\tilde{\mu}^{N,\nu}$
	is time-space continuous. 
	Following the proof of Proposition 
	\ref{prop: well_pos_mu}, we can prove that 
	 $\tilde{\mu}^{N,\nu} \in  \rmC^{1/2}([0,T];L^2({\mathbb T}^d))$.
Moreover, 
since 
$\tilde{\mu}^{N,\nu}$ 
takes values in $H^{{\lfloor q-2-d/2 \rfloor}}({\mathbb T}^d)$, uniformly in time, 
with ${\lfloor q-2-d/2 \rfloor=q-2 + \lfloor -d/2 \rfloor \geq d+1+ \lfloor -d/2 \rfloor} \geq d+1-d/2-1/2=d/2+1/2>d/2$, we deduce from Sobolev embedding that 
$\tilde{\mu}^{N,\nu}$ takes values in a compact subset of 
$\rmC({\mathbb T}^d)$, uniformly in time. Joint time-space continuity follows from the combination of the last two 
properties. 
\end{proof}
We conclude our analysis of $\tilde{\mu}^{N,\nu}$ by showing that, if the initial density is 
\addtxtr{strictly positive (uniformly on ${\mathbb T}^d$)}, then the same is true all along the flow of probability measures.
\begin{lemma}\label{lemma: lower_bound_muN}
	Let Assumptions \ref{hp: 1+2} and \ref{hp: 3} be in force. Then, there exists a constant \addtxtr{$\tilde\gamma=\tilde{\gamma}(d,T,\cF,\cG,H,q,\gamma) \in (0,1)$},  such that $\inf_{t\in[0,T]}\tilde{\mu}^{N,\nu}_t\geq\tilde\gamma>0$. 
\end{lemma}
\begin{proof}
	From \cite[Theorem 1]{aronson} and 
		Proposition \ref{prop: infinitybounds}, we have that the fundamental solution $\Gamma = \Gamma(x,t;y,\tau)$ to the forward equation in \eqref{eqn: FP_tilde_nu} is bounded from below by a positive constant $\bar c$ times a suitable gaussian kernel $\bar p =\bar p_{t-\tau}(x-y)$, all depending only on $T$ and the data. Thus, for any $0<\tau<t$ and $x\in\T^d$, we have
	\begin{equation*}
		\tilde{\mu}^{N,\nu}_t(x) = \int_{\T^d}\Gamma(y,t;x,\tau)m(y)\de y\geq\gamma\bar c\int_{\T^d}\bar p_{t-\tau}(x-y) \de y=:\tilde\gamma>0.
	\end{equation*} 
\end{proof}
\subsubsection{Proving that $\mu^{N,\nu}$ and $\mu^{N,\nu}*D^N$ are flow of probabilities}\label{sec: mutilde_and_truc_are_prob}
We now return to 
 $\mu^{N,\nu}$.
As we have just done  for $\tilde{\mu}^{N,\nu}$, we first prove that $\mu^{N,\nu}$ has some Sobolev regularity when $m \in H^{q-1}(\T^d)$.
\begin{proposition}\label{prop: tildemu_in_Hp} 
	Let Assumption \ref{hp: 1+2} hold,  and $m\in H^{{q-1}}(\T^d)$. Then,  {$\mu^{N,\nu}$ belongs to $\rmC^{1/2,0}([0,T]\times \T^d)$ and for any $t\in[0,T]$, $\mu^{N,\nu}_t$ belongs to  $H^{{q-1}}(\T^d)$}. Moreover, there exists a positive constant $C = C(d,T,\cF,\cG,H,q)$ such that 
	\begin{equation*}
		\sup_{t\in[0,T]}\left(\norm{\mu^{N,\nu}_t}_{2,{\lfloor q-2-d/2 \rfloor}} + \norm{\mu^{N,\nu}_t}_\infty\right) \leq C(d,T,\cF,\cG,H,q)\norm{m}_{2,{\lfloor q-2-d/2 \rfloor}}.
	\end{equation*}
\end{proposition}
\begin{proof}
	Once noticed that $\norm{\mu^{N,\nu}_t*D^N}_2\leq \norm{\mu^{N,\nu}_t}_2$ (which follows from the contraction property of $D^N$ in $L^2({\mathbb T}^d)$), the proof of the {Sobolev regularity in space goes} exactly as the one of Proposition \ref{prop: smooth_muN}. Then, the $L^\infty$ bound follows from Sobolev embedding. {Continuity in time 
	can be also established as in the proof of 
Proposition \ref{prop: smooth_muN}.}
\end{proof}
\addtxtr{\begin{remark}
	Note that in Propositions \ref{prop: smooth_muN} and \ref{prop: tildemu_in_Hp}, bounds are stated with respect to the norm $\norm{\cdot}_{2,\lfloor q-2-d/2 \rfloor}$, even though the initial datum $m$ is required to be more regular (it belongs to $H^{q-1}(\T^d)$). These (uniform in $N$) bounds cannot be improved in the sense that one cannot expect 
better (uniform in $N$) regularity for 
	$\tilde{\mu}^{N,\nu}$ and ${\mu}^{N,\nu}$	(at least with the arguments used in the paper). The limitation comes from the fact that the (uniform in $N$)  bounds
	we have for $\nabla u^{N,\nu}$ just hold true up to the order $\lfloor q-2-d/2 \rfloor$. However, assumptions could be weakened: at this stage, we could just require $m$ to belong to $H^{\lfloor q-2-d/2 \rfloor}(\T^d)$. We would obtain the same results. The only difference is that in this case, $\tilde{\mu}^{N,\nu}_t$ and $\mu^{N,\nu}_t$ would belong to $H^{\lfloor q-2-d/2 \rfloor}(\T^d)$ and not to $H^{q-1}(\T^d)$ (under the assumption $m \in  H^{q-1}(\T^d)$, they do belong to 
	$H^{q-1}(\T^d)$, but not uniformly in $N$). For simplicity, we prefer to keep $m \in H^{q-1}(\T^d)$ as an assumption, 
	because this is the assumption we need in 	
	 Section \ref{sec: convergence}.
\end{remark}}
We have all the ingredients to prove that, under suitable conditions, $\mu^{N,\nu}$ is a probability measure when $N$ is large. The result is obtained by comparing $\mu^{N,\nu}$ with the probability $\tilde{\mu}^{N,\nu}$. 

\begin{proposition}\label{prop: tildemu_prob_positive}
		Let Assumptions \ref{hp: 1+2} and \ref{hp: 3} hold. Then, there exists a positive integer
		$\overline{N} = \overline{N}(d,T,\cF,\cG,H,\addtxtr{q},\addtxtr{\gamma})$
		 (independent of $\nu$ and greater than the former choice of $\widetilde N$), such that for any $N\geq \addtxtr{\overline{N}}$, the solution $\mu^{N,\nu}$ of \eqref{eqn: opt_fp_prox_nu} is a flow of probability measures. Moreover, there exists 
		\addtxtr{$\hat\gamma \in (0,1)$} (independent of $\nu$) such that for any $N \geq  \addtxtr{\overline{N}}$, $\inf_{t\in[0,T]}\mu^{N,\nu}_t\geq\hat\gamma$. 
\end{proposition}
\addtxtr{Before proceeding with the proof of Proposition \ref{prop: tildemu_prob_positive}, let us comment the notation for the threshold $\widetilde{N}$. As recalled above, we are always considering $N\geq \widetilde{N}$, with $\widetilde{N}$ as in Remark \ref{rmk: H_is_useless}. Next, we also need 
the conclusion of Proposition \ref{prop: tildemu_prob_positive} to hold true and thus $N$ to be also greater than $\overline{N}$ (with this $\overline{N}$ 
being in fact already assumed to be larger than $\widetilde N$). To remain consistent with our original notation, we will redefine $\widetilde{N}$ as $\overline{N}$
 and then always consider $N\geq\widetilde{N}$ in the sequel. }
\begin{proof}
	From \eqref{eqn: closed_sys_ODE} it follows that $\widehat{\mu}^{N,\nu}_t(0) = 1$ for every $t\in[0,T]$, so that $\mu^{N,\nu}_t$ has mass equal to $1$. To prove that $\mu^{N,\nu}_t$ is positive, we show that $\norm{\tilde{\mu}^{N,\nu}_t - \mu^{N,\nu}_t}_\infty\to 0$ as $N\to\infty$, where $\tilde{\mu}^{N,\nu}$ is the time-dependent probability measure that solves \eqref{eqn: FP_tilde_nu}. Indeed,  with $p_t(x)$ denoting the standard heat kernel on $\T^d$, it holds that
	\begin{align}
		\tilde{\mu}^{N,\nu}_t& (x) - \mu^{N,\nu}_t (x)\nonumber \\
		&= - \int_0^t \int_{\T^d} \nabla p_{t-s}(x-y) \Big[\diff_p H(y,\nabla u^N_s(y))\left(\tilde{\mu}^{N,\nu}_s(y) - \mu^{N,\nu}_s*D^N(y)\right)\Big]\de y \de s\nonumber\\
		&= - \int_0^t \int_{\T^d} \nabla p_{t-s}(x-y)\diff_p H(y,\nabla u^{N,\nu}_s(y)) \Big[\tilde{\mu}^{N,\nu}_s(y) - \tilde{\mu}^{N,\nu}_s*D^N(y)\Big] \de y \de s\label{eqn: estdiffmu2}\\
		&\quad - \int_0^t \int_{\T^d} \nabla p_{t-s}(x-y)\diff_p H(y,\nabla u^{N,\nu}_s(y))  \Big[\tilde{\mu}^{N,\nu}_s*D^N(y) - \mu^{N,\nu}_s*D^N(y)\Big]\de y \de s\label{eqn: estdiffmu3}.
	\end{align}
	The fact that $\tilde{\mu}^{N,\nu}_t\in H^{{\lfloor q-2-d/2 \rfloor}}(\T^d)$ (with a norm bounded uniformly in time, see Proposition \ref{prop: smooth_muN}), Lemma 
	\ref{lemma: estinftytruncation} and Proposition \ref{prop: infinitybounds} entail
	\begin{align*}
		\lvert\eqref{eqn: estdiffmu2}\rvert&\leq C(d)\int_0^t \frac{1}{\sqrt{t-s}} \norm{\tilde{\mu}^{N,\nu}_s - \tilde{\mu}^{N,\nu}_s*D^N}_\infty \norm{\diff_p H(\cdot,\nabla u^{N,\nu}_s)}_\infty \de s\\
		&\leq   \frac{C(d,T,\cF,\cG,H,q)}{N^{{\lfloor q-2-\frac{d}{2} \rfloor}-\frac{d}{2}}}\int_0^t \frac{1}{\sqrt{t-s}}\norm{\tilde{\mu}^{N,\nu}_s}_{2,{\lfloor q-2-d/2 \rfloor}}\de s\leq  \frac{C(d,T,\cF,\cG,H,q)}{N^{{\lfloor q-2-\frac{d}2 \rfloor}-\frac{d}{2}}}\norm{m}_{2,q-1}.
	\end{align*}
	Finally, from
		Propositions 
	\ref{prop: smooth_muN}
	and
\ref{prop: tildemu_in_Hp}
and
	 Remark \ref{rmk: infinity_est_trunc},
	\begin{align*}
		\lvert\eqref{eqn: estdiffmu3}\rvert&\leq C(d)\int_0^t\frac{1}{\sqrt{s-t}} \norm{\diff_p H(\cdot,\nabla u^{N,\nu}_s)}_\infty \norm{(\tilde{\mu}^{N,\nu}_s - \mu^{N,\nu}_s)*D^N}_\infty\de s\\
		&\leq C(d,T,\cF,\cG,H,q)\left(\frac{\norm{m}_{2,q-1}}{N^{\lfloor q-2-\frac{d}2 \rfloor-\frac{d}2}}+\int_0^t \frac{1}{\sqrt{t-s}}\norm{\tilde{\mu}^{N,\nu}_s - \mu^{N,\nu}_s}_\infty\de s\right),
	\end{align*}
	and we can conclude by a generalized version of Gronwall's inequality (that allows us to handle $1/{\sqrt{t-s}}$, see \cite{yegaoding} for the result in full generality or \cite[Lemma A.1]{daudindelaruejackson} for the particular case used here) that
	\begin{equation*}
		\sup_{t\in[0,T]}\norm{\tilde{\mu}^{N,\nu}_t-\mu^{N,\nu}_t}_\infty\leq \frac{C(d,T,\cF,\cG,H,q)}{N^{{\lfloor q-2-\frac{d}{2} \rfloor}-\frac{d}{2}}}\norm{m}_{2,q-1}.
	\end{equation*}
	Since Lemma \ref{lemma: lower_bound_muN} implies $\inf_{t\in[0,T]}\tilde{\mu}^{N,\nu}_t\geq\tilde\gamma>0$
	and since 
	$\lfloor q-2-d/2 \rfloor-d/2\geq 1/2$
(see the proof of Proposition \ref{prop: smooth_muN}), we have that $\inf_{t\in[0,T]}\mu^{N,\nu}_t\geq\tilde\gamma/2:=\hat\gamma>0$ for $N$ sufficiently large.
\end{proof}

\begin{remark}\label{rmk: mutilde_n_prob}
	As a consequence of Propositions \ref{prop: tildemu_in_Hp} and \ref{prop: tildemu_prob_positive}, we can find $N$ large enough (again we will say $N\geq\widetilde{N}$) and independent of $t\in[0,T]$, such that 
	not only $\mu^{N,\nu}_t$ but
	also $\mu^{N,\nu}_t*D^N$ is a probability measure. Indeed, from Lemma \ref{lemma: estinftytruncation},
	\begin{equation*}
		\mu^{N,\nu}_t*D^N(x) \geq \mu^{N,\nu}_t(x) - \frac{C(d,T,\cF,\cG,H,q)}{N^{{\lfloor q-2-\frac{d}2 \rfloor}-\frac{d}{2}}}\norm{\mu^{N,\nu}_t}_{2,{\lfloor q-2-\frac{d}2 \rfloor}}
		\geq \hat{\gamma} - \frac{C(d,T,\cF,\cG,H,q)}{N^{{\lfloor q-2-\frac{d}2 \rfloor}-\frac{d}{2}}}\norm{m}_{2,q-1}>0,
	\end{equation*}
	for $N$ large enough. The fact that $\mu^{N,\nu}_t*D^N$ is normalized trivially follows from the computation of the Fourier coefficient of order $0$.
\end{remark}
\subsection{Well-posedness of the approximating forward-backward system}\label{sec: fixed_point}
We finally go back to the original forward-backward approximating system \eqref{eqn: nonlocal_HJ_tilde} - \eqref{eqn: opt_fp_prox} (or, equiv. \eqref{eqn: FWBKW_intro}), that we re-state here to have a clearer exposition: 
\begin{equation*}
	\begin{cases}
		\partial_t \mu^N_t &= \Delta\mu^N_t + \div(\diff_p H(\cdot,\nabla u^N_t)(\mu^N_t*D^N)),\quad t\in[0,T],\\
		\smallskip
		\mu^N_0 &= m,\\
		\smallskip
		\partial_ t  u^N_t  &= - \Delta  u^N_t  +  H( \cdot,\nabla  u^N_t)*D^N  - \lf\cF(\mu^N_t*D^N,\cdot)*D^N,\quad t\in[0,T],\\
		\smallskip
		  u^N_T &= \lf\cG(\mu^N_T*D^N,\cdot)*D^N.
	\end{cases}
\end{equation*}
The goal of this subsection is to combine the results of Subsections \ref{sec: nonlocalHJ} and \ref{sec: prop_of_mutilde} with a suitable fixed point argument to conclude that there exists a classical solution $(\mu^N, u^N)$ to \eqref{eqn: FWBKW_intro}, such that both $\mu^N$ and $u^N$ enjoy suitable regularity properties and $\mu^N$ and $\mu^N * D^N$ are flows of probabilities. 
\subsubsection{A fixed-point argument and the existence of a solution to \eqref{eqn: FWBKW_intro}}
Let us stress again the fact that $N\geq\widetilde{N}$
(recall 
Remark 
\ref{rmk: mutilde_n_prob} and Proposition \ref{prop: tildemu_prob_positive}). We define the space $X\subseteq\rmC([0,T];\prob_N(\T^d))$ 
of time-dependent measures $\nu$ with null Fourier coefficients of order higher than $N$ such that, for a certain $C>0$ \addtxtr{(possibly depending on $N$)} and any $s,t\in[0,T]$,
\begin{equation}\label{eqn: lip_cond_fix}
	\norm{\nu_t - \nu_s}_2 = \Bigg (\sum_{\lvert k\rvert\leq N}\lvert \hat{\nu}_t(k) - \hat{\nu}_s(k)\rvert^2\Bigg)^{\frac{1}{2}}\leq C\lvert t-s\rvert^{\frac{1}{2}} .
\end{equation}
We equip $\prob_N(\T^d)$ with the topology induced by the $L^2$ convergence
(which is in fact the same as the topology induced by any standard norm since 
$\prob_N(\T^d)$ is finite-dimensional), and so by Ascoli-Arzelà theorem it follows that $X$ is compact with respect to the uniform norm in time. Moreover an easy calculation shows that $X$ is also convex. Our aim is to build a suitable fixed point argument for a mapping over $X$.
\begin{proposition}\label{prop: ex_uniq_approx_sys}
	Let Assumptions \ref{hp: 1+2} and \ref{hp: 3} be in force. For $N\geq\widetilde{N}$, there exists a unique solution $(\mu^N, u^N)$ to the forward-backward system \eqref{eqn: FWBKW_intro}. More precisely, $u^N$ is a classical solution {and belongs to $\rmC^{1,{q}}([0,T]\times\T^d)$; $\mu^N$ is in $\rmC^{1/2,0}([0,T]\times \T^d)$ and admits a density in $H^{q-1}(\T^d)$.}
	Moreover, 
	there exists a positive constant $C = C(d,T,\cF,\cG,H,q)$ such that
	\begin{align*}
			&\sup_{t\in[0,T]}\norm{\mu^N_t}_{2,{\lfloor q-2-d/2 \rfloor}} \leq C(d,T,\cF,\cG,H,q)\norm{m}_{2,q-1},\\
			&\sup_{t\in[0,T]}\norm{u^N_t}_{\infty,{\lfloor q-1-d/2 \rfloor}}  \leq C(d,T,\cF,\cG,H,q).
	\end{align*}
	Finally, $\mu^N*D^N\in\rmC([0,T];\prob_N(\T^d))$.  
\end{proposition}
\begin{proof}
	Uniqueness is addressed separately in the forthcoming Subsection \ref{sec: uniqueness}. We focus on the existence. Let us fix $\nu\in X$. Then, by Proposition \ref{prop: tildemu_prob_positive} and Remark \ref{rmk: mutilde_n_prob} we have that $\mu^{N,\nu}$ and $\mu^{N,\nu}*D^N$ belong to $\rmC([0,T];\prob(\T^d) \cap L^2({\mathbb T}^d))$ and $\rmC([0,T];\prob_N(\T^d))$ respectively. 
	 Using the form of 	\eqref{eqn: opt_fp_prox_nu} (reformulated in Fourier coefficients)
	and Proposition 
	\ref{prop: infinitybounds}, we see
	 that \eqref{eqn: lip_cond_fix} holds for $\mu^{N,\nu}*D^N$ for a certain constant $C$ independent of $\nu$.
	Thus, for every $N\geq\widetilde{N}$, we can define the mapping $\Phi\colon X\to X$ by $\nu\mapsto \Phi(\nu) := \mu^{N,\nu}*D^N$.\\
	
	We now prove that $\Phi$ is continuous. Let us consider a sequence $\{\nu^\ell\}_{\ell\in\N}\subset X$ such that $\sup_{t\in[0,T]}\norm{\nu^\ell_t - \nu_t}_2\to 0$ as $\ell\to\infty$, for a $\nu\in X$.
	The proof relies on the ODE interpretation of 
	\eqref{eqn: nonlocal_HJ_tilde_nu}--\eqref{eqn: opt_fp_prox}. 
	First, we claim that $\norm{ u^{N,\nu^\ell}_t -  u^{N,\nu}_t}_2\to 0$ as $\ell\to\infty$ and hence $\norm{\nabla u^{N,\nu^\ell}_t - \nabla u^{N,\nu}_t}_2\to 0$ thanks to the relation $\reallywidehat{\nabla u}(k) = \ii 2\pi k\widehat u(k)$, $k\in\Z^d$. 
	Indeed, the Lipschitz property in the measure argument (with respect to the $L^2$ norm and uniformly with respect to $x$) of $\lf\cF$ implies, for $k\in\Z^d$, $\lvert \reallywidehat{ \lf\cF(\nu^\ell_t,\cdot)}(k)-\reallywidehat{ \lf\cF(\nu_t,\cdot)}(k)\rvert\leq C(d,\cF)\norm{\nu_t^\ell - \nu_t}_2$, and similarly for $\lf\cG$. Recalling that $\{\widehat{u}^{N,\nu}(k)\}_{\lvert k\rvert >N}$ is equal to $0$ and that $\{\widehat{u}^{N,\nu}(k)\}_{\lvert k\rvert \leq N}$ solves \eqref{eqn: fourier_utilden_nu}, continuity (in $\nu$) of $u^{N,\nu}$ is a consequence of the Cauchy-Lipschitz 
	stability theory for 
	ODEs \addtxtr{(see the introduction of \S\ref{subsubse:regularity:uNnu}).}
	With this, we can argue in the same way from \eqref{eqn: closed_sys_ODE} and \eqref{eqn: remainders_ODE} and prove that $\norm{\mu^{N,\nu^\ell}-\mu^{N,\nu}}_2\to 0$ as $\ell\to \infty$. 
	\vskip 6pt

	Since we proved that $\Phi\colon X\to X$ is continuous and defined over a compact and convex set, we can apply the Shauder fixed-point theorem and obtain that $\Phi$ has a fixed-point. This proves the existence of a flow of probabilities $\mu^N*D^N$ (i.e., determined by $\{\hat{\mu}^N(k)\}_{\lvert k \rvert \leq N}$) solving 
	the forward equation in 
	\eqref{eqn: FWBKW_intro}, when written in	
	the form 
	\eqref{eqn: closed_sys_ODE}, as well as the existence of a solution $u^N$ to the backward equation, when written in the form 
	\eqref{eqn: fourier_utilden_nu}.  The flow of probabilities $\mu^N$ 
	(i.e., $\{\hat{\mu}^N(k)\}_{\lvert k \rvert > N}$)
	can be easily determined by \eqref{eqn: remainders_ODE}. 
	Finally the estimates in $\rmC^{{\lfloor q-1-d/2 \rfloor}}(\T^d)$ and $H^{{\lfloor q-2-d/2 \rfloor}}(\T^d)$ norms follow from Propositions \ref{prop: infinitybounds} and \ref{prop: tildemu_in_Hp}. 
\end{proof}

\subsubsection{Stability and uniqueness}\label{sec: uniqueness}
Here we provide a stability result for the forward-backward system \eqref{eqn: FWBKW_intro}, which implies the uniqueness property claimed in Proposition \ref{prop: ex_uniq_approx_sys}. In order to have lighter and clearer notation, we omit the superscript $N$ when we refer to a pair $(\mu^N,u^N)$  that solves the system \eqref{eqn: FWBKW_intro} (but this is only in this paragraph that we do so). The proof 
relies on the convexity properties of ${\mathcal F}$ and ${\mathcal G}$ and 
follows the classical duality argument from the mean field game literature 
on monotone forward-backward systems
(see, e.g.~\cite[Lemma 3.1.2]{cardaliaguetdelaruelasrylions}). 
\begin{proposition}
\label{prop:uniqueness}
	Let Assumption \ref{hp: 1+2} hold and $N\geq\widetilde{N}$. Let $\{m^i\}_{i=1,2}$ be two probability measures satisfying Assumption \ref{hp: 3}, and $\{(\mu^i,u^i)\}_{i=1,2}$ be two solutions of the forward-backward system \eqref{eqn: FWBKW_intro} starting respectively  from $\{m^i\}_{i=1,2}$. Then, 
	there exists a constant $C=C(d,T,{\mathcal F},{\mathcal G},H)$ such that 
	 \begin{equation*}
	 	\int_0^T\scalprod{\rvert\nabla u^2_t - \nabla u^1_t\lvert^2}{\mu^1_t*D^N + \mu^2_t*D^N}\de t \leq C\scalprod{m^1-m^2}{u^1_0-u^2_0}.
	 \end{equation*}
In particular, the forward-backward system \eqref{eqn: FWBKW_intro} has a unique solution.
\end{proposition}
\begin{proof}
	Let us consider the two equations satisfied by $(u^1 - u^2)$ and $(\mu^1-\mu^2)$, obtained by subtracting the two versions of \eqref{eqn: FWBKW_intro} with initial conditions $\{m^i\}_{i=1,2}$. Integration by parts leads to 
   	 \begin{align*}
    		\frac{\de}{\de t}\scalprod{\mu^1_t - \mu^2_t}{u^1_t -  u^2_t}
    		&= -\scalprod{\diff_p H(\cdot,\nabla u^1_t) (\mu^1_t*D^N) - \diff_p H(\cdot,\nabla u^2_t)( \mu^2_t*D^N)}{\nabla u^1_t - \nabla u^2_t}\\
    		 & \quad+ \scalprod{ \mu^1_t -  \mu^2_t}{H(\cdot,\nabla u^1_t)*D^N - H(\cdot,\nabla u^2_t)*D^N}\\
    		 & \quad- \scalprod{ \mu^1_t -  \mu^2_t}{\lf\cF( \mu^1_t*D^N,\cdot)*D^N - \lf\cF( \mu^2_t*D^N,\cdot)*D^N} \\
		 & = -\scalprod{\{H(\cdot,\nabla u^2_t)- H(\cdot,\nabla u^1_t) - \diff_p H(\cdot,\nabla u^1_t) \cdot(\nabla u^2_t - \nabla u^1_t)\}}{\mu^1_t*D^N}\\
		 &\quad - \scalprod{\{H(\cdot,\nabla u^1_t)- H(\cdot,\nabla u^2_t) - \diff_p H(\cdot,\nabla u^2_t) \cdot(\nabla u^1_t - \nabla u^2_t)\}}{\mu^2_t*D^N} \\
		 & \quad- \scalprod{ \mu^1_t*D^N -  \mu^2_t*D^N}{\lf\cF( \mu^1_t*D^N,\cdot) - \lf\cF( \mu^2_t*D^N,\cdot)},
	 \end{align*}
	 where we used the symmetry of the Dirichlet kernel to move $D^N$ from one side to the other in the duality brackets. Then, by integrating in time and using the monotonicity of $\lf\cF$ and $\lf\cG$, we obtain 
	 \begin{align*}
	 		\int_0^T\scalprod{\{H(\cdot,\nabla u^2_t)&- H(\cdot,\nabla u^1_t) - \diff_p H(\cdot,\nabla u^1_t) \cdot(\nabla u^2_t - \nabla u^1_t)\}}{\mu^1_t*D^N} \de t\\
		 &\quad + \int_0^T \scalprod{\{H(\cdot,\nabla u^1_t)- H(\cdot,\nabla u^2_t) - \diff_p H(\cdot,\nabla u^2_t) \cdot(\nabla u^1_t - \nabla u^2_t)\}}{\mu^2_t*D^N}\de t\\
		& \leq \scalprod{m^1-m^2}{u^1_0-u^2_0}.
	 \end{align*}
	 The strong convexity of $H$ implies, for a positive $C$, the desired stability property: 
	 \begin{equation*}
	 	\int_0^T\scalprod{\rvert\nabla u^2_t - \nabla u^1_t\lvert^2}{\mu^1_t*D^N + \mu^2_t*D^N}\de t \leq C\scalprod{m^1-m^2}{u^1_0-u^2_0}.
	 \end{equation*}
	 In particular, if we consider $m^1 = m^2$, the right-hand side in the above estimate vanishes. Since $\mu^1_t*D^N$ and $\mu^2_t*D^N$ are measures with positive densities
	 (see Proposition \ref{prop: tildemu_prob_positive}), we have that $\nabla u^1 =\nabla u^2$ in $(0,T)\times\T^d$.  We can then realize that $\mu^1 = \mu^2$, since they solve the same Fokker-Planck equation. In the same manner we conclude that $u^1 = u^2$, and so we have obtained uniqueness for the forward-backward system \eqref{eqn: FWBKW_intro}. 
\end{proof}
\subsubsection{A nonlinear system of ODEs for the Fourier coefficients of $(\mu^N, u^N)$}\label{sec: fourier_properties}
The purpose of this final paragraph on the approximating forward-backward system is to collect the properties of the Fourier coefficients of its unique solution $(\mu^N, u^N)$. The results we present here are merely a restatement of the results we obtained earlier in the initial part of Subsections \ref{sec: nonlocalHJ} and \ref{sec: prop_of_mutilde}, with the additional insight provided by the fixed-point argument and the existence result established in Proposition \ref{prop: ex_uniq_approx_sys}.\\\\
First, let us recall that due to the presence of the Dirichlet kernel, the solution to \eqref{eqn: nonlocal_HJ_tilde} has null Fourier coefficients of order greater than $N$. On the other hand, the coefficients $\{\widehat{u}^N(k)\}_{\lvert k \rvert\leq N}$ can be computed by solving the system of nonlinear ODEs
\begin{equation}\label{eqn: fourier_utilden}
\begin{cases}
	\frac{\de}{\de t} \widehat{ u}^N_t(k) &= \lvert2\pi k\rvert^2  \widehat{ u}^N_t(k) - \reallywidehat{ \lf\cF(\mu^N_t*D^N,\cdot)}(k) + \reallywidehat{H(\cdot,\nabla  u^N_t)}(k),\quad\lvert k\rvert\leq N\\
	\quad\\
	\widehat{  u}^N_T (k)& = \reallywidehat{ \lf\cG(\mu^N_T*D^N,\cdot)}(k),
\end{cases}
\end{equation}
	where we notice that $\nabla u^N = \ii 2\pi\sum_{\lvert j\rvert\leq N} j \widehat{u}^N(j)e_j$. 
 Regarding the solution to \eqref{eqn: opt_fp_prox}, the first Fourier coefficients $\{\widehat{\mu}^N(k)\}_{\lvert k\rvert\leq N}$ can be computed by solving the system of ODEs
 \begin{equation}\label{eqn: closed_sys_ODE_final}
	\begin{cases}
		\frac{\de}{\de t} \widehat{\mu}^N_t(k) &= -\lvert 2\pi k \rvert^2 \widehat{\mu}^N_t(k) + \ii 2\pi k\cdot \sum_{\lvert k - l\rvert\leq N} \reallywidehat{\diff_p H(\cdot,\nabla u^N_t)}(l)\widehat{\mu}^N_t(k-l),\quad \lvert k \rvert \leq N,\\
		\quad\\
		 \widehat{\mu}^N_0(k) &= \widehat m (k),
	\end{cases}
\end{equation}
whilst the remaining ones $\{\widehat{\mu}^N(k)\}_{\lvert k\rvert> N}$ are obtained from
\begin{equation}\label{eqn: remainders_ODE_final}
	\begin{cases}
		\frac{\de}{\de t} \widehat{ \mu}^N_t(k) & = -\lvert2\pi k \rvert^2 \widehat{ \mu}^N_t(k) +  \ii 2\pi k\cdot \sum_{\lvert k - l\rvert\leq N} \reallywidehat{\diff_p H(\cdot,\nabla u^N_t)}(l)\widehat{ \mu}^N_t(k-l),\quad \lvert k\rvert>N,\\
		\quad\\
		\widehat{ \mu}^N_0(k) & = \widehat m (k),
	\end{cases}
\end{equation}
and we notice that the forcing term in \eqref{eqn: remainders_ODE_final} depends only upon the coefficients up to order $N$.\\ 

\noindent Looking more carefully at \eqref{eqn: fourier_utilden}, we can observe that only $ \mu^N*D^N$ is involved in the equation. This means that we only need the Fourier coefficients $\{\widehat{ \mu}^N(j)\}_{\lvert j\rvert\leq N}$ to study \eqref{eqn: fourier_utilden}. Moreover, the nonlinear term $\reallywidehat{H(\cdot,\nabla  u^N_t)}(k)$ depends on $\nabla  u^N$, which can be expressed in terms of $\{\widehat{ u}^N(j)\}_{\lvert j\rvert\leq N}$. The same holds for the term involving $\nabla u^N$ in \eqref{eqn: closed_sys_ODE_final}. Thus, we can join \eqref{eqn: fourier_utilden} and \eqref{eqn: closed_sys_ODE_final} to obtain a nonlinear system of ODEs in closed form which involves only the finite sequences of Fourier coefficients $\{\widehat{ \mu}^N(j)\}_{\lvert j\rvert\leq N}$ and $\{\widehat{ u}^N(j)\}_{\lvert j\rvert\leq N}$. We collect this observation, as well as the previous comment, in the following theorem:
\begin{theorem}\label{thm: fourier_magical_properties}
	Let Assumptions \ref{hp: 1+2} and \ref{hp: 3} hold, and let $N\geq\widetilde{N}$. Then, the Fourier coefficients  $\{(\hat{\mu}^N(k), \hat{ u}^N(k)\}_{\lvert k\rvert \leq N}$ of the unique solution $(\mu^N,u^N)$ to \eqref{eqn: FWBKW_intro} are uniquely determined by the nonlinear system of ODEs  \eqref{eqn: fourier_utilden}--\eqref{eqn: closed_sys_ODE_final}, which is in closed form.
	Moreover, $\{ \hat{ u}^N(k)\}_{\lvert k\rvert > N}$ are identically null, whilst $\{\hat{\mu}^N(k)\}_{\lvert k\rvert > N}$ solve the linear system \eqref{eqn: remainders_ODE_final}.
\end{theorem}
The result stated in Theorem \ref{thm: fourier_magical_properties}  is one of the most important feature of the approximating forward-backward system \eqref{eqn: FWBKW_intro}. Indeed, it provides an explicit rule to compute the Fourier coefficients of order $\lvert k \rvert \leq N$ by solving a finite dimensional system of ODEs. This, combined with the convergence results obtained in the forthcoming Section \ref{sec: convergence}, gives a way to reduce the infinite dimensional MFC control problem to a finite dimensional one.
\begin{remark}\label{rmk: first_on_init_cond}
	Throughout our discussion on the forward-backward system \eqref{eqn: FWBKW_intro}, and particularly for  the approximated Fokker-Planck equation \eqref{eqn: opt_fp_prox}, we have considered an initial condition $m$ that satisfies Assumption \ref{hp: 3} and may have an infinite number of non-zero Fourier coefficients. However, in the spirit of the comment above, we can also consider an initial condition with a finite number of non-zero Fourier coefficients (for instance, with coefficients of order greater than $N$ being null). With this in mind, one can consider the original MFC problem and the system \eqref{eqn: fwdbkw_infty_intro} starting from a general $m$ satisfying Assumption \ref{hp: 3}, and then look at the approximating system starting from $m*D^N$ (which remains a  strictly positive density  as long as we take $N$ sufficiently large).
\end{remark}
\begin{example}
	Let us consider the case of a quadratic Hamiltonian, that is $H(x,p) = \frac{1}{2}\lvert p\rvert^2$. Since $\lvert p\rvert^2 = p\cdot p$, we can use the discrete convolution to compute $\reallywidehat{H(\cdot,\nabla  u^N_t)}(k)$.  
	Then, the ODE in \eqref{eqn: fourier_utilden} reads, for any $\lvert k \rvert\leq N$,
	\begin{equation*}
			\frac{\de}{\de t} \widehat{ u}^N_t(k)  = \lvert2\pi k\rvert^2 \widehat{ u}^N_t(k)  - \reallywidehat{ \lf\cF( \mu^N_t*D^N,\cdot)}(k) - \addtxtr{2 \pi^2} \sum_{\substack{\lvert k-l\rvert\leq N \\ \lvert l \rvert \leq N}}[ (k-l) \widehat{ u}^N_t(k-l)]\cdot[l\  \widehat{ u}^N_t(l)].
	\end{equation*}
\end{example}

\subsection{Characterization of the optimal control for the approximating MFC problem}\label{sec: Pontryagin}
We have all the tools to rigorously prove that the pair $(\mu^N, u^N)$, which is the unique solution to \eqref{eqn: FWBKW_intro}, is optimal for the approximating control problem \eqref{eqn: value_N} (with $t=0$). More precisely, the optimal feedback control is unique and given by $\alpha^{*,N} = -\diff_pH(\cdot,\nabla u^N)$, with $\mu^N$ as 
 optimal trajectory. Moreover, we can also introduce an equivalent formulation of the approximating MFC problem as a finite-dimensional control problem.
 
\subsubsection{Existence and uniqueness of the optimal control} \label{sec: pontryagin_new}
Intuitively, the system \eqref{eqn: FWBKW_intro} 
should be regarded as the Pontryagin system 
associated with the approximating MFC problem \eqref{eqn: value_N}. 
The following result clarifies this fact: 
using convexity of $\cF$ and $\cG$, we show that $\alpha^{*,N} = -\diff_p H(\cdot,\nabla u^N)$ is indeed the unique minimizer of \eqref{eqn: value_N}. The proof follows the one of \cite[Proposition 3.7.2]{cardaliaguetdelaruelasrylions}.
\begin{proposition}
\label{prop: optimality}
	Let Assumptions \ref{hp: 1+2} and \ref{hp: 3} hold, and let $N\geq\widetilde{N}$. Then, the control $\alpha^{*,N} = -\diff_p H(\cdot,\nabla u^N)$ is the unique minimizer of the approximating MFC problem  \eqref{eqn: value_N}, where $(\mu^N, u^N)$ is the unique solution to \eqref{eqn: FWBKW_intro}.
\end{proposition}
\begin{proof}
{Let $\check{\alpha}^N$ be a feedback control for \eqref{eqn: value_N} such that the associated trajectory $\check{\mu}^N*D^N$ is a flow of probability measures.} Let us recall that $L(x,\alpha) = \sup_{p\in\R^d}\left ( - \alpha\cdot p - H(x,p)\right)$ 
 is uniformly strictly convex in $\alpha$ thanks to the uniform strict convexity of $H$ in the momentum $p$.  
Therefore, there exists a constant $c>0$ such that, for any $(x,\alpha) \in {\mathbb T}^d \times {\mathbb R}^d$, 
\begin{equation*} 
\begin{split}
L(x,\alpha) &\geq L( x, \alpha^{*,N}(t,x) ) 
 +\partial_\alpha L
( x, \alpha^{*,N}_t(x) ) \cdot ( \alpha -  \alpha^{*,N}_t(x) ) + c \bigl\vert 
\alpha 
- 
 \alpha^{*,N}_t(x) \bigr\vert^2
 \\
&= - \nabla u^N_t(x) \cdot \alpha 
 +
 L( x, \alpha^{*,N}(t,x) ) 
 +
 \nabla u^N_t(x)  \cdot    \alpha^{*,N}_t(x) 
 + c \bigl\vert 
\alpha 
- 
 \alpha^{*,N}_t(x) \bigr\vert^2
 \\
 &=- \nabla u^N_t(x) \cdot \alpha 
 -
 H( x, 
 \nabla u^N_t(x) ) 
 + c \bigl\vert 
\alpha 
- 
 \alpha^{*,N}_t(x) \bigr\vert^2, 
\end{split} 
\end{equation*} 
where we used the following two identities: 
\begin{equation*}
\partial_\alpha L
( x, \alpha^{*,N}_t(x) )
= - \nabla u^N_t(x), \quad  -\alpha_t^{*,N}(x) \cdot\nabla  u^N_t(x) - H(x,\nabla u^N_t(x)) = L(x,\alpha^{*,N}_t(x)).
\end{equation*} 
Then, combining this with the convexity of $\cF$ and $\cG$, we get
	\begin{align*}
		J^N(\check{\alpha}^N,0,m)& \geq \int_0^T \int_{\T^d}\Big[- \check{\alpha}^N_s(x)\cdot \nabla u^N_s(x) - H(x,\nabla u^N_s(x))\Big] (\check{\mu}^N_s*D^N)(\!\de x) \de s\\
		&\quad + \int_0^T \bigg\{\cF(\mu^N_s*D^N) + \int_{\T^d}\lf\cF(\mu^N_s*D^N,x)\Big[\check{\mu}^N_s*D^N - \mu^N_s*D^N\Big](\!\de x)\bigg\}\de s\\
		&\quad + \cG(\mu^N_T*D^N) + \int_{\T^d}\lf\cG(\mu^N_T*D^N,x)\Big[\check{\mu}^N_T*D^N - \mu^N_T*D^N\Big](\!\de x)
		\\
		&\quad + c \int_0^T \int_{{\mathbb T}^d} \bigl\vert 
\check{\alpha}^N_t(x)
- 
 \alpha^{*,N}_t(x) \bigr\vert^2
		(\mu^N_s*D^N) (\!\de x)\de s. 
	\end{align*}
	Using the identity $-\alpha^{*,N}\cdot\nabla  u^N - H(\cdot,\nabla u^N) = L(\cdot,\alpha^{*,N})$, it follows
	\begin{align}
		J^N(\check{\alpha}^N,0,m)& \geq J^N(\alpha^{*,N},0,m)+ c \int_0^T \int_{{\mathbb T}^d} \bigl\vert 
\check{\alpha}^N(x)
- 
 \alpha^{*,N}_t(x) \bigr\vert^2
		(\check \mu^N_s*D^N) (\!\de x)\de s
		\nonumber\\
		&\quad - \int_0^T \int_{\T^d}\nabla u^N_s(x)\cdot\Big[\check{\alpha}^N(x)(\check{\mu}^N_s*D^N) - \alpha^{*,N}(x)(\mu^N_s*D^N)\Big] (\!\de x)\de s\label{eqn: opt_1}\\
		&\quad - \int_0^T \int_{\T^d}H(x,\nabla u^N_s(x)) \Big[\check{\mu}^N_s*D^N - \mu^N_s*D^N\Big](\!\de x) \de s\label{eqn: opt_2}\\
		&\quad + \int_0^T \int_{\T^d}\lf\cF(\mu^N_s*D^N,x)\Big[\check{\mu}^N_s*D^N - \mu^N_s*D^N\Big](\!\de x)\de s\label{eqn: opt_3}\\
		&\quad + \int_{\T^d}\lf\cG(\mu^N_T*D^N,x)\Big[\check{\mu}^N_T*D^N - \mu^N_T*D^N\Big](\!\de x).\label{eqn: opt_4}
	\end{align}
	To conclude, let us prove that $\eqref{eqn: opt_1}+ \eqref{eqn: opt_2}+ \eqref{eqn: opt_3}+ \eqref{eqn: opt_4} = 0$. By using the fact that $\mu^N$ and $\check{\mu}^N$ solve \eqref{eqn: fp_n_OC}, together with integration by  parts, we get
	\begin{align*}
		 \eqref{eqn: opt_1} &= \int_0^T\int_{\T^d} u^N_s(x)\div\Big\{\check{\alpha}^N(x)\big(\check{\mu}^N_s*D^N(x)\big) - \alpha^{*,N}(x)\big(\mu^N_s*D^N(x)\big)\Big\}\de x\de s\\
		 & = \bigg[ \int_{\T^d} u^N_s (x)\big(\mu^N_s(\!\de x) - \check{\mu}_s^N(\!\de x)\big)\bigg]_0^T 
		   \addtxtr{+} \int_0^T\int_{\T^d}\Big(\partial_s u^N_s(x) + \Delta u^N_s(x)\Big)\Big[\check{\mu}^N_s - \mu^N_s\Big](\!\de x) \de s.
	\end{align*}
	Then, since $u^N$ solves \eqref{eqn: nonlocal_HJ_tilde} and by the symmetry of the Dirichlet kernel $D^N$, we obtain
	\begin{align*}
		 \eqref{eqn: opt_1} &= - \int_{\T^d}\lf\cG(\mu^N_T*D^N,x)\Big[\check{\mu}^N_T*D^N - \mu^N_T*D^N\Big](\!\de x)\\
		&\quad - \int_0^T \int_{\T^d}\lf\cF(\mu^N_s*D^N,x)\Big[\check{\mu}^N_s*D^N - \mu^N_s*D^N\Big](\!\de x)\de s\\
		& \quad + \int_0^T \int_{\T^d}H(x,\nabla u^N_s(x)) \Big[\check{\mu}^N_s*D^N - \mu^N_s*D^N\Big](\!\de x) \de s,
	\end{align*}
	which proves that
	$\eqref{eqn: opt_1}+ \eqref{eqn: opt_2}+ \eqref{eqn: opt_3}+ \eqref{eqn: opt_4}$ is indeed equal to $0$ and therefore that 
	$\alpha^{*,N}$ is optimal. Furthermore, if $\check \alpha^N$ is also optimal, then 
	\begin{equation*}
	 \int_0^T \int_{{\mathbb T}^d} \bigl\vert 
\check{\alpha}^N_s(x) 
- 
 \alpha^{*,N}_s(x) \bigr\vert^2
		(\check \mu^N_s*D^N) (\!\de x)\de s = 0. 
\end{equation*} 
Recalling that, for each $s \in [0,T]$, $\check \mu^N_s*D^N$ is assumed to be a probability density, we notice that it must be strictly positive 
except maybe at a finite number of zeros (as a non-trivial trigonometric polynomial function has at most a finite number of  zeros). Therefore, we get 
that, for almost every $s \in [0,T]$, $\check{\alpha}^N_s$ and 
$\alpha^{*,N}_s * D^N$ are almost everywhere equal. This completes the proof. 
\end{proof}

\subsubsection{The approximating MFC problem as a finite dimensional control problem}\label{sec: finite_dum_OC_prob}
Here we continue the discussion started in Subsection \ref{sec: fourier_properties} about the Fourier expansion of the optimal pair $(\mu^N,u^N)$. Let us go back to the original formulation of our approximating MFC problem stated in Subsection \ref{sec: approx_MFC}, with $t=0$. 
For an initial condition 
$m \in\prob(\T^d)$ such that 
$m*D^N \in {\mathcal P}_N({\mathbb T}^d)$, the cost functional to be minimized is
\begin{equation}\label{eqn: cost_new}
	J^N(\alpha,0,m)= \cG(\mu^N_T*D^N) + \int_0^T \left\{\cF(\mu^N_s*D^N) + \int_{\T^d}L(x,\alpha_s(x)) (\mu^N_s*D^N)(\!\de x)\right\}\de s,
\end{equation}
over the set of bounded and Borel measurable feedback controls $\alpha : [0,T] \times {\mathbb T}^d \rightarrow {\mathbb R}^d$ for which the flow 
$(\mu^N_t * D^N)_{0 \le t \le T}$ takes values in 
${\mathcal P}({\mathbb T}^d)$
(or equivalently in 
${\mathcal P}_N({\mathbb T}^d)$).  
From Subsection \ref{sec: pontryagin_new}, we know that the optimal control must be $\alpha^{*,N}(\cdot) =  -\diff_p H( \cdot,\nabla  u^N(\cdot))$, where $u^N$ solves \eqref{eqn: nonlocal_HJ_tilde} and has null Fourier coefficients of order $\lvert k\rvert > N$ (see Theorem \ref{thm: fourier_magical_properties}). Moreover, from \eqref{eqn: cost_new} it is clear that,
for an arbitrary control $\alpha$, 
only the first $\lvert k \rvert \leq N$ Fourier coefficients of the state variable $\mu^N$ are involved in the optimization problem, and they can be computed through the analogue of the closed system of linear ODEs \eqref{eqn: closed_sys_ODE_final}
(with 
the Fourier coefficients
$\reallywidehat{\diff_p H( \cdot,\nabla  u^N)}(k-l)$ being replaced by $-\widehat{\alpha}(k-l)$). In particular, we have $J^N(\alpha,0,m) = J^N(\alpha,0,m*D^N)$.

Thus, it becomes natural to introduce the following finite-dimensional control problem. Let us introduce the set ${\mathscr O}_N$ of vectors of complex numbers $\mathbf{z} = \{z_k\}_{\lvert k \rvert\leq N}$
(with $z_k \in {\mathbb C}$ for each $k$)  
 which are Fourier coefficients of a probability measure in $\prob_N(\T^d)$ (a precise definition and a characterization of this set can be found in \cite[Section 3.1]{cecchindelarue}).  For $\mathbf{z}\in{\mathscr O}_N$, we consider the system of controlled ODEs
 \begin{equation}\label{eqn: closed_sys_ODE_OC}
	\begin{cases}
		\frac{\de}{\de t} \xi^N_t(k) &= -\lvert 2\pi k \rvert^2 \xi^N_t(k) - \textrm{\rm i} 2\pi k\cdot \sum_{\lvert k - l\rvert\leq N} \widehat{\alpha}_t(l)\xi^N_t(k-l),\quad \lvert k \rvert \leq N,\\
		\quad\\
		\xi^N_0(k) &= z_k,
	\end{cases}
\end{equation}
where $\{ (\xi_t(k))_{0 \le t \le T} \}_{\lvert k\rvert\leq N}$ are $\C$-valued functions, and we associate with it the cost functional
\begin{equation*}
\begin{aligned}
	&\Gamma^N(\alpha,0,\mathbf{z})\\
	&= \cG\Big(\sum_{\lvert k\rvert\leq N}\xi^N_{T}(k)e_k\Big) + \int_0^T \bigg\{\cF\Big(\sum_{\lvert k\rvert\leq N}\xi^N_{s}(k)e_k\Big) + \int_{\T^d}L\big(x,\alpha_s(x)\big) \Big(\sum_{\lvert k\rvert\leq N}\xi^N_{s}(k)e_k\Big)(\!\de x)\bigg\}\de s.
\end{aligned}
\end{equation*}
We claim that the infimum of 
$\Gamma^N(\alpha,0,\mathbf{z})$ over the class of bounded and measurable feedback controls $\alpha$ for which $(\{\xi^N_t(k)\}_{\vert k \vert \leq N})_{0 \le t \le T}$ takes values in ${\mathscr O}_N$
is equal to 
$V^N(t,m)$ in  \eqref{eqn: value_N},
where $m := \sum_{\vert k \vert \leq N} z_k e_k \in {\mathcal P}_N({\mathbb T}^d)$. 
The proof is as follows. We can extend the collection $\{ (\xi_t(k))_{0 \le t \le T}\}_{\vert k \vert \leq N}$ to indices 
$k$ such that $\vert k \vert >N$, by letting 
\begin{align*}
\frac{\de}{\de t} \xi^N_t(k) = -\lvert 2\pi k \rvert^2 \xi^N_t(k)  - \textrm{\rm i} 2\pi k\cdot \sum_{\lvert k - l\rvert\leq N} \widehat{\alpha}_t(l)\xi^N_t(k-l),\quad \lvert k \rvert > N. 
\end{align*}
It is pretty obvious to see that $\overline{\xi_t^N(k)} = \xi_t^N(-k)$ (with 
$\bar{z}$ denoting the conjugate of $z$, for  $z \in {\mathbb C}$), from which we deduce that, for each $t \in [0,T]$, 
\begin{equation*} 
\check{\mu}^N_t := \sum_{k \in {\mathbb Z}^d} \xi_t^N(k) e_k
\end{equation*}
is a real-valued Schwartz distribution (this would be possible to find an index ${\mathfrak s}>0$, depending on the dimension 
$d$, such that $\check{\mu}^N_t \in  H^{2,-{\mathfrak s}}({\mathbb T}^d)$ for any $t \in [0,T]$, but there is no need for us to push the analysis further in this direction). 
In this manner, we can see 
$\{\xi^N_t(k)\}_{\vert k \vert \leq N}$ as the lower Fourier coefficients of 
$\check{\mu}^N_t$
and hence identify $\Gamma^N(\alpha,0,{\mathbf z})$
with $J^N(\alpha,0,m)$. 
In particular, the minimum of $\Gamma^N(\alpha,0,\mathbf{z})$
is attained at 
$(\alpha_t^{*,N} = -\diff_p H(\cdot,\textrm{\rm i} 2\pi\sum_{\lvert l\rvert\leq N}l\widehat{u}^N_t(l)e_l))_{0 \le t \le T}$
and the latter is the unique minimizer, 
with $\{ (\widehat{u}^N_t(l))_{0 \le t \le T} \}_{\vert l \vert \leq N}$ solving 
\eqref{eqn: closed_sys_ODE_OC}. 
The optimal trajectory is thus given by 
$\{ (\xi^N_t(k))_{0 \le t \le T} \}_{\vert k \vert \leq N}
=
\{ (\widehat{\mu}^N_t(k))_{0 \le t \le T} \}_{\vert k \vert \leq N}
$, see \eqref{eqn: closed_sys_ODE_final}. 

To sum-up, if we denote by $\Gamma^N\left(\alpha,t,\textbf{z}\right)$ the cost functional for the finite-dimensional problem starting at $t\in[0,T]$ instead of $0$ and we introduce the value function
\begin{equation}\label{eqn: value_finit_dim}
	\Upsilon^N(t,\mathbf{z}):=\inf \Gamma^N\left(\alpha,t,\textbf{z}\right),\quad (t,\mathbf{z})\in[0,T]\times{\mathscr O}_N,
\end{equation}
with the infimum being taken over the same feedback functions $\alpha$ as before (except that 
$\alpha$ is defined on $[t,T] \times {\mathbb T}^d$), 
we have $V^N(t,m) = V^N(t,m*D^N) = \Upsilon^N (t,\{ z_k \}_{\lvert k\rvert\leq N} )$ 
for $m = \sum_{\vert k \vert \leq N} z_k e_k \in {\mathcal P}_N({\mathbb T}^d)$.
Remarkably, the state variable 
in the controlled dynamics 
\eqref{eqn: closed_sys_ODE_OC}
is finite-dimensional. 
Also, given the shape of the optimal feedback function $\alpha^{*,N}$, we can restrict the choice of 
the controls $\alpha$ (which are \textit{a priori} infinite-dimensional) to the 
subclass of time measurable functions taking values in the 
parametric (hence finite-dimensional) set
\begin{equation*} 
{\mathscr A}_N := 
\Bigl\{ 
a = - {\rm D}_p H \Bigl( \cdot, \textrm{\rm i} 2 \pi \sum_{\vert l \vert \leq N} l 
\zeta(l) e_l \Bigr), 
\ \  
\textrm{\rm with} 
\ \ 
\zeta(l)
\in 
{\mathbb C} 
:
\overline{\zeta(l)} = 
\zeta(-l), \ \ \vert l \vert \leq N \Bigr\}. 
\end{equation*}
We then call ${\mathcal A}_N$ 
the set of measurable 
functions from $[t,T]$ to ${\mathscr A}_N$
 (for simplicity, the index $t$ is omitted in the notation ${\mathcal A}_N$). 

The following lemma provides a sufficient (and explicit) condition (on $\alpha$) under which 
the solution to 
\eqref{eqn: closed_sys_ODE_OC} takes values in 
${\mathscr O}_N$:

\begin{lemma}
\label{lem:sufficient:condition:ON}
For fixed 
\addtxtr{constants $c, C>0$ and $\gamma \in (0,1)$},
consider $\alpha=(\alpha_s)_{t \leq s \leq T} \in {\mathcal A}_N$ such that
the weights $\{(\zeta_s(l))_{t \leq s \leq T}\}_{\vert l \vert \leq N}$ 
in the parametric representation of $\alpha$ 
satisfy
\begin{equation} 
\label{eq:sufficient:condition:ON}
\sum_{\vert l \vert \leq N} \vert l \vert^{{2d+6}} \vert \zeta_s(l) \vert^2 \leq C, \quad s \in [t,T].
\end{equation} 
Then, 
there exists
an integer $\widetilde N$, depending on \addtxtr{$c$}, $C$ and $\gamma$, such that
for $N \geq \widetilde N$
and
$m := \sum_{\vert k \vert \leq N} z_k e_k \in {\mathcal P}_N({\mathbb T}^d)$
satisfying  
$m \geq \gamma$
and \addtxtr{$\sum_{\vert k \vert \leq N}  \vert z_k\vert^2 
( 1 + \vert k \vert^2)^{q-1} 
\leq c$},
 the solution to 
\eqref{eqn: closed_sys_ODE_OC} 
(starting at time $t$) 
takes values in 
${\mathscr O}_N$.
\end{lemma}

\begin{proof} 
We let 
\begin{equation*} 
v_s : = \sum_{\vert l \vert \leq N  \addtxtr{ : \, l \not = 0}}  \zeta_s(l) e_l, \quad s \in [t,T]. 
\end{equation*}
By assumption, \addtxtr{it is bounded in  
 $H^{d+3}({\mathbb T}^d)$ independently of $N$}.   
By Sobolev embedding, 
\begin{equation*} 
\sup_{s \in [t,T]} \| v_s \|_{{\infty,r}}  \leq \widetilde C, 
\end{equation*}  
for some $r\geq \lfloor d/2+2 \rfloor$ and some $\widetilde C$ depending on $C$. 
The result then follows from the same analysis as the one developed in 
Proposition 
\ref{prop: tildemu_prob_positive} and 
Remark \ref{rmk: mutilde_n_prob}
(in short, $ \lfloor d/2+2 \rfloor$ corresponds to 
the regularity order of $u^{N,\nu}$ in the latter two statements when $q$ is equal to the minimal threshold $d+3$\addtxtr{; here, the role of 
$\nabla u^{N,\nu}_s$ is played by $\nabla v_s$ and, in fact, the zero-order Fourier mode of $v_s$ does not matter}). 
\end{proof}
Of course, Proposition 
\ref{prop: ex_uniq_approx_sys}
says that the 
optimal feedback 
$\alpha^{*,N}$ 
satisfies 
\eqref{eq:sufficient:condition:ON} when $q \geq 3d/2+4$
and $C$ is sufficiently large (with respect to
the value of  
$\sup_{t \leq s \leq T} \| u^N_s \|_{q,\infty}$). 
In this way, we have characterized 
the optimal feedback $\alpha^{*,N}$ as the 
solution of 
the control problem
\eqref{eqn: value_finit_dim} defined 
over \addtxtr{controls in ${\mathcal A}_N$ satisfying 
\eqref{eq:sufficient:condition:ON}}: both the state-space ${\mathscr O}_N$
and the control-space ${\mathscr A}_N$ (with the corresponding version of the condition 
\eqref{eq:sufficient:condition:ON})
 are finite dimensional. 

\subsubsection{The HJB equation associated to the approximating MFC problem}
To conclude this section, we discuss (at least at a formal level) the Hamilton-Jacobi-Bellman equation associated 
\addtxtr{with} the approximating MFC problem \eqref{eqn: value_N} and the equivalent finite-dimensional optimal control problem \eqref{eqn: value_finit_dim}. First, following the usual steps based on the dynamic programming principle, we obtain the following HJB equation \addtxtr{for \eqref{eqn: value_N}:}
 \begin{equation}\label{eqn: HJB_N}
\begin{cases}
	\medskip
	\partial_t V^N(t,m)
	+ \int_{\T^d} \Delta [\lf V^N](t,m,x)m(\!\de x)\\
	\quad\quad- \int_{\T^d}\left(H(\cdot,\nabla[\lf V^N](t,m,\cdot))*D^N\right)(x) m(\!\de x)
	+ \cF(m*D^N) = 0,\\
	\quad\\
	 V^N(T,m)  = \cG(m*D^N),
\end{cases}
\end{equation}
on a suitable subset 
$\cB$ of $[0,T] \times \prob(\T^d)$.
Although 
the form of 
\eqref{eqn: HJB_N}
is \addtxtr{fairly predictible}
(because it is very similar to the HJB equation for the original MFC problem),
the choice of the domain $\cB$ raises interesting questions.  
\addtxtr{In line} with the 
statement
of 
Proposition \ref{prop: tildemu_prob_positive},
one should restrict 
\eqref{eqn: HJB_N}
to pairs $(t,m)$ 
for which $m*D^N$ is 
bounded from below by a strictly positive constant \addtxtr{and has a Sobolev norm bounded by another fixed constant}. 
The difficulty \addtxtr{arises} from the fact that the lower bound
(which is a constraint on the initial condition $(t,m)$) 
may not be preserved by the Fokker-Planck dynamics 
\eqref{eqn: fp_n_OC}, even within the setting of 
Proposition \ref{prop: tildemu_prob_positive}. 
\addtxtr{While} this has a limited (mostly technical) impact \addtxtr{on} the formal derivation of the HJB equation (the proof usually \addtxtr{consists of} a verification argument based upon 
the dynamic programming principle), this raises more challenging questions about the well-posedness 
\addtxtr{of} the above equation,
\addtxtr{because} 
the restriction of the domain to the smaller subset 
${\mathcal B}$ of $[0,T] \times {\mathcal P}({\mathbb T^d})$ creates a boundary. 

\addtxtr{To overcome this drawback and free ourselves from the influence of the boundary, we would ideally have to choose a set ${\mathcal B}$ that is preserved by the optimal trajectories, whatever the choice of the initial condition.}
However, the \addtxtr{actual} benefit \addtxtr{may} be rather limited from a practical point of view, \addtxtr{particularly} if the definition 
of such a set ${\mathcal B}$ is not \addtxtr{sufficiently} explicit: for \addtxtr{example}, 
such a domain 
may be difficult to \addtxtr{deal with} from a numerical point of view if it is not clearly (or simply) defined. 
Instead, we thus propose a cheaper strategy, which we plan to address in detail in a future work. 
We \addtxtr{argue} that from a practical point of view, 
it is indeed \addtxtr{sufficient} to choose 
${\mathcal B}$ 
\addtxtr{as the subset of a wider 
set $\widetilde{\mathcal B} \subset {\mathscr O}_N$
such that the boundary $\partial \widetilde{\mathcal B}$ of $\widetilde{\mathcal B}$ is unreachable \addtxtr{from ${\mathcal B}$}, and this whatever the choice 
of the strategy $\alpha \in {\mathcal A}_N$ satisfying 
\eqref{eq:sufficient:condition:ON}}. In our setting, we can revisit 
Lemma 
\ref{lemma: lower_bound_muN}, \addtxtr{Propositions
\ref{prop: tildemu_prob_positive} 
and
\ref{prop: ex_uniq_approx_sys}}
quite easily 
to get the following 
\addtxtr{reformulation of Lemma 
\ref{lem:sufficient:condition:ON}}
 (in order to ease the presentation, we directly write down the 
new
statement 
for the 
finite-dimensional 
version formulated in 
\eqref{eqn: value_finit_dim}, but there is no difficulty in formulating 
a similar statement for 
\eqref{eqn: value_N}):
\begin{lemma}\label{lemma: lower_bound_muN:2}
	Let Assumption \ref{hp: 1+2} be in force. Then,
	\addtxtr{for any non-negative constants $c$, $C$ and $\gamma \in (0,1)$, 
	there exist two constants $\tilde \gamma \in (0,1)$ and 
	$\tilde c \geq 0$ and an integer $\widetilde N$}, 
	with the following property: 
	for any 
	initial condition $(t,m) \in 
	\prob_N(\T^d)$,
	with $\| m \|_{2,q-1}^2 \leq c$ and $m  \geq \gamma$,
	for any $N \geq \widetilde N$
	and
	for any measurable function $\alpha : [t,T] \rightarrow 
	{\mathscr A}_N$ satisfying 
	\eqref{eq:sufficient:condition:ON}
	with respect to $C$, the 
	solution 
	$(\xi_s^N)_{t \leq s \leq T}$
	to 
	\eqref{eqn: closed_sys_ODE_OC} (with ${\mathbf z} = \{\widehat{m}(k)\}_{\vert k \vert \leq N}$)
	satisfies 
	\begin{equation*} 
	\sum_{\vert k \vert \leq N} \xi_s^N(k) e_k \geq \addtxtr{\tilde{\gamma}}, 
	\quad 
	\addtxtr{\sum_{\vert k \vert \leq N} \vert  \xi_s^N(k) \vert^2 \bigl( 1 + \vert k \vert^2 \bigr)^{2 \lfloor q-2-d/2 \rfloor
} \leq \tilde{c}}.
	\end{equation*} 
\end{lemma}
\addtxtr{As we already explained in the previous subsection, $C$ right above must be sufficiently large so that 
$\alpha^{*,N}$, which already belongs to ${\mathcal A}_N$, also satisfies 
\eqref{eq:sufficient:condition:ON}.} 
\vskip 6pt

Lemma \ref{lemma: lower_bound_muN:2}
says that we can put for free an additional reflection (or absorption) term on the boundary of ${\mathscr O}_N$ or on the boundary of any subset 
$\widetilde{\mathcal B} \subset {\mathscr O}_N$ containing ${\mathcal B} := \{ {\mathbf z} \in {\mathscr O}_N : 
\addtxtr{\sum_{\vert k \vert \leq N} \vert z_k \vert^2 (1+ \vert k \vert^2) ^{q-1} \leq 
{c}
\ {\rm and} \ 
\sum_{\vert k \vert \leq N} z_k e_k \geq  {\gamma}}\}$. 
\color{black}{Given such a subset 
$\widetilde{\mathcal B}$, we can indeed 
follow 
the definition
\eqref{eqn: value_finit_dim}
and then 
consider the optimal control problem
\begin{equation}
\label{eq:boundary:HJB:inf:GammaN}
\inf \Gamma^N\left(\alpha,t,\textbf{z}\right),\quad (t,\mathbf{z})\in[0,T]\times \widetilde{\mathcal B},
\end{equation}
the infimum being taken 
over controls 
$\alpha \in {\mathcal A}_N$ satisfying 
\eqref{eq:sufficient:condition:ON}, but 
(differently from 
\eqref{eqn: value_finit_dim})
the corresponding controlled trajectories 
 $(\{\xi^N_t(k)\}_{\vert k \vert \leq N})_{0 \le t \le T}$
in 
\eqref{eqn: closed_sys_ODE_OC}
being reflected on the boundary of 
$\widetilde{\mathcal B}$ (or being absorbed, with then an appropriate extension of the definition of the cost functional 
$\Gamma^N$ to cover the case when absorption occurs). 
From Lemma \ref{lemma: lower_bound_muN:2}, we notice that
for  
${\mathbf z} \in {\mathcal B}$, the 
controlled trajectories starting from ${\mathbf z}$ are 
never reflected (nor absorbed). As a result, the  
value function associated with 
\eqref{eq:boundary:HJB:inf:GammaN}
and denoted $\widetilde{\Upsilon}^N$
coincides 
with $\Upsilon^N$ on $[0,T] \times {\mathcal B}$: this gives a way to access 
$\Upsilon^N$ through a finite-dimensional control problem with 
a properly defined boundary condition. 
Typically, we can choose the domain $\widetilde{\mathcal B}$ as the intersection of 
an ellipsoid and 
a polyhedron:
\begin{equation*} 
\begin{split}
\widetilde{\mathcal B} &:= 
\Bigl\{ {\mathbf z} \in {\mathbb C}^{(2N+1)^d} :
\sum_{\vert k \vert \leq N} \vert z_k \vert^2 (1+ \vert k \vert^2) ^{\lfloor q-2-d/2 \rfloor} \leq \tilde c 
\Bigr\}
\cap 
\biggl( 
\bigcap_{\vert k \vert \leq N} 
\Bigl\{ {\mathbf z} \in {\mathbb C}^{(2N+1)^d} :
z_{-k} = \overline{z}_k
\Bigr\} \biggr) 
\\
&\hspace{15pt}   \cap \bigl\{{\mathbf z}  \in {\mathbb C}^{(2N+1)^d} : z_0=1 \bigr\} \cap \biggl( \bigcap_{i=1}^m \Bigl\{{\mathbf z}  \in {\mathbb C}^{(2N+1)^d} :
\sum_{\vert k \vert \leq N} z_k e_k(x_i) \geq   \tilde \gamma
\Bigr\} \biggr),
\end{split} 
\end{equation*} 
with $(x_i)_{i=1,\dots,m}$ 
forming a $\delta$-net of ${\mathbb T}^d$ for some $\delta>0$ (i.e., 
the union of the balls of center $x_i$, for $i=1,\cdots,m$, and radius 
$\delta$ covers ${\mathbb T}^d$). 
It is clear that ${\mathcal B} \subset \widetilde{\mathcal B}$ (since $c\leq \tilde c$ and $\gamma \geq 
\tilde \gamma$). 
The only difficulty (with this choice) is thus to check that 
$\widetilde{\mathcal B}$ is included in ${\mathscr O}_N$, which we can prove as follows. 
For $x \in {\mathbb T}^d$, we  obviously have
$\sum_{\vert k \vert \leq N} z_k e_k(x) \in {\mathbb R}$. 
To prove that the latter is non-negative, we notice that there exists $i
\in \{1,\cdots,m\}$ such that $\vert x - x_i \vert \leq \delta$ and then
obtain for 
 ${\mathbf z} \in 
 \widetilde{\mathcal B}$, 
\begin{equation*} 
\begin{split} 
\sum_{\vert k \vert \leq N} z_k e_k(x) \geq 
\sum_{\vert k \vert \leq N} z_k e_k(x_i) - \sum_{\vert k \vert \leq N} \vert z_k \vert \vert e_k(x_i) -e_k(x) \vert
&\geq  \tilde{\gamma} - 2 \pi \delta \sum_{\vert k \vert \leq N} \vert z_k \vert \vert k\vert
\\
&\geq  \tilde{\gamma} - 2 C_{\tilde c,q} \pi \delta,
\end{split}
\end{equation*} 
for a constant $C_{\tilde c,q}$ depending on $\tilde c$, $q$ (and $d$). For $\delta$ small, the right-hand side is strictly positive.

The
control problem 
\eqref{eq:boundary:HJB:inf:GammaN}
with 
a reflection (or an absorption) condition  
 leads to an HJB equation with a Neumann (or Dirichlet) boundary condition, 
and the value function 
$\widetilde{\Upsilon}^N$ 
of this new control problem with reflection (or absorption) coincides 
 with 
$\Upsilon^N$
in
 \eqref{eqn: value_finit_dim},
 at least
on $[0,T] \times {\mathcal B}$. 
 On the finite-dimensional space $[0,T] \times \widetilde{\mathcal B}$, 
 the corresponding HJB equation 
 formally writes:
 \color{black}

\begin{equation}\label{eqn: HJB_fourier}
\begin{cases}
	\medskip
	\partial_t \widetilde{\Upsilon}^N(t,\mathbf{z})
	- \sum_{\lvert k \rvert \leq N} \lvert 2\pi k\rvert^2 z_k \partial_{z_k}\widetilde{\Upsilon}^N(t,\mathbf{z})
	\\
	\quad\quad-  \sum_{\lvert k \rvert \leq N}z_k\reallywidehat{\check H^N\left(\cdot,\ii 2\pi\sum_{\lvert j\rvert\leq N}j\partial_{z_j}\widetilde{\Upsilon}^N(t,\textbf{z})e_j(\cdot)\right)}(k)
	+ \cF(\mathbf{z}) = 0,\\
	\quad\\
	 \widetilde{\Upsilon}^N(T,\textbf{z})  = \cG(\mathbf{z}).
\end{cases}
\end{equation}
where, with a small abuse of notation, we have written $\cG(\mathbf{z}) = \cG(\sum_{\lvert k\rvert \leq N} \addtxtr{z_k} e_k(\cdot))$ and $\cF(\mathbf{z}) = \cF(\sum_{\lvert k\rvert \leq N} \addtxtr{z_k} e_k(\cdot))$. 
Here, the notation 
$\partial_{z_k}\widetilde{\Upsilon}^N(t,\mathbf{z})$
is rather abusive. It is not a complex (holomorphic) derivative but just a vector of two real derivatives:
\begin{equation*}
\begin{split}
\partial_{z_k} &:= 
\frac12 \partial_{\Re[z_k]}
-  \frac{\ii}2 \partial_{\Im[z_k]}.
\end{split}
\end{equation*}
As for $\check H^N$, it is a new Hamiltonian, whose expression must take into account the definition of 
${\mathcal A}_N$ together with the condition 
\eqref{eq:sufficient:condition:ON} in the statement of Lemma 
\ref{lem:sufficient:condition:ON}, namely
\begin{equation*} 
\begin{split}
\check{H}^N(x,p) &:= \sup
\biggl\{ - L\biggl(x,-D_p H \Bigl(x, \textrm{\rm i} 2 \pi 
\sum_{\vert l \vert \leq N} l \zeta(l) e_l (x) \Bigr) 
\biggr) +
D_p H \Bigl(x, \textrm{\rm i} 2 \pi 
\sum_{\vert l \vert \leq N} l \zeta(l) e_l (x) \Bigr) \cdot p \, ; 
\\
&\hspace{45pt} 
\zeta(l) \in {\mathbb C} : \overline{\zeta(l)} = \zeta(-l), \ \sum_{\vert l \vert \leq N} 
\vert l \vert^{ 2d+6} 
\vert \zeta(l) \vert^2
\leq C \biggr\}. 
\end{split} 
\end{equation*} 
When $p$ itself can be written in the form 
$p= 
\textrm{\rm i} 2 \pi 
\sum_{\vert l \vert \leq N} l \zeta(l) e_l (x)$ for $\{\zeta(l)\}_{\vert l \vert \le N}$ as above, it holds 
$\check{H}^N(x,p) =H(x,p)$. 
Of course, so should be the case when 
$p =
\ii 2\pi\sum_{\lvert j\rvert\leq N}j\partial_{z_j}\widetilde{\Upsilon}^N(t,\textbf{z})e_j(x)$
for $(t,{\mathbf z}) \in [0,T] 
\times {\mathcal B}$, in which case 
$p$ should be also equal to
$\ii 2\pi\sum_{\lvert j\rvert\leq N}j\partial_{z_j} \Upsilon^N(t,\textbf{z})e_j(x)$. 
This should permit to identify 
\eqref{eqn: HJB_fourier} on 
$[0,T] \times {\mathcal B}$
with 
 \eqref{eqn: HJB_N}.
 For simplicity, we feel better not to give the full proof of 
this claim, but it does not raise any main difficulty. In short, 
the point is to prove that
\begin{equation} 
\label{eq:identification:H:check H:2} 
\sum_{\vert j \vert \leq N} 
\vert j \vert^{ 2d+6} 
\vert 
\partial_{z_j}\Upsilon^N(t,\textbf{z})
 \vert^2
\leq C,
\end{equation} 
for ${\mathbf z} \in {\mathcal B}$.
Intuitively, Pontryagin's principle  provides a 
representation of 
$\partial_{z_j}\Upsilon^N(t,\textbf{z})$, but
for  
${\mathbf z}$  in 
${\mathcal B}$,
the Pontryagin system may be identified with 
\eqref{eqn: FWBKW_intro}
and in turn the quantity 
$\ii 2\pi\sum_{\lvert j\rvert\leq N}j\partial_{z_j}\Upsilon^N(t,\textbf{z})e_j(x)$
may be identified
with 
$\nabla u^N_t(x)$ (when 
\eqref{eqn: FWBKW_intro}
is initialized from 
$m=\sum_{\vert k \vert \leq N}
z_k e_k$ at time $t$). 
Display 
\eqref{eq:identification:H:check H:2} 
then follows from the choice of $C$, 
as we assumed the latter to be sufficiently large so that 
$\alpha^{*,N}$ satisfies 
\eqref{eq:sufficient:condition:ON}
 (under the additional condition that 
$q \geq 3d/2+4$). 
\vskip 6pt

Again, we plan to elaborate on 
\eqref{eqn: HJB_fourier}
in future works dedicated to numerical applications. 
Conceptually, 
\eqref{eqn: HJB_fourier}
should be seen as a finite-dimensional reformulation 
(up to the change of Hamiltonian) 
of 
 \eqref{eqn: HJB_N}.
The reader is referred to 
\cite{cecchindelarue}
for more details on the 
Fourier formulation of PDEs set on ${\mathcal P}({\mathbb T}^d)$ 
and to 
\cite{Lions-neumann-HJB} for more details on control problems with Neumann boundary conditions. 

\subsubsection{Connection with mean field games}
\label{subsubse:MFG}
The forward-backward system 
\eqref{eqn: FWBKW_intro}
can be regarded as the mean field game system of a (so-called potential) mean field game admitting 
$\delta_\mu {\mathcal F}$ and $\delta_\mu {\mathcal G}$ as respective running and terminal costs, see
for instance 
\cite[Chapter 3.7]{cardaliaguetdelaruelasrylions}. Subsequently, 
the system 
\eqref{eqn: FWBKW_intro}
can be 
regarded as a Fourier-Galerkin approximation of this mean field game system. 
In fact, it may be easily seen that 
the specific structure of the running and terminal costs as derivatives of functions 
defined on ${\mathcal P}({\mathbb T}^d)$ 
play no role in the construction of the Fourier-Galerkin approximation. The same approach could be applied to 
more general costs from ${\mathbb T}^d \times {\mathcal P}({\mathbb T}^d)$, replacing the convexity properties of
${\mathcal F}$ and ${\mathcal G}$ by a suitable monotonicity properties. 
The convergence results exposed in Section \ref{sec: convergence} can be 
adapted in a direct way. 

Our choice to focus on mean field control mostly comes from the fact that the Fourier-Galerkin approximation 
has a clear interpretation since it can be connected itself with a finite-dimensional control problem (see the two previous paragraphs). 
When dealing with a general mean field game, the interpretation of the 
system \eqref{eqn: FWBKW_intro} as a mean field game system is certainly less obvious because one should then regard
\eqref{eqn: fp_n_OC}
as the flow of marginal laws of one player. Although conceivable, this would be rather artificial. 
Apart from this, the question is meaningful and we indeed plan to come back to the application to mean field games, 
but from a numerical prospect, in a future work.

\section{Convergence results}\label{sec: convergence}
Now that we have shown the well-posedness of the approximating problem, we can  address its convergence to the mean field control problem \eqref{eqn: value} presented in Subsection \ref{sec: mfc_problem}. As discussed in Subsection \ref{sec: classicresult}, we recall that under Assumptions \ref{hp: 1+2} and \ref{hp: 3}, the forward-backward system \eqref{eqn: fwdbkw_infty_intro}, that we re-state here
\begin{equation*}
	\begin{cases}
		\partial_t \mu^\infty_t &= \Delta\mu^\infty_t + \div(\diff_p H (\cdot,\nabla u^\infty_t)\mu^\infty_t),\quad t\in[0,T],\\
		\mu^\infty_0 &= m,\\
		\partial_ t u_t^\infty  &= - \Delta u_t^\infty  + H(\cdot, \nabla u^\infty_t) - \lf\cF(\mu^\infty_t,\cdot),\quad t\in[0,T],\\
		u^\infty_T &= \lf\cG(\mu^\infty_T, \cdot),
	\end{cases}
\end{equation*}
allows us to characterize the optimal trajectory $\mu^\infty$ together with the optimal control $\alpha^{*,\infty}(\cdot) = - \diff_p H(\cdot, \nabla u^\infty(\cdot))$. Similarly, in Section \ref{sec: approx_fwdbkw_problem} we showed that under Assumptions \ref{hp: 1+2} and \ref{hp: 3}, the system \eqref{eqn: FWBKW_intro}, that is
\begin{equation*}
	\begin{cases}
		\partial_t \mu^N_t &= \Delta\mu^N_t + \div(\diff_p H(\cdot,\nabla u^N_t)(\mu^N_t*D^N)),\quad t\in[0,T],\\
		\smallskip
		\mu^N_0 &= m,\\
		\smallskip
		\partial_ t  u^N_t  &= - \Delta  u^N_t  +  H( \cdot,\nabla  u^N_t)*D^N  - \lf\cF(\mu^N_t*D^N,\cdot)*D^N,\quad t\in[0,T],\\
		\smallskip
		  u^N_T &= \lf\cG(\mu^N_T*D^N,\cdot)*D^N,
	\end{cases}
\end{equation*}
gives us the optimal pair $\mu^N$ and $\alpha^{*,N}(\cdot) = -\diff_p H(\cdot,\nabla u^N(\cdot))$ for the approximating MFC problem \eqref{eqn: value_N}, whenever $N$ is chosen large enough. 

\addtxtr{We  establish three different results of convergence,  each with explicit rates. First, we  address  the convergence of the approximating optimal control in the $L^2$ sense. This is done 
in Subsection \ref{sec: monotonicity}
by  means of a stability argument based on the monotonicity property of the coefficients. 
 Then, the convergence is shown to hold  with respect to the uniform norm under an additional condition on the terminal cost $\cG$, see Subsection \ref{sec: unif_conv}}. In the same subsection, similar results are obtained for the optimal trajectory. Finally, in  Subsection \ref{sec: conv_value}, we study the convergence of the value function $V^N$ defined by \eqref{eqn: value_N} to the value function of the original MFC problem.  We identify subsets of  $\prob(\T^d)$ on which
  the convergence is uniform.
\vskip 6pt

\addtxtr{In most of the results exposed in this section, 
the approximating forward-backward system \eqref{eqn: FWBKW_intro} (and thus the approximating MFC control problem) is initiliazed from the same
 initial condition $m$
 as the original system 
 \eqref{eqn: fwdbkw_infty_intro}. This makes the presentation easier. 
 However, in line with Remark 
 \ref{rmk: first_on_init_cond} and with our initial objective, we show in Proposition 4.4 that 
 the main result remains true if one replaces $m$ by $m*D^N$ 
 in the initial condition of 
  \eqref{eqn: FWBKW_intro}.}
\vskip 6pt

\noindent To conclude, we emphasize that for the remainder of this section, we will consider again $N\geq\widetilde{N}$ according  to Remarks \ref{rmk: H_is_useless} and \ref{rmk: mutilde_n_prob}.

\subsection{$L^2$ convergence of the approximating optimal control}\label{sec: monotonicity}
Our aim here is to establish the convergence of $\alpha^{*,N}$ to $\alpha^{*,\infty}$ by exploiting the monotonicity of $\lf\cF$ and $\lf\cG$ (that is, the convexity of $\cF$ and $\cG$, see Remark \ref{rem: convex-monoton}). In particular, we will once again take inspiration from the usual Lasry-Lions monotonicity argument (which we already alluded to in the proof of Proposition 
\ref{prop:uniqueness}, see also e.g.~ \cite[Lemma 3.1.2]{cardaliaguetdelaruelasrylions}). This result, see Theorem \ref{thm: conv_of_controls} below, gives us the rate of convergence of the approximating optimal control in a suitable $L^2$ norm.  Its proof relies on the following auxiliary estimate for $\norm{\mu^\infty_t - \mu^N_t}_2$:
\begin{lemma}\label{lemma: est_diff_mu_l2}
	Let Assumptions \ref{hp: 1+2} and \ref{hp: 3} hold. Then there exists a positive constant $C = C(d,T,\cF,\cG,H,q)$ such that
	\begin{equation}\label{eqn: est_diff_mu_l2}
		\sup_{t\in[0,T]} \norm{\mu^\infty_t - \mu^N_t}_2^2\leq C(d,T,\cF,\cG,H,q)\norm{m}^2_{2,q-1}\left(\frac{1}{N^{2(q-1)}} + \int_0^T\norm{\nabla  u^N_t - \nabla u^\infty_t}_2^2 \de t\right).
	\end{equation}
\end{lemma}
\begin{proof}
	We have that $(\mu^\infty - \mu^N)$ solves 
	\begin{equation}\label{eqn: fp_diff_muinf_mutilde}
    		\partial_t(\mu^\infty_t - \mu^N_t) = \Delta (\mu^\infty_t - \mu^N_t) + \div(\diff_p H(\cdot,\nabla u^\infty_t)\mu^\infty_t - \diff_p H(\cdot,\nabla u^N_t)(\mu^N_t*D^N)).
    \end{equation} 
	By integrating the equation against $(\mu^\infty - \mu^N)$ and then by Young's inequality it follows
	\begin{align*}
		\frac{1}{2}\norm{\mu^\infty_t &- \mu^N_t}_2^2 + \int_0^t \norm{\nabla\mu^\infty_s - \nabla\mu^N_s}_2^2 \de s\\
		&\leq \int_0^t  \norm{\diff_p H(\cdot,\nabla u^\infty_s)\mu^\infty_s - \diff_p H(\cdot,\nabla  u^N_s)(\mu^N_s*D^N)}_2\norm{\nabla\mu^\infty_s - \nabla\mu^N_s}_2\de s\\
		&\leq \frac{1}{2}\int_0^t \norm{\diff_p H(\cdot,\nabla u^\infty_s)\mu^\infty_s - \diff_p H(\cdot,\nabla  u^N_t)(\mu^N_s*D^N)}^2_2\de s+\frac{1}{2}\int_0^t \norm{\nabla\mu^\infty_s - \nabla\mu^N_s}_2^2\de s. 
	\end{align*}
	Thus, 
	\begin{align*}
		&\norm{\mu^\infty_t - \mu^N_t}_2^2\\
		&\leq C \int_0^t \Big\lVert\Big(\diff_p H(\cdot,\nabla u^\infty_s)-\diff_p H(\cdot,\nabla  u^N_s)\Big)\mu^\infty_s\Big\rVert_2^2\de s \\
		&\quad+ C\int_0^t \Big\lVert\diff_p H(\cdot,\nabla  u^N_s)(\mu^\infty_s-{\mu^{{\infty}}_s*D^N})\Big\rVert_2^2\de s 
		+ C\int_0^t \Big\lVert\diff_p H(\cdot,\nabla  u^N_s)(\mu^{{\infty}}_s-\mu^N_s)*D^N\Big\rVert_2^2\de s
		\\
		&\leq C(d,T,\cF,\cG,H,q) \left(\norm{m}^2_{2,q-1}\int_0^t\norm{\nabla u^\infty_s-\nabla  u^N_s}_2^2 \de s + \frac{\norm{m}^2_{2,q-1}}{N^{2(q-1)}}
		+ \int_0^t \norm{\mu^\infty_s-\mu^N_s}_2^2\de s \right),
	\end{align*}
	where the last inequality follows from the $L^\infty$ bound \eqref{eqn: muinft_bound} on $\mu^\infty$ (for the first term), the boundedness of $\diff_p H(\cdot,\nabla u^N(\cdot))$ (for the second and third terms) and the fact that $\mu^{{\infty}}_s\in H^{q-1}(\T^d)$  (for the {second} term, in combination with Lemma \ref{lemma: estl2truncation}). Moreover we used the Lipschitz property of the mapping $p\mapsto\diff_p H(x,p)$ uniformly with respect to $x$, which is entailed by Assumption \ref{hp: 1+2} - (H.5). Finally, we obtain \eqref{eqn: est_diff_mu_l2} from Gronwall's inequality. 
\end{proof}
We can now state and prove the main result of this subsection.
\begin{theorem}\label{thm: conv_of_controls}
	Let Assumptions \ref{hp: 1+2} and \ref{hp: 3} hold. Then there exists a positive constant $C = C(d,T,\cF,\cG,H,q,\addtxtr{\gamma})$ such that
	\begin{equation*}
		\int_0^T\norm{\nabla  u^N_t - \nabla u^\infty_t}_2^2 \de t \leq \frac{C(d,T,\cF,\cG,H,q,\addtxtr{\gamma})}{N^{2(q-1)}}\norm{m}^4_{2,q-1}.
	\end{equation*}
\end{theorem}
\begin{proof}
    First, we notice that $(u^\infty -  u^N)$ solves in $[0,T]\times\T^d$
    \begin{equation}
    	\begin{cases}
    		\partial_t (u^\infty_t -  u^N_t) &= -\Delta(u^\infty_t -  u^N_t) + (H(\cdot, \nabla u^\infty_t) - H(\cdot,\nabla u^N_t)*D^N)\\
    		&\qquad \qquad \qquad\qquad- (\lf\cF(\mu^\infty_t,\cdot) - \lf\cF(\mu^N_t*D^N,\cdot)*D^N),\\
		\quad\\
    		(u^\infty_T -  u^N_T)(x) &= \lf\cG(\mu^\infty_T,x) - \lf\cG(\mu^N_T,\cdot)*D^N(x),
    	\end{cases}
    \end{equation}
    whilst $(\mu^\infty - \mu^N)$ solves \eqref{eqn: fp_diff_muinf_mutilde}.
    We will follow the duality argument often used in the literature on mean field games (see, e.g.~\cite[Lemma 3.1.2]{cardaliaguetdelaruelasrylions})): integration by parts leads to 
    \begin{align}
    	\frac{\de}{\de t}\scalprod{\mu^\infty_t &- \mu^N_t}{u^\infty_t -  u^N_t} 
    	 \label{eqn: duality_monotonicity}\\
	& = -\scalprod{\diff_p H(\cdot,\nabla u^\infty_t)\mu^\infty_t - \diff_p H(\cdot,\nabla u^N_t)(\mu^N_t*D^N)}{\nabla u^\infty_t - \nabla u^N_t}\\
    	 & \quad+ \scalprod{\mu^\infty_t - \mu^N_t}{H(\cdot,\nabla u^\infty_t) - H(\cdot,\nabla u^N_t)*D^N}\nonumber\\
    	 & \quad- \scalprod{\mu^\infty_t - \mu^N_t}{\lf\cF(\mu^\infty_t,\cdot) - \lf\cF(\mu^N_t*D^N,\cdot)*D^N} \nonumber\\
    	 & =  - \scalprod{\diff_p H(\cdot,\nabla u^\infty_t) (\mu^\infty_t*D^N) - \diff_p H(\cdot,\nabla u^N_t)(\mu^N_t*D^N)}{\nabla u^\infty_t - \nabla u^N_t}\label{eqn: dualprod_1}\\
    	 &\quad + \scalprod{\mu^\infty_t - \mu^N_t}{H(\cdot,\nabla u^\infty_t)*D^N - H(\cdot,\nabla u^N_t)*D^N}\label{eqn: dualprod_2}\\
    	 & \quad- \scalprod{\mu^\infty_t - \mu^N_t}{\lf\cF(\mu^\infty_t,\cdot) - \lf\cF(\mu^N_t*D^N,\cdot)*D^N}\label{eqn: dualprod_3}\\
    	 & \quad+ 	\scalprod{\diff_pH(\cdot,\nabla u^\infty_t) (\mu^\infty_t*D^N - \mu^\infty_t)}{\nabla u^\infty_t - \nabla  u^N_t}\label{eqn: dualprod_4}\\
    	 &\quad + \scalprod{\mu^\infty_t - \mu^N_t}{H(\cdot,\nabla u^\infty_t) - H(\cdot,\nabla u^\infty_t)*D^N}.\label{eqn: dualprod_5}
    \end{align}
    Regarding \eqref{eqn: dualprod_1} and \eqref{eqn: dualprod_2}, from the strong convexity of $H$ in the $p$ argument 
    and the symmetry of $D^N$, we have, for a constant $C>0$,
    \begin{equation*}
    	\eqref{eqn: dualprod_1} + \eqref{eqn: dualprod_2} \leq - C\scalprod{(\mu^N_t + \mu^\infty_t)*D^N}{\lvert \nabla  u^N_t - \nabla u^\infty_t\rvert^2},
    \end{equation*}
    and thanks to the symmetry of $D^N$ again, we can write \eqref{eqn: dualprod_3} as
    \begin{align*}
    	\eqref{eqn: dualprod_3} &= -\scalprod{\mu^\infty_t*D^N - \mu^N_t*D^N}{\lf\cF(\mu^\infty_t*D^N,\cdot) - \lf\cF(\mu^N_t*D^N,\cdot)}\\
    	&\quad+ \scalprod{(\mu^\infty_t - \mu^N_t)*D^N}{\lf\cF(\mu^\infty_t*D^N,\cdot) - \lf\cF(\mu^\infty_t,\cdot)}\\
    	&\quad+\scalprod{\mu^\infty_t - \mu^N_t}{\lf\cF(\mu^\infty_t,\cdot)*D^N - \lf\cF(\mu^\infty_t,\cdot)}.
    \end{align*}
    Notice that $\scalprod{\mu^\infty_t*D^N - \mu^N_t*D^N}{\lf\cF(\mu^\infty_t*D^N,\cdot) - \lf\cF(\mu^N_t*D^N,\cdot)}\geq 0$ due to the monotonicity of $\lf\cF$ and the fact that $\mu^\infty_t*D^N$ and $\mu^N_t*D^N$ are both probability measures.
    Therefore, by integrating in time \eqref{eqn: duality_monotonicity} we arrive at the following expression
    \begin{align*}
    	\scalprod{\mu^\infty_T &- \mu^N_T}{\lf\cG(\mu^\infty_T,\cdot) - \lf\cG(\mu^N_T*D^N,\cdot)*D^N} - \scalprod{\mu^\infty_0 - \mu^N_0}{u^\infty_0 -  u^N_0} \\
    	 \leq & - \int_0^T C\scalprod{(\mu^N_t + \mu^\infty_t)*D^N}{\lvert \nabla  u^N_t - \nabla u^\infty_t\rvert^2} \de t\\
    	&+ \int_0^T\scalprod{(\mu^\infty_t - \mu^N_t)*D^N}{\lf\cF(\mu^\infty_t*D^N,\cdot) - \lf\cF(\mu^\infty_t,\cdot)} \de t \\
    	&+ \int_0^T\scalprod{\mu^\infty_t - \mu^N_t}{\lf\cF(\mu^\infty_t,\cdot)*D^N - \lf\cF(\mu^\infty_t,\cdot)}  \de t \\
    	& + \int_0^T \scalprod{\diff_p H(\cdot,\nabla u^\infty_t) (\mu^\infty_t*D^N - \mu^\infty_t)}{\nabla u^\infty_t - \nabla  u^N_t} \de t\\
    	&+ \int_0^T  \scalprod{\mu^\infty_t - \mu^N_t}{H(\cdot, \nabla u^\infty_t) - H(\cdot, \nabla u^\infty_t)*D^N} \de t.
    \end{align*}
    For the term in the left hand side containing $\lf\cG$, we can proceed as for $\eqref{eqn: dualprod_3}$. Thus, by re-arranging the terms, we obtain
     \begin{align}
     	 C\int_0^T \scalprod{(\mu^N_t + \mu^\infty_t)*&D^N}{\lvert \nabla  u^N_t - \nabla u^\infty_t\rvert^2} \de t\label{eqn: monestimate_lhs}\\
     	\leq&-\scalprod{\mu^\infty_T*D^N - \mu^N_T*D^N}{\lf\cG(\mu^\infty_T*D^N,\cdot) - \lf\cG(\mu^N_T*D^N,\cdot)}\label{eqn: monestimate_1}\\
    	&+ \scalprod{(\mu^\infty_T - \mu^N_T)*D^N}{\lf\cG(\mu^\infty_T*D^N,\cdot) - \lf\cG(\mu^\infty_T,\cdot)}\label{eqn: monestimate_2}\\
    	&+\scalprod{\mu^\infty_T - \mu^N_T}{\lf\cG(\mu^\infty_T,\cdot)*D^N - \lf\cG(\mu^\infty_T,\cdot)}\label{eqn: monestimate_3}\\
	& + \scalprod{\mu^\infty_0 - \mu^N_0}{u^\infty_0 -  u^N_0} \label{eqn: monestimate_4}\\
    	&+ \int_0^T\scalprod{(\mu^\infty_t - \mu^N_t)*D^N}{\lf\cF(\mu^\infty_t*D^N,\cdot) - \lf\cF(\mu^\infty_t,\cdot)} \de t \label{eqn: monestimate_6}\\
    	&+ \int_0^T\scalprod{\mu^\infty_t - \mu^N_t}{\lf\cF(\mu^\infty_t,\cdot)*D^N - \lf\cF(\mu^\infty_t,\cdot)}  \de t \label{eqn: monestimate_7}\\
    	& + \int_0^T \scalprod{\diff_p H(\cdot,\nabla u^\infty_t) (\mu^\infty_t*D^N - \mu^\infty_t)}{\nabla u^\infty_t - \nabla  u^N_t} \de t \label{eqn: monestimate_8}\\
    	&+ \int_0^T  \scalprod{\mu^\infty_t - \mu^N_t}{H(\cdot,\nabla u^\infty_t) - H(\cdot,\nabla u^\infty_t)*D^N} \de t \label{eqn: monestimate_9}.
    \end{align}
    First, we notice from the monotonicity of $\lf\cG$ that $\eqref{eqn: monestimate_1}\leq 0$, and, since $\mu^\infty$ and $\mu^N$ have the same initial condition $m$, it also holds that $\eqref{eqn: monestimate_4} = 0$ (in \addtxtr{Proposition} \ref{rmk: different_ic} we discuss the case $\mu^N_0 = m*D^N$). Now, to deal with the remaining terms, we want to exploit the spatial regularity of $\lf\cG$ and $\lf\cF$, as well as the smoothness of $\mu^\infty$ and $u^\infty$. 
    To handle \eqref{eqn: monestimate_2}, we first rely on the Lipschitz property
    of $\lf\cG$ 
     with respect to the $L^2$ norm:
    \begin{equation*}
	\eqref{eqn: monestimate_2}\leq \norm{\mu^\infty_T - \mu^N_T}_2\norm{\lf\cG(\mu^\infty_T*D^N, \cdot) - \lf\cG(\mu^\infty_T, \cdot)}_2
	\leq C(d,\cG)\norm{\mu^\infty_T - \mu^N_T}_2\norm{\mu^\infty_T*D^N - \mu^\infty_T}_2.
    \end{equation*}
    Then, Young's inequality, Lemma \ref{lemma: est_diff_mu_l2} and the regularity of $\mu^\infty$ given by Proposition \ref{prop: smooth_mu_inf} combined with Lemma \ref{lemma: estl2truncation} entail that, for any $\eps>0$,
    \begin{align*}
    	\eqref{eqn: monestimate_2}&\leq \frac{\eps}{2}\norm{\mu^\infty_T - \mu^N_T}_2^2 + \frac{1}{2\eps}\norm{\mu^\infty_T*D^N - \mu^\infty_T}_2^2\\ 
	&\leq C(d,T,\cF,\cG,H,q)	\norm{m}^2_{2,q-1}\left(\frac{\eps}{ N^{2(q-1)}}+ \eps \int_0^T\norm{\nabla  u^N_t - \nabla u^\infty_t}_2^2 \de t 
	+  \frac{\eps^{-1}}{N^{2(q-1)}}\right).
    \end{align*}
    Similarly, by the Lipschitz property of $\lf\cF$, we obtain the same estimate (possibly with a different constant) for \eqref{eqn: monestimate_6}.

    Regarding \eqref{eqn: monestimate_3}, since $\lf\cG(\mu,\cdot)\in \rmC^{q}(\T^d)$ for every $\mu$, with $\norm{\lf\cG(\mu,\cdot)}_{\infty,q}$ independent of $\mu$, we have for any $\eps>0$
    \begin{align*}
    	\eqref{eqn: monestimate_3}&\leq \norm{\mu^\infty_T - \mu^N_T}_2 \norm{\lf\cG(\mu^\infty_T,\cdot)*D^N - \lf\cG(\mu^\infty_T,\cdot)}_2\\
	&\leq \frac{\eps}{2}\norm{\mu^\infty_T - \mu^N_T}_2 ^2 + \frac{\eps^{-1}}{2}\norm{\lf\cG(\mu^\infty_T,\cdot)*D^N - \lf\cG(\mu^\infty_T,\cdot)}_2^2\\
	&\leq C(d,T,\cF,\cG,H,q)\left(\frac{\eps}{ N^{2(q-1)}}\norm{m}^2_{2,q-1}+ \eps \norm{m}^2_{2,q-1}\int_0^T\norm{\nabla  u^N_t - \nabla u^\infty_t}_2^2 \de t 
	+  \frac{\eps^{-1}}{N^{2q}}\right),
    \end{align*}
    where we used Lemma \ref{lemma: estl2truncation} to control $\norm{\lf\cG(\mu^\infty_T,\cdot)*D^N - \lf\cG(\mu^\infty_T,\cdot)}_2$, recalling that $\lf\cG(\mu,\cdot)\in \rmC^{q}(\T^d)\subset H^{q}(\T^d)$. In the same manner and by noticing that the estimates on the $L^2$ norm of $\mu^\infty$ and $\mu^N$ are uniform in time, we obtain an analogue bound for \eqref{eqn: monestimate_7}.

    For the last terms \eqref{eqn: monestimate_8} and \eqref{eqn: monestimate_9}, the convergence is guaranteed by the fact that $\mu^\infty_t\in H^{q-1}(\T^d)$ and $\nabla u_t^\infty\in\rmC^{{q-1}}(\T^d)$, with norms bounded by constants independent of $t$. In particular, from Propositions \ref{prop: estimates_uinf} and  \ref{prop: smooth_mu_inf},  Remark \ref{rmk: smoothness_H_of_grad} and Lemma \ref{lemma: estl2truncation}, we have for any $\eps>0$     
    \begin{align*}
    	 \eqref{eqn: monestimate_8}&\leq \sup_{t\in[0,T]}\norm{\diff_p H(\cdot,\nabla u^\infty_t)}_\infty \int_0^T\norm{\mu^\infty_t - \mu^\infty_t*D^N}_2\norm{\nabla u^\infty_t -\nabla  u^N_t}_2 \de t\\ 
	& \leq C(d,T,\cF,\cG,H,q)\frac{\eps^{-1}}{\norm{m}^2_{2,q-1}}\int_0^T \norm{\mu^\infty_t - \mu^\infty_t*D^N}_2^2 \de t +\eps
	\norm{m}^2_{2,q-1}
	\int_0^T\norm{\nabla u^\infty_t -\nabla  u^N_t}_2^2\de t\\
	& \leq C(d,T,\cF,\cG,H,q)\frac{\eps^{-1}}{N^{2(q-1)}} + \eps\norm{m}^2_{2,q-1}\int_0^T\norm{\nabla u^\infty_t -\nabla  u^N_t}_2^2\de t.
    \end{align*}
    Moreover, by leveraging again Lemma \ref{lemma: est_diff_mu_l2}, Remark \ref{rmk: smoothness_H_of_grad} and Lemma \ref{lemma: estl2truncation}, we can argue as we did for \eqref{eqn: monestimate_3} and obtain that, for any $\eps>0$,
    \begin{align*}
    	\eqref{eqn: monestimate_9}&\leq\int_0^T \norm{\mu^\infty_t - \mu^N_t}_2\norm{H(\cdot,\nabla u^\infty_t) - H(\cdot,\nabla u^\infty_t)*D^N}_2\de t\\
	&\leq C(d,T,\cF,\cG,H,q)\left(\frac{\eps}{ N^{2(q-1)}}\norm{m}^2_{2,q-1}+ \eps\norm{m}^2_{2,q-1} \int_0^T\norm{\nabla  u^N_t - \nabla u^\infty_t}_2^2 \de t 
	+  \frac{\eps^{-1}}{N^{2({q-1})}}\right).
    \end{align*}
    
    \noindent Regarding the left-hand side, we have
	\begin{equation*}
		\eqref{eqn: monestimate_lhs}\geq\int_0^T\scalprod{\mu^\infty_t*D^N}{\lvert \nabla  u^N_t - \nabla u^\infty_t\rvert^2}\de t\geq \bar\gamma\int_0^T\norm{\nabla  u^N_t - \nabla u^\infty_t}_2^2 \de t
	\end{equation*}
	Here we used the fact that $\mu^N_t*D^N$ is a positive measure, together with $\inf_{t\in[0,T]}\mu^\infty_t*D^N\geq\bar\gamma>0$ for $N$ large enough (see Lemma \ref{lemma: lower_bound_mu_infty}).
	
	Finally, by suitably tuning the various $\eps$ (we choose them $\sim\norm{m}^{-2}_{2,q-1}$) and by combining all the previous estimates, we obtain
	\begin{equation*}
		\int_0^T\norm{\nabla  u^N_t - \nabla u^\infty_t}_2^2 \de t \leq \frac{C(d,T,\cF,\cG,H,q,\addtxtr{\gamma})}{N^{2(q-1)}}\norm{m}^4_{2,q-1}.
	\end{equation*}
	Notice that the presence of the sole term $\norm{m}^4_{2,q-1}$ in the right-hand side is due to the fact that $\norm{m}_{2,q-1}\geq\norm{m}_1 = 1$, and so $\max\left\{\norm{m}^4_{2,q-1}, \norm{m}^2_{2,q-1},1\right\} = \norm{m}^4_{2,q-1}$.
\end{proof}
\color{black}

\begin{remark}
	 The fourth power of $\norm{m}_{2,{q-1}}$
	 in the statement of Theorem 
	 \ref{thm: conv_of_controls}
	 may seem rather unusual. In fact, it 	 
	  arises solely as a consequence of the estimates for
	 the two terms \eqref{eqn: monestimate_2} and \eqref{eqn: monestimate_6}. Specifically, 
	 computations for these two quantities lead to upper bounds comprising products of the form 
	 $\eps^{-1}\norm{m}^2_{2,q-1}$, whereas 
	 all the other terms of order $\eps^{-1}$ handled in the proof 
	 do not include the extra factor $\norm{m}^2_{2,q-1}$. 
	 When choosing $\eps\sim\norm{m}^{-2}_{2,q-1}$, we obtain the term $\norm{m}^4_{2,q-1}$.	
	 
	 Of course, one may wonder about the optimality of the exponent. In this regard, the first observation is that 
	 the additional factor $\norm{m}^2_{2,q-1}$ appearing in the estimates for \eqref{eqn: monestimate_2} and \eqref{eqn: monestimate_6} comes from the available bounds for $(\norm{\mu^\infty_t*D^N - \mu^\infty_t}_2)_{0 \le t \le T}$.
	 Basically, the latter can be controlled by 
	 Lemma 
	 \ref{lemma: estl2truncation}, with sharp bounds depending on the smoothness of $(\mu_t^\infty)_{0 \le t \leq T}$. 
	 Up to this point, we do not see any way to modify the proof. 	 
	 The second 
	 observation is that, in our proof, 	 
	 we have just controlled  $(\norm{\mu^\infty_t}_{2,q-1})_{0 \le t \le T}$
	 by $\norm{m}_{2,q-1}$. This is the most direct and systematic way to proceed (and leads to the final bound in Theorem 
	 \ref{thm: conv_of_controls}). Our feeling 
	 is that this could be potentially improved by using further smoothing properties
	 of the Fokker-Planck equation in 
	 \eqref{eqn: fwdbkw_infty_intro}: in positive time $t \in (0,T]$,
	 $\mu^\infty_t$ has better regularity than 
	 $m$ and lives in a higher Sobolev space. 
	 Combining the higher regularity with some interpolation inequality and the fact that $\norm{\mu_t^{\infty}}_1=1$
	 for any $t \in [0,T]$, we think it is possible to prove that, for large values of $q$ (with respect to $d$) and for $\alpha \in (0,1)$ small, 
	 there exist two constants $C$ and $\theta \in (0,1)$ such that 
	 \begin{equation*}
		\norm{\mu_t^\infty}_{2,q-1}\leq \frac{C}{t^{\alpha}}\norm{m}_{2,q-1}^{\theta}.
	\end{equation*}
Inserting the above bound in  Lemma 
	 \ref{lemma: estl2truncation}, this could be a way to (slightly) decrease the exponent 4 appearing in 
	 Theorem 
	 \ref{thm: conv_of_controls}.
\end{remark}
We emphasize that Theorem 
\ref{thm: conv_of_controls}
still applies when \eqref{eqn: FWBKW_intro} starts from the truncated measure $\mu^N_0 = m*D^N$ instead of $m$:
\begin{proposition} 
\label{rmk: different_ic}
Let Assumptions \ref{hp: 1+2} and \ref{hp: 3} hold.
Initialize \eqref{eqn: FWBKW_intro} from 
$\mu^N_0 = m*D^N$. Then,
the conclusion of  
Theorem 
\ref{thm: conv_of_controls}
remains true. 
\end{proposition}

 \begin{proof}
We just provide a sketch of proof. 
In comparison with the proof of Theorem \ref{thm: conv_of_controls}, the only difference comes from the term \eqref{eqn: monestimate_4}, which is not identically equal to $0$. 
So, the only difficulty is to show that, for any $\eps>0$,  
	\begin{equation}\label{eqn: est_init_cond}
	\begin{aligned}
		\lvert\scalprod{&\mu^\infty_0 - \mu^N_0}{u^\infty_0 -  u^N_0} \rvert\\
		&\leq C(d,T,\cF,\cG,H,q)\norm{m}^2_{2,q-1}\left(\frac{\eps}{ N^{2(q-1)}}+ \eps\int_0^T\norm{\nabla  u^N_t - \nabla u^\infty_t}_2^2 \de t 
	+  \eps^{-1}\frac{\norm{m}^{2}_{2,q-1}}{N^{2(q-1)}}\right).
	\end{aligned}
	\end{equation}
	Once \eqref{eqn: est_init_cond} has been established, one can tune tune $\eps$ as done in the proof of Theorem \ref{thm: conv_of_controls} in order to reach the same conclusion. 
		
In the rest of the proof, we just address the derivation of \eqref{eqn: est_init_cond}.
	First, using the fact that $m\in H^{q-1}(\T^d)$, we deduce from Young's inequality and Lemma \ref{lemma: estl2truncation} that, for any $\eps >0$,
	\begin{equation}\label{eqn: est_init_cond_2}
		\lvert\scalprod{\mu^\infty_0 - \mu^N_0}{u^\infty_0 -  u^N_0} \rvert\leq \eps^{-1}\norm{m - m*D^{{N}}}^2_2 + \eps\norm{u^\infty_0-u^N_0}^2_2\leq\eps^{-1}\frac{\norm{m}_{2,q-1}^2}{N^{2(q-1)}} +\eps \norm{u^\infty_0-u^N_0}^2_2.
	\end{equation}
	Let us focus on $\norm{u^\infty_0-u^N_0}^2_2$. As an initial step, let us observe that Lemma \ref{lemma: est_diff_mu_l2} holds even if $\mu^N_0 = m*D^N$. The proof is almost the same. The only change comes from the difference between the initial conditions, but it can be estimated by means of Lemma \ref{lemma: estl2truncation}. Then, 
	forming the difference between the two equations
	solved by $u^\infty$ and  $u^N$ respectively, we obtain
	\begin{align}
		\frac{1}{2}\norm{u^\infty_0 - u^{N}_0}_2^2 &+ \int_0^T \norm{\nabla u^\infty_s - \nabla u^N_s}_2^2 \de s\nonumber\\
		&\leq \frac{1}{2}\norm{\lf \cG (\mu^\infty_T,\cdot) - \lf\cG(\mu^N_T*D^N,\cdot)*D^N}_2^2\label{eqn: est_u0_1}\\
		&\quad + \int_0^T\left\lvert \scalprod{\lf \cF (\mu^\infty_s,\cdot) - \lf\cF(\mu^N_s*D^N,\cdot)*D^N}{\nabla u^\infty_s - \nabla u^N_s}\right\rvert \de s\label{eqn: est_u0_2}\\
		&\quad + \int_0^T\left\lvert \scalprod{ H (\cdot,\nabla u^\infty_s) - H(\cdot,\nabla u^N_s)*D^N}{\nabla u^\infty_s - \nabla u^N_s}\right\rvert \de s.\label{eqn: est_u0_3}
	\end{align}
	Regarding \eqref{eqn: est_u0_1}, the regularity of $\cG$ and $\mu^N$ combined with Lemmas \ref{lemma: trivial_est_trunc} and \ref{lemma: estl2truncation} and the new version of Lemma \ref{lemma: est_diff_mu_l2} entails
	\begin{align*}
		\eqref{eqn: est_u0_1}& \leq \norm{\lf \cG (\mu^\infty_T,\cdot) - \lf\cG(\mu^\infty_T,\cdot)*D^N}_2^2 + \norm{\lf \cG (\mu^\infty_T,\cdot)*D^N - \lf\cG(\mu^N_T,\cdot)*D^N}_2^2\\
		&\quad + \norm{\lf \cG (\mu^N_T,\cdot)*D^N - \lf\cG(\mu^N_T*D^N,\cdot)*D^N}_2^2\\
		& \leq  C(d,T,\cF,\cG,H,q)\norm{m}_{2,q-1}^2\left (\frac{\norm{m}_{2,q-1}^{-2}}{N^{2q}} + \frac{1}{N^{2(q-1)}} + \int_0^T\norm{\nabla  u^N_t - \nabla u^\infty_t}_2^2 \de t \right). 
	\end{align*}
	For \eqref{eqn: est_u0_2},we can proceed similarly, up to an additional 
	use of Young's inequality to separate 
$\nabla u^\infty_s - \nabla u^N_s$ from the difference of the two $\delta_\mu {\mathcal F}$ terms.

	The last term can be treated by using Young's inequality and then by leveraging the smoothness of $H$:
	\begin{align*}
		\eqref{eqn: est_u0_3} &\leq \frac{1}{2}\int_0^T\norm{\nabla  u^N_t - \nabla u^\infty_t}_2^2 \de t
		+ \int_0^T\norm{H(\cdot, \nabla u^\infty_s) - H(\cdot,\nabla u^\infty_s)*D^N}_2^2\de s \\
		&\quad + \int_0^T \norm{H(\cdot,\nabla u^\infty_s)*D^N - H(\cdot,\nabla u^N_s)*D^N}_2^2\de s\\
		& \leq \frac{1}{2}\int_0^T\norm{\nabla  u^N_t - \nabla u^\infty_t}_2^2 \de t + C(d,T,\cF,\cG,H,q)\left (\frac{1}{N^{2(q-1)}} + \int_0^T\norm{\nabla  u^N_t - \nabla u^\infty_t}_2^2 \de t\right ).
	\end{align*}	
	By combining the bounds for 
	\eqref{eqn: est_u0_1}, \eqref{eqn: est_u0_2}
and \eqref{eqn: est_u0_3}, and by recalling that $\norm{m}_{2,q-1}\geq 1$, we get
	\begin{equation*}
		\norm{u^\infty_0 - u^{{N}}_0}_2^2\leq   C(d,T,\cF,\cG,H,q)\norm{m}_{2,q-1}^2\left ( \frac{1}{N^{2(q-1)}} + \int_0^T\norm{\nabla  u^N_t - \nabla u^\infty_t}_2^2 \de t \right).
	\end{equation*}
	We can finally plug this estimate in \eqref{eqn: est_init_cond_2} and obtain \eqref{eqn: est_init_cond}.
\end{proof}
\color{black} 
\medskip
\noindent 
\addtxtr{Theorem \ref{thm: conv_of_controls} 
and Proposition 
\ref{rmk: different_ic}
imply a similar result for the convergence of the optimal control.}
\begin{corollary}\label{cor: conv_of_controls}
	Let Assumptions \ref{hp: 1+2} and \ref{hp: 3} hold. 
		\addtxtr{Whether \eqref{eqn: FWBKW_intro} be initialized from  
$\mu^N_0 = m$ 
or
$\mu^N_0 = m*D^N$},  there exists a positive constant $C = C(d,T,\cF,\cG,H,q,\addtxtr{\gamma})$  such that
	\begin{equation*}
		\int_0^T\norm{\diff_p H(\cdot,\nabla  u^N_t) - \diff_p H(\cdot,\nabla u^\infty_t)}_2^2 \de t \leq \frac{C(d,T,\cF,\cG,H,q,\addtxtr{\gamma})}{N^{2(q-1)}}\norm{m}^4_{2,q-1}.
	\end{equation*}
\end{corollary}
\begin{proof}
	It is enough to notice that, for any $x\in\T^d$, the mapping $p\mapsto\diff_p H(x,p)$ is Lipschitz uniformly with respect to $x$. Indeed this is a consequence of Assumption \ref{hp: 1+2} - (H.5). Thus, the result follows from Theorem \ref{thm: conv_of_controls}
	\addtxtr{or 
	Proposition \ref{rmk: different_ic}}. 
\end{proof}

\subsection{Convergence in $L^\infty$ norm of the optimal trajectory and control}\label{sec: unif_conv}
From Theorem \ref{thm: conv_of_controls}, we can deduce estimates for the convergence of $\mu^N$ to $\mu^\infty$ in both $L^2$ and $L^\infty$ norms. Consequently, we can also analyze the convergence of the optimal feedback in $L^\infty$ norm (and not only in $L^2$ as done in Corollary \ref{cor: conv_of_controls}). To derive the estimates with respect to the $L^\infty$ norm, we follow the approach initially used in Proposition \ref{prop: infinitybounds}, and based on Duhamel's formula.
\vskip 6pt

\noindent To address the uniform convergence of the control, we need more regularity on $\lf\cG$.
\begin{assumption}\label{hp: 4}
	Let $\lf\cG\colon\prob(\T^d)\times\T^d\to\R^d$ be as in Assumption \ref{hp: 1+2}. We require that, for any $x\in\T^d$, the mapping $\mu\mapsto\nabla \lf\cG(\mu,x)$ is Lipschitz with respect to $L^2$, with a constant independent of $x$. 
\end{assumption}
We emphasize that we require this additional assumption only for $\lf\cG$ and not for $\lf\cF$. The rationale behind this distinction will become evident in the proof of Theorem \ref{thm: unif_conv_control}.
\begin{remark}
	We can notice that if $\cG\in\rmC^2(\prob(\T^d))$, with  $(y,z)\mapsto \lf^2\cG(\mu,y,z)$ being twice continuously differentiable and $z \mapsto\nabla_y\lf^2\cG(\mu,y,z)$ 
	being  bounded in $L^2({\mathbb T}^d)$ uniformly in $(y,\mu)$ (which is the case if the mapping is jointly continuous in the three variables $(\mu,y,z)$ when 
	$\prob(\T^d)$ is equipped with $W_1$), then $\nabla \lf\cG(\cdot,y)$ is Lipschitz in the measure argument with respect to $L^2$, with a constant uniform in space. Indeed, by definition of the linear derivative, it holds for any $\mu,\mu'\in\prob(\T^d)\cap L^2(\T^d)$, 
	\begin{equation*}
		\lvert \nabla \lf\cG(\mu,y) - \nabla \lf\cG(\mu',y)\rvert\leq \norm{\nabla_y\lf^2\cG(\mu,y,\cdot)}_2 \norm{\mu - \mu'}_2.
	\end{equation*}
\end{remark}
\begin{theorem}\label{thm: unif_conv_control}
	Let Assumptions \ref{hp: 1+2} and \ref{hp: 3} hold. Then there exists a positive constant $C = C(d,T,\cF,\cG,H,q,\addtxtr{\gamma})$ such that
	\begin{equation}\label{eqn: l2_est_final_mu}
		\sup_{t\in[0,T]} \norm{\mu^\infty_t - \mu^N_t}_2\leq \frac{C(d,T,\cF,\cG,H,q,\addtxtr{\gamma})}{N^{q-1}}\norm{m}^2_{2,q-1},
	\end{equation}
Moreover, if Assumption \ref{hp: 4} is in force, it also holds
	\begin{equation}\label{eqn: linf_esst_final_mu}
		\sup_{t\in[0,T]} \norm{\mu^\infty_t - \mu^N_t}_\infty\leq \frac{C(d,T,\cF,\cG,H,q,\addtxtr{\gamma})}{N^{q-1-\frac{d}{2}}}\norm{m}^2_{2,q-1},
	\end{equation}
and
	\begin{equation}\label{eqn: linf_esst_final_u}
		\sup_{t\in[0,T]}(\norm{\nabla  u^N_t - \nabla u^\infty_t}_\infty+ \norm{\diff_p H(\cdot,\nabla  u^N_t) - \diff_p H(\cdot,\nabla u^\infty_t)}_\infty )\leq \frac{C(d,T,\cF,\cG,H,q,\addtxtr{\gamma})}{N^{q-1-\frac{d}{2}}}\norm{m}^2_{2,q-1}.
	\end{equation}
\end{theorem}
\begin{proof}
We first notice that 
	\eqref{eqn: l2_est_final_mu} easily follows by plugging the result of Theorem \ref{thm: conv_of_controls} in the estimate of Lemma \ref{lemma: est_diff_mu_l2}.
\medskip

	We now turn to  \eqref{eqn: linf_esst_final_u}. The proof is very similar to the one of Proposition \ref{prop: infinitybounds}, and we will sketch here the main lines. By Duhamel's representation formula we have, for any $(t,x)\in[0,T)\times\T^d$, 
	\begin{align}
		\nabla u^\infty_t(x) - \nabla&  u^N_t(x)\nonumber\\
		&= \int_{\T^d}p_{T-t}(x-y)\Big [\nabla[\lf\cG](\mu^\infty_T,y) - \nabla[\lf\cG](\mu^N_T*D^N,\cdot)*D^N(y)\Big]\de y\label{eqn: est_inf_grad_1} \\
		&\quad- \int_t^T\int_{\T^d}\nabla p_{s-t}(x-y) \left[H(y,\nabla u^\infty_s(y)) - {H}(\cdot,\nabla  u^N_s)*D^N(y)\right]\de y \de s\label{eqn: est_inf_grad_2}\\
		&\quad + \int_t^T\int_{\T^d}\nabla p_{s-t}(x-y)\Big[\lf\cF(\mu^\infty_s,y) - \lf\cF(\mu^N_s*D^N,\cdot)*D^N(y)\Big]\de y \de s\label{eqn: est_inf_grad_3}.
	\end{align}
	A first term to control is
	\begin{align*}	
		\lvert\eqref{eqn: est_inf_grad_2}\rvert
		&\leq \int_t^T\int_{\T^d}\lvert \nabla p_{s-t}(x-y)\rvert\big\lvert H(y,\nabla u^\infty_s(y))-{H(\cdot,\nabla u^\infty_s)*D^N(y)}\big\rvert\de y\de s\\
		&\quad+\int_t^T\int_{\T^d}\lvert \nabla p_{s-t}(x-y)\rvert\big\lvert 
		 \bigl( H(\cdot,\nabla u^{\infty}_s)-H(\cdot,\nabla u^N_s)\bigr)*D^N(y)\big\rvert\de y\de s
		\\
		&\leq  \frac{C(d,T,\cF,\cG,H,q)}{N^{q-1-\frac{d}2}}
		+ C(d,T,\cF,\cG,H,q) \int_t^T  \frac{N^{\frac{d}2}}{\sqrt{s-t}}
	 \norm{\nabla u^\infty_s - \nabla u^N_s}_2 \de s,
	\end{align*}
	where we used Lemma
	\ref{lemma: estinftytruncation} to get the first line together with Lemma \ref{lemma: trivial_est_trunc} to get the second line. 
	
	Admit now for a while that (the proof is given a few lines below) 
\begin{equation}
\label{eq:thm:4.7:extra:bound}
\sup_{t \in [0,T]} 
\norm{\nabla u^\infty_t - \nabla u^N_t}_2
\leq
\frac{C(d,T,\cF,\cG,H,q,\addtxtr{\gamma})}{N^{q-1}},
\end{equation}		
	and then deduce that  
		\begin{align*}	
		\lvert\eqref{eqn: est_inf_grad_2}\rvert
		&\leq \frac{C(d,T,\cF,\cG,H,q,\addtxtr{\gamma})}{N^{q-1-\frac{d}{2}}}.	
		\end{align*}
	Regarding the term involving $\lf\cG$, we have
(recalling the notation $p_t(x)$ for the heat kernel at $(t,x)$)
\begin{align*}
		\lvert\eqref{eqn: est_inf_grad_1}\rvert
		&\leq \int_{\T^d} p_{T-t}(x-y) \big\lvert\nabla[\lf\cG](\mu^\infty_T,y) - \nabla[\lf\cG]({\mu^\infty_T*D^N},y)\big\rvert\de y\\
		&\quad+\int_{\T^d} p_{T-t}(x-y) \big\lvert\nabla[\lf\cG]({\mu^\infty_T*D^N},y) - \nabla[\lf\cG](\mu^N_T*D^N,y)\big\rvert\de y\\
		&\quad+\int_{\T^d} p_{T-t}(x-y) \big\lvert \nabla[\lf\cG](\mu^N_T*D^N,y) -\nabla[\lf\cG](\mu^N_T*D^N,\cdot)*D^N(y)\big\rvert\de y\\
		&\leq C(d,T,\cF,\cG,H,q,\addtxtr{\gamma})\left(\frac{\norm{m}_{2,q-1}}{N^{q-1}}+\frac{\norm{m}^2_{2,q-1}}{N^{q-1}} + \frac{1}{N^{{q-1-\frac{d}{2}}}}\right),
	\end{align*}
	where we used the fact that $\nabla[\lf\cG](\cdot,y)$ is Lipschitz (with respect to $L^2$, uniformly in $y$) together with 
	Proposition 
	\ref{prop: smooth_mu_inf} and Lemma \ref{lemma: estl2truncation} to control the first term
	and then invoked 
	\eqref{eqn: l2_est_final_mu} to control the second term.
	 The third term is handled by means of  
	 Lemma \ref{lemma: estl2truncation} again
	 together with 
	  the smoothness in space of $\nabla \lf\cG$.
	In a similar manner, we obtain
	\begin{equation*}
		\lvert\eqref{eqn: est_inf_grad_3}\rvert
		\leq C(d,T,\cF,\cG,H,q,\addtxtr{\gamma})\left(\frac{\norm{m}_{2,q-1}}{N^{q-1}}+\frac{\norm{m}^2_{2,q-1}}{N^{q-1}} + \frac{1}{N^{q-\frac{d}{2}}}\right).
	\end{equation*}
	By the Duhamel representation formula,  we obtain the desired estimate for $\norm{\nabla  u^N_t - \nabla u^\infty_t}_\infty$. To conclude, the estimate on $\norm{\diff_p H(\cdot,\nabla  u^N_t) - \diff_p H(\cdot,\nabla u^\infty_t)}_\infty$ follows again from the Lipschitz property of $\diff_p H$.
\medskip

In order to complete the proof of 
\eqref{eqn: linf_esst_final_u}, one must now establish 
\eqref{eq:thm:4.7:extra:bound}. In fact, using 
\eqref{eqn: l2_est_final_mu} as we have just  done  a few lines above (combined with the $L^2$-Lipschitz continuity of $\lf\cF$ and $\lf\cG$ in the measure argument), the bound can be obtained by following the proof of Lemma \ref{lemma: conv_l2_gradu_utilde}.

	\medskip

It now remains to prove \eqref{eqn: linf_esst_final_mu}. It follows from the same argument as the one used in the proof of Proposition \ref{prop: tildemu_prob_positive} combined with Theorem \ref{thm: conv_of_controls}. Indeed, we have
	\begin{align}
		\mu^\infty_t& (x) - \mu^N_t (x)\nonumber \\
		&= - 
		\int_0^t \int_{\T^d} \nabla p_{t-s}(x-y) \Big[\diff_p H(y,\nabla u^\infty_s(y)) - \diff_p H(y,\nabla u^N_s(y))\Big]\mu^\infty_s(y)\de y \de s\label{eqn: estdiffmu1_final}\\
		&\quad -\int_0^t \int_{\T^d} \nabla p_{t-s}(x-y)\diff_p H(y,\nabla u^N_s(y)) \Big[\mu^\infty_s(y) - \mu^\infty_s*D^N(y)\Big] \de y \de s\label{eqn: estdiffmu2_final}\\
		&\quad - \int_0^t \int_{\T^d} \nabla p_{t-s}(x-y)\diff_p H(y,\nabla u^N_s(y))  \Big[\mu^\infty_s*D^N(y) - \mu^N_s*D^N(y)\Big]\de y \de s\label{eqn: estdiffmu3_final}.
	\end{align}
	We can treat \eqref{eqn: estdiffmu2_final} 
	by means of Lemma 
	\ref{lemma: estinftytruncation}	
	and then handle \eqref{eqn: estdiffmu3_final} 
	by means of Lemma 
	\ref{lemma: trivial_est_trunc}
and 
\eqref{eqn: l2_est_final_mu}. Regarding \eqref{eqn: estdiffmu1_final}, 
	we can bound it by using \eqref{eqn: linf_esst_final_u}.
\end{proof}
\addtxtr{
	Theorem \ref{thm: unif_conv_control} is implicitly understood for $\mu_0^N=m$ but, as in the case of Theorem \ref{thm: conv_of_controls},  it remains valid even if \eqref{eqn: FWBKW_intro} starts from the truncated measure $\mu^N_0 = m*D^N$ instead of $m$. To prove  this, it is enough to use Proposition \ref{rmk: different_ic} in place of Theorem \ref{thm: conv_of_controls}, and to notice that $\norm{m - m*D^N}_2$ and $\norm{m - m*D^N}_\infty$ can be easily controlled with the desired rates thanks to Lemmas \ref{lemma: estl2truncation} and \ref{lemma: estinftytruncation}.}

\begin{remark}\label{rmk: conv_of_approx}
	By combining Proposition \ref{prop: tildemu_in_Hp}, Theorem \ref{thm: unif_conv_control} and Lemma \ref{lemma: estl2truncation}, we obtain
	\begin{align*}
		\sup_{t\in[0,T]}\norm{\mu^\infty_t - \mu^N_t*D^N}_2
		&\leq \sup_{t\in[0,T]}\norm{\mu^\infty_t  - \mu^\infty_t * D^N}_2 + \sup_{t\in[0,T]}\norm{\mu^\infty_t*D^N- \mu^N*D^N_t }_2\\
		&\leq \frac{C(d,T,\cF,\cG,H,q,\addtxtr{\gamma})}{N^{q-1}}\norm{m}^2_{2,q-1},
	\end{align*}
	while from Lemma \ref{lemma: estinftytruncation} we get
	\begin{equation*}
		\sup_{t\in[0,T]}\norm{\mu^\infty_t - \mu^N_t*D^N}_\infty\leq \frac{C(d,T,\cF,\cG,H,q,\addtxtr{\gamma})}{N^{q-1 - \frac{d}{2}}}\norm{m}^2_{2,q-1}.
	\end{equation*}
\end{remark}
\begin{remark}
We believe that 
	the $L^2$ estimate in Theorem \ref{thm: unif_conv_control} can be extended to the $H^{r}$ norm, for 
	${\lfloor q-2-d/2 \rfloor} > r>\frac{d}{2}$, $r\in\N$. More precisely, it should hold
	\begin{equation*}
		\sup_{t\in[0,T]} \Bigl( \norm{\mu^\infty_t - \mu^N_t}_{2,r}
		+\norm{u^\infty_t - u^N_t}_{2,r} \Bigr) 
		\leq \frac{C(d,T,\cF,\cG,H,q,\addtxtr{\gamma})}{N^{q-1-r}}\norm{m}^2_{2,q-1},
	\end{equation*}
	for any $\lfloor q-2-d/2 \rfloor > r>\frac{d}{2}$, {$r \in {\mathbb N}$}. This could be shown by following the techniques used to prove Propositions
	\ref{prop: infinitybounds_tilde}
	and
	 \ref{prop: smooth_muN}.
	\end{remark}

\subsection{Convergence of the value functions}\label{sec: conv_value}
In this section we investigate the convergence of the value functions, as a byproduct of the convergence of the optimal control proved in Subsections \ref{sec: monotonicity} and \ref{sec: unif_conv}. Let us recall the shape of the value functions, when the costs are evaluated at the optimal controls: 
\begin{equation}\label{eqn: values}
\begin{aligned}
	V^\infty(t,m)&:= \cG(\mu^\infty_T) + \int_t^T \left\{\cF(\mu^\infty_s) + \int_{\T^d}L(x,\alpha_s^{*,\infty}(x)) \mu^\infty_s(\!\de x)\right\}\de s,\\
	  V^N(t,m)&:= \cG(\mu^N_T*D^N) 
	+ \int_t^T \left\{\cF(\mu^N_s*D^N) + \int_{\T^d}L(x,\alpha_s^{*,N}(x))(\mu^N_s*D^N)(\!\de x)\right\}\de s,
\end{aligned}
\end{equation}
where $(\mu^\infty, u^\infty)$ and $(\mu^N, u^N)$ are the solutions to the forward-backward systems \eqref{eqn: fwdbkw_infty_intro} and \eqref{eqn: FWBKW_intro} respectively,  and $\alpha^{*,\infty}(\cdot) = -\diff_p H(\cdot,\nabla u^\infty(\cdot))$ and $\alpha^{*,N}(\cdot) = -\diff_p H(\cdot,\nabla u^N(\cdot))$ ({see Proposition \ref{prop: optimality}}). 
\addtxtr{Notice that, for any $m$ satisfying Assumption \ref{hp: 3} and $N$ large enough, it holds that $V^N(t,m) = V^N(t,m*D^N)$. This trivially follows from the definition of $V^N$, which involves only $\mu^N*D^N$.  Therefore, it is equivalent to consider  \eqref{eqn: FWBKW_intro} starting from $m$ or  the truncated measure  $m*D^N$.}

Before stating the result of this section, let us introduce a family of subspaces of $\prob(\T^d)$. For any $R>0$, we set
\begin{equation}\label{eqn: def_b_q}
	\cB^{q,\gamma}_R:=\{m\in\prob(\T^d)\text{ s.t. Assumption \ref{hp: 3} holds and } \norm{m}^2_{2,q-1}\leq R\}.
\end{equation}

\begin{proposition}\label{prop: conv_of_the_value}
Let Assumption \ref{hp: 1+2} hold. Then, there exists a positive constant $C = C(d,T,\cF,\cG,H,\gamma, q, R)$ such that
\begin{equation*}
	\sup_{m\in \cB^{q,\gamma}_R}\sup_{t\in[0,T)} \lvert V(t,m) -   V^N(t,m)\rvert\leq \frac{C(d,T,\cF,\cG,H,\gamma,q,R)}{N^{q-1}}.
\end{equation*}
\end{proposition}

\begin{proof}
	From \eqref{eqn: values} we have, for any $t\in[0,T]$ and $m\in\cB^{q,\gamma}_R$,
	\begin{align}
		\lvert V(t,m) -   V^N&(t,m) \rvert\nonumber\\
		&\leq \left\lvert   \cG(\mu^\infty_T) -  \cG(\mu^N_T*D^N)\right\vert \label{eqn: conv_value_1}\\
		&\quad+\left \lvert \int_t^T \left\{\cF(\mu^\infty_s) - \cF(\mu^N_s*D^N)\right\}\de s  \right \rvert\label{eqn: conv_value_2}\\
		&\quad+\left \lvert \int_t^T 	\left\{\int_{\T^d}L(x,\alpha_s^{*,\infty}(x)) \mu^\infty_s(\!\de x) - \int_{\T^d}L(x,\alpha_s^{*,N}(x)) (\mu^N_s*D^N)(\!\de x)\right\}\de s  \right \rvert.\label{eqn: conv_value_3}
	\end{align}
	We  notice from Assumption \ref{hp: 1+2}  that $\cF$ and $\cG$ are Lipschitz with respect to the $L^2$ norm. Indeed,  {because their derivatives are bounded, both}  $ \lf\cF$ and $ \lf\cG$ are Lipschitz  continuous  with respect to the $1$-Wasserstein distance and thus also with respect to $L^2$ (see Remark \ref{rem: tv_vs_sobolev}).
	 Then, we have
	\begin{equation*}
		\eqref{eqn: conv_value_1}\leq C(\cG)\norm{\mu^\infty_T - \mu^N_T*D^N}_2\leq \frac{C(d,T,\cF,\cG,H,q,\addtxtr{\gamma})}{N^{q-1}}\norm{m}^2_{2,q-1},
	\end{equation*}
	where the last inequality follows directly from Remark \ref{rmk: conv_of_approx}. Similarly, we get
	\begin{equation*}
		\eqref{eqn: conv_value_2}\leq \frac{C(d,T,\cF,\cG,H,q,\addtxtr{\gamma})}{N^{q-1}}\norm{m}^2_{2,q-1}.
	\end{equation*}
	For the last term, we are going to exploit the boundedness of the optimal controls obtained in  \S\ref{sec: aux_problem}, together with the convergence results obtained in 
	 {Subsections}  \ref{sec: monotonicity} and \ref{sec: unif_conv}. It holds
	\begin{align*}
		\eqref{eqn: conv_value_3}&\leq\bigg\lvert\int_0^T\int_{\T^d} \left(L(x,\alpha_s^{*,\infty}(x)) - L(x,\alpha_s^{*,N}(x))\right)\mu^\infty_s(\!\de x) \de s\bigg\rvert\\
		&\quad+\bigg\lvert\int_0^T\int_{\T^d}L(x,\alpha_s^{*,N}(x)) \left(\mu^\infty_s - \mu^N_s*D^N\right)(\!\de x)\de s\bigg\rvert\\
		&\leq C(d,T,\cF,\cG,H,q) 
		\left(\int_0^T\norm{\nabla u^\infty_s - \nabla  u^N_s}_2^2\de s\right)^\frac{1}{2}
		\\
		&\quad+ C(d,T,\cF,\cG,H,q) 
		\int_0^T\norm{\mu^\infty_s - \mu^N_s*D^N}_2\de s
		\leq \frac{C(d,T,\cF,\cG,H,q,\addtxtr{\gamma})}{N^{q-1}}\norm{m}^2_{2,q-1}.
	\end{align*}
	In the inequalities above, we first used the fact that $L(\cdot, \alpha^{*,\infty}) = \diff_p H(\cdot,\nabla u^\infty)\cdot\nabla u^\infty - H(\cdot, \nabla u^\infty)$ (and the analogue formula with $\alpha^{*,N}$ and $\nabla  u^N$) in order to control the difference of the two Lagrangians. Then, we can use the estimates on $H$ and $\diff_p H$ following from Assumption \ref{hp: 1+2} - (H.5)  together with  Remark \ref{rmk: smoothness_H_of_grad}. Finally, to deal with the first summand, we used Cauchy-Schwarz inequality together with  Theorem \ref{thm: conv_of_controls} and the boundedness of $\mu^\infty$, whilst for the second one we used Remark \ref{rmk: conv_of_approx}. 
\end{proof}
\appendix
\section{}
\label{sec: prelimiraries}
We present here a collection of classical results on Fourier series and technical remarks that we have used all along our discussion. Let us recall that $i^2 = -1$ and $\lvert k\rvert = \max_{j=1,\dots,d}\lvert k_{j}\rvert$, for any $k\in\Z^d$. First, let us state a well known fact about Fourier coefficients (see, for instance, \cite[Theorem 3.3.9]{grafakos}).
\begin{proposition}\label{prop: est_on_fourier_coeff}
	Let $q\in\N\setminus\{0\}$. For a function $\varphi\colon\T^d\to\R$, assume that, for all multi-indices $\lvert a\rvert\leq q$, $\partial^a\varphi$ exists and is integrable . Then
	\begin{equation*}
		\lvert \widehat\varphi(k)\rvert\leq \left(\frac{\sqrt{d}}{2\pi}\right)^q \frac{\max_{\lvert a\rvert = q}\lvert\reallywidehat{\partial^a\varphi}\rvert}{\lvert k\rvert^q},\quad k\in \Z^d,k\neq0.
	\end{equation*}
\end{proposition}
\begin{remark}
	From Proposition \ref{prop: est_on_fourier_coeff}, it follows that 
	if  $\varphi\in H^q(\T^d)$,
	\begin{equation*}
		\lvert \widehat\varphi(k)\rvert\leq \frac{C(d,q)}{\lvert k\rvert^q}\norm{\varphi}_{2,q},\quad k\in \Z^d,k\neq0. 
	\end{equation*}
\end{remark}
We present now a series of results concerning the relation between $L^\infty$ and $H^q$ norms for truncated Fourier series (or for the remainders). They play a key role in almost all our results, especially the ones in Section \ref{sec: convergence}. 
\begin{lemma}\label{lemma: trivial_est_trunc}
	Let us consider a \addtxtr{function} $\mu$ in $H^q(\T^d)$, $q\geq 0$. Then, for a positive constant $C = C(d)$ it holds that 
	\begin{equation*}
		\norm{\mu*D^N}_{2,q}\leq\norm{\mu}_{2,q},\quad
		\norm{\mu*D^N}_{\infty}\leq C(d) N^\frac{d}{2}\norm{\mu}_{2}.
	\end{equation*}
	Moreover, if $q>\frac{d}{2}$, there exists a positive constant $C = C(d,q)$ such that
	\begin{equation*}
		\norm{\mu*D^N}_{\infty}\leq C(d,q)\norm{\mu}_{2,q}.
	\end{equation*}
\end{lemma}
\begin{proof}
By definition of Sobolev norm and $D^N$, it follows
	\begin{equation*}
		\norm{\mu*D^N}^2_{2,q} = \sum_{\lvert k\rvert\leq N }(1 + \lvert k\rvert^2)^q\lvert\hat{\mu}(k)\rvert^2 
		\leq \sum_{k\in\Z^d }(1 + \lvert k\rvert^2)^q\lvert\hat{\mu}(k)\rvert^2 = \norm{\mu}^2_{2,q}.
	\end{equation*}
Moreover, by Cauchy-Schwarz inequality and Parseval's theorem
\begin{equation*}
	\lvert \mu *D^N(x)\rvert \leq  \sum_{\lvert k\rvert\leq N} \lvert \hat{\mu}(k)\rvert\leq C(d)N^{\frac{d}{2}}\left(\sum_{\lvert k\rvert\leq N} \lvert \hat{\mu}(k)\rvert^2\right)^{\frac{1}{2}} \addtxtr{\leq} C(d)N^{\frac{d}{2}} \norm{\mu}_2.
\end{equation*}
	Finally, if $q>\frac{d}{2}$, by Cauchy-Schwarz inequality it follows
	\begin{equation*}
		\lvert \mu *D^N(x)\rvert \leq  \sum_{\lvert k\rvert\leq N} \lvert \hat{\mu}(k)\rvert (1 + \lvert k \rvert^2)^{\frac{q}{2}}(1 + \lvert k \rvert^2)^{-\frac{q}{2}}\leq \left ( \sum_{k\in\Z^d} \frac{1}{(1 + \lvert k \rvert^2)^q} \right)^{\frac{1}{2}} \!\! \norm{\mu}_{2,q}  = C(d,q) \norm{\mu}_{2,q} .
	\end{equation*}
\end{proof}
\begin{remark}\label{rem: tv_vs_sobolev}
	If $\mu,\nu$ have densities in $L^2(\T^d)$, it holds
	\begin{equation*}
		\norm{\mu - \nu}_{\rm TV} = \sup_{\norm{h}_\infty\leq 1}\left\lvert \int_{\T^d} h(x)(\mu(x) - \nu(x))\de x \right\rvert\leq C(d)\norm{\mu-\nu}_2,
	\end{equation*}
	since $L^\infty(\T^d)\subset L^2(\T^d)$ and $\norm{\mu-\nu}_2=\sup_{h\in L^2}\lvert\scalprod{\mu-\nu}{h}\vert$. Similarly, if $\mu,\nu\in H^q(\T^d)$, the inclusion 
	$L^\infty(\T^d)\subset H^{-q}(\T^d)$ entails $\norm{\mu - \nu}_{\rm TV}\leq C(d)\norm{\mu -\nu}_{2,q}$.
\end{remark}
\begin{lemma}\label{lemma: estl2truncation}
	Let us consider a \addtxtr{function} $\mu$ in $H^q(\T^d)$, $q\geq 0$. Then, it holds that 
	\begin{equation*}
		\norm{\mu - \mu*D^N}_2\leq\frac{\norm{\mu}_{2,q}}{N^{q}}. 
	\end{equation*}
\end{lemma}
\begin{proof}
	The case $q=0$ is immediate. For $q>0$, it holds
	\begin{align*}
		\norm{\mu - \mu*D^N}^2_2 = \sum_{\lvert k \rvert >N}\lvert \widehat{\mu} (k) \rvert^2
		& =  \sum_{\lvert k \rvert >N}\lvert \widehat{\mu} (k) \rvert^2 (1 + \lvert k \rvert^2)^{q}(1 + \lvert k \rvert^2)^{-q}\\
		&\leq \frac{1}{N^{2q}}\sum_{k\in\Z^d}\lvert \widehat{\mu} (k) \rvert^2 (1 + \lvert k \rvert^2)^{q} = \frac{\norm{\mu}_{2,q}^2}{N^{2q}}.
	\end{align*}
\end{proof}

\begin{lemma}\label{lemma: estinftytruncation}
	Let $\varphi\in H^q(\T^d)$ for a certain $q>\frac{d}{2}$. Then, there exists a positive constant $C = C(d,q)$ such that
	\begin{equation*}
		\norm{\varphi -\varphi*D^N}_\infty\leq\frac{C(d,q)}{N^{q - \frac{d}{2}}}\norm{\varphi}_{2,q}.
	\end{equation*}
\end{lemma}
\begin{proof}
	By Cauchy-Schwarz inequality we have
	\begin{align*}
		\norm{\varphi -\varphi*D^N}_\infty&\leq \sum_{\lvert k \rvert> N} \lvert \widehat{\varphi}(k)\rvert = \sum_{\lvert k \rvert> N} \lvert \widehat{\varphi}(k)\rvert(1 + \lvert k\rvert^2)^{\frac{q}{2}}(1 + \lvert k\rvert^2)^{-\frac{q}{2}}\\
		&\leq  \Bigg(\sum_{\lvert k \rvert> N} \frac{1}{(1 + \lvert k\rvert^2)^q}\Bigg)^{\frac{1}{2}} \Bigg(\sum_{\lvert k \rvert> N} (1 + \lvert k\rvert^2)^q\lvert \widehat{\varphi}(k)\rvert^2\Bigg)^{\frac{1}{2}} \leq \frac{C(d,q) }{N^{q-\frac{d}{2}}}\norm{\varphi}_{2,q},
	\end{align*}
	where we used that fact that $\sum_{\lvert k \rvert> N} \frac{1}{(1 + \lvert k\rvert^2)^q}\leq \frac{C(d,q)}{N^{2q-d}}$ since $q>\frac{d}{2}$.
\end{proof}

\begin{remark}\label{rmk: infinity_est_trunc}
	From Lemma \ref{lemma: estinftytruncation} we can deduce some estimates for the truncated Fourier series. Indeed, if $\varphi\in H^q(\T^d)$, with $q>\frac{d}{2}$, then
	\begin{equation}
		\norm{\varphi *D^N}_\infty\leq \norm{\varphi}_\infty +\frac{C(d,q)}{N^{q - \frac{d}{2}}}\norm{\varphi}_{2,q}.
	\end{equation}
	Of course, if $\varphi$ belongs also to $\rmC^q(\T^d)$, we can exchange the $H^q$ norm with the $\rmC^q$ norm.
\end{remark}

For our discussion in Subsection \ref{sec: aux_problem}, we also need some standard estimates for the heat kernel on the torus, that we collect here for simplicity.
\begin{lemma}\label{lemma: est_heat_kernel}
	For $t\in[0,T]$, let us denote by $p = p_t(x)$, the usual heat kernel over $\T^d$.  
	Then, for every $t \in [0,T]$ it holds 
	\begin{equation*}
		\norm{p_t}_2 \leq C(d,T)t^{-\frac{d}{4}},\quad \norm{\nabla p_t}_2\leq C(d,T)t^{-(\frac{d}{4} + \frac{1}{2})}.
	\end{equation*}
\end{lemma} 
\bibliographystyle{plain}
\bibliography{biblio}
\end{document}